\documentclass{elsarticle}

\usepackage{lineno,times}
\modulolinenumbers[5]

\usepackage{hyperref}
\usepackage{amsmath, amsthm, amscd, amsfonts, amssymb, graphicx, color}
\newtheorem{theorem}{Theorem}[section]
\newtheorem{main}{Main Theorem}
\newtheorem{lemma}[theorem]{Lemma}
\newtheorem{assumption}{Assumption}

\theoremstyle{definition}
\newtheorem{definition}[theorem]{Definition}

\newtheorem{remark}[theorem]{Remark}
\numberwithin{equation}{section}
\usepackage{latexsym}
\usepackage{amsmath}
\usepackage{amssymb}
\allowdisplaybreaks[4]

 \usepackage[text={135mm,210mm},centering]{geometry}

\newcommand{\be}{\begin{equation}\begin{aligned}}
\newcommand{\ee}{\end{aligned}\end{equation}}
\newcommand{\ben}{\begin{equation}\nonumber\begin{aligned}}
\newcommand{\dist}{{\rm dist}}
\newcommand{\R}{\mathbb{R}}
\newcommand{\N}{\mathbb{N}}
\newcommand{\D}{\mathcal{D}}
\newcommand{\A}{\mathcal{A}}
\newcommand{\B}{\mathfrak{B}}

\newcommand{\TT}{\mathbb{T}}
\newcommand{\LL}{\mathbb{L}_{per}}

\renewcommand{\d}{{\rm{d}}}

    \usepackage{mathabx}

\begin{document}

\begin{frontmatter}

\title{   Regularity of random attractor and fractal dimension  of  fractional stochastic Navier-Stokes equations  on three-dimensional torus}

\author[lh]{Hui Liu}

\ead{ss\_liuhui@ujn.edu.cn}

\author[scf]{Chengfeng Sun\corref{mycorrespondingauthor}}
\cortext[mycorrespondingauthor]{Corresponding author}
\ead{sch200130@163.com}

\author[xj]{Jie Xin}

\ead{fdxinjie@sina.com}

\address[lh]{School of Mathematical Sciences, University of Jinan, Jinan 250022, PR China}

\address[scf]{School of Applied Mathematics, Nanjing University of Finance and Economics, Nanjing 210023, PR China}
\address[xj]{School of Information Engineering, Shandong Youth University of Political Science, Jinan 250103, PR China}

\begin{abstract}

 In this paper we will  study the asymptotic dynamics of  fractional   Navier-Stokes (NS) equations  with  additive white noise  on three-dimensional torus $\mathbb T^3$.  Under the conditions that the external forces $f(x)$ belong to the phase space $ H$ and the noise intensity function $h(x)$  satisfies $\|\nabla h\|_{L^\infty} < \sqrt \pi  \nu \lambda_1^\frac{5}{4}$,
 where  $ \nu $ is the kinematic viscosity of the fluid and $\lambda_1$ is  the first eigenvalue of the Stokes operator,  we shown that the random fractional three-dimensional NS equations possess a tempered $(H,H^\frac{5}{2})$-random attractor  whose fractal dimension in $H^\frac{5}{2}$ is finite. This was proved by  establishing, first, an $H^\frac{5}{2}$ bounded   absorbing set and, second, a local $(H,H^\frac{5}{2})$-Lipschitz continuity  in initial values  from which    the $(H,H^\frac{5}{2})$-asymptotic compactness  of the system follows. Since the forces $f$ belong only to $H$, the $H^\frac{5}{2}$ bounded absorbing set was constructed by an indirect approach of estimating the $H^\frac{5}{2}$-distance     between the solutions  of the random fractional three-dimensional NS equations and that of the corresponding deterministic  equations.  Furthermore, under the conditions that the external forces $f(x)$ belong to the $ H^{k-\frac{5}{4}}$ and the noise intensity function $h(x)$  belong to $H^{k+\frac{5}{4}}$ for $k\geq\frac{5}{2}$, we shown that the random fractional three-dimensional NS equations possess a tempered $(H,H^k)$-random attractor  whose fractal dimension in $H^k$ is finite. This was proved by using iterative methods and establishing, first, an $H^k$ bounded   absorbing set and, second, a local $(H,H^k)$-Lipschitz continuity  in initial values  from which    the $(H,H^k)$-asymptotic compactness  of the system follows.
  \end{abstract}

\begin{keyword}
\texttt{Navier-Stokes equations;
   $H^k$ regularity; $H^k$ absorbing set; Local Lipschitz continuity; Fractal dimension.}
\MSC[2020] 35B40, 35B41, 60H15
\end{keyword}
\end{frontmatter}

\tableofcontents

\section{Introduction}

In this  paper we introduce the following fractional stochastic three-dimensional Navier-Stokes (NS) equations on  $\mathbb{T}^3=[0,L]^3$, $L>0$:
 \begin{align} \label{1}
 \left \{
  \begin{aligned}
 &  \d u_h+ \left ( \nu (-\Delta)^\alpha u_h+(u_h\cdot\nabla)u_h+\nabla p\right) \d t=  f(x) \d t+h(x)\d \omega,\\
  & \nabla\cdot u_h=0,\\
  & u_h(0)=u_{h,0},
  \end{aligned} \right.
\end{align}
 endowed with periodic boundary conditions, where $u_h$ denotes the fluid velocity, $ \nu >0$ denotes the kinematic viscosity, $p$ denotes the scalar pressure and $f:=f(x) \in H $  is  the fluid forces, where $H= \left \{u\in \LL ^2:
  \int_{\mathbb T^3}u  \d x=0,\,  {\rm div} u=0 \right \}$ is the phase space of the system. In Section \ref{sec7}, we set the forces $f\in H^{k-\frac{5}{4}}$ for $k\geq\frac{5}{2}$. $\omega(t)$ is a two-sided real valued Wiener process on a probability space $(\Omega,\mathcal{F},\mathbb{P})$ with $\Omega=\{\omega\in C(\mathbb{R},\mathbb{R}):\omega(0)=0\}$, $\mathcal{F}$ is the Borel $\sigma$-algebra induced by the compact open topology of $\Omega$, and $\mathbb{P}$ is the corresponding Wiener measure on $(\Omega,\mathcal{F})$. The subscript $``_{h}"$ represents the dependence on the noise intensity $h(x)$.
When $h\equiv0$, the equations \eqref{1} reduce to the fractional three-dimensional Navier-Stokes equations  on  $\mathbb{T}^3$:
 \begin{align}\label{2}
 \left \{
  \begin{aligned}&
   \frac{\d u}{\d t}+\nu (-\Delta)^\alpha u+(u\cdot\nabla)u+\nabla p=f(x),\\
  & \nabla\cdot u=0,\\
   &u(0)=u_{0}.
  \end{aligned} \right.
\end{align}
Global well-posedness of solutions of the fractional three-dimensional NS equations has been investigated by many authors for $\alpha\geq\frac54$, see, e.g., \cite{wu2003,wu2008,zhou2007}, here $\alpha=\frac54$ is critical value. In this paper, for the sake of simplicity, we only consider the critical value $\alpha=\frac54$. For $\alpha>\frac{5}{4}$, by using the similar method, we get the following corresponding to the main results.

The  NS equations are important mathematical models in \cite{foias,temam}. The global attractors of the autonomous deterministic two-dimensional NS equations have been studied, see, e.g., \cite{constantin,ju,liuvx,liuvx1,rosa,sun,wang,zhao};
while \cite{cui24siads,cui24ma,julia,hou,langa,lu,lukaszewicz} proved the asymptotic dynamics of the non-autonomous deterministic two-dimensional NS equations. Meanwhile, the asymptotic dynamics of the stochastic two-dimensional  NS equations have been studied, see, e.g., \cite{Brzezniak13,Brzezniak15,Brzezniak2015,Brzezniak2018,Brzezniak,crauel97jdde,crauel,liJDE,li,mohammed,wang2023}. Well-posedness and asymptotic dynamics of the stochastic three-dimensional modified Navier-Stokes equations have been studied, see, e.g., \cite{Caraballo,gaoJDE,HanJDE,Kinra,liu2018,Rockner}.

The fractional fluid equations are also important mathematical models in \cite{foias,temam}. The asymptotic dynamics of the  fractional three-dimensional NS equations have been studied, see, e.g., \cite{gal,kostianko,li2020,wangsiam}. Meanwhile, \cite{Anh,Tinh} proved the asymptotic dynamics of the other fractional  fluid equations. Well-posdeness and dynamics of the stochastic fractional fluid equations have been widely studied, see, e.g. \cite{ali,debbi,gess,li2021,li2023,liusd,rockner,zhang,zhu}.

In \cite{liu}, if $f\in H$, the global attractor $\mathcal{A}_0$ of the deterministic fractional three-dimensional NS equations is bounded in $H^{\frac{5}{2}}$. When $f\in H^{k-\frac{5}{2}}$ for $k>\frac{5}{2}$, the absorbing ball of the deterministic fractional three-dimensional NS equations is bounded in $H^{k}$, see , e.g. \cite{liu2023}. Inspired by the Proposition 12.4 in \cite{robinson01} and \cite{liu2023,liu}, we  introduce the following main results.
\begin{lemma}\label{lemma1.1}
Let $f\in H$, then the deterministic fractional three-dimensional NS equations \eqref{2} generate a semigroup $S$ which possesses a global attractor $\mathcal{A}_0$ in $H$ and has an absorbing set which is bounded in $H^{\frac52}$. And then, the global attractor $\mathcal{A}_0$ is bounded in $H^{\frac52}$.
\end{lemma}
By the Lemma \ref{lemma1.1} and $f\in H$, we have the global attractor $\mathcal{A}_0$ of the deterministic fractional three-dimensional NS equations \eqref{2} is a global $(H,H^\frac54)$-attractor. In other words, it is a compact set in $H^{\frac54}$ and attracts a bounded set in $H$ under the norm of $H^{\frac54}$, see , e.g. \cite{liu2023,liu,robinson01}. Moreover, it means that any weak solutions of the deterministic fractional three-dimensional NS equations \eqref{2} are attracted by these strong solutions in $H^{\frac54}$. By the $H^{\frac{5}{2}}$-boundedness of $\mathcal{A}_0$ implies that the strong solutions in  $\mathcal{A}_0$ are also smooth solutions, see , e.g. \cite{Ladyzenskaja,liu,robinson01,temam1,temam}. But the global attractor $\mathcal{A}_0$ is not shown to be $(H,H^\frac52)$.

\begin{lemma}\label{lemma1.2}
Let $f\in H^{k-\frac{5}{2}}$ for $k>\frac{5}{2}$, then the deterministic fractional three-dimensional NS equations \eqref{2} generate a semigroup $S$ which possesses a global attractor $\mathcal{A}_0$ in $H$ and has an absorbing set which is bounded in $H^{k}$. And then, the global attractor $\mathcal{A}_0$ is bounded in $H^{k}$.
\end{lemma}
By the Lemma \ref{lemma1.2} and $f\in H^{k-\frac{5}{2}}$ for $k>\frac{5}{2}$,  it is a compact set in $H^{k-\frac{5}{4}}$ and attracts a bounded set in $H$ under the norm of $H^{k-\frac{5}{4}}$, see, e.g. \cite{liu2023,robinson01}.  By the $H^{k}$-boundedness of $\mathcal{A}_0$ implies that the strong solutions in  $\mathcal{A}_0$ are also smooth solutions, see, e.g. \cite{Ladyzenskaja,liu,robinson01,temam1,temam}. But the global attractor $\mathcal{A}_0$ is not shown to be $(H,H^k)$.

Recently, the $(H,H^1)$-attractor of the deterministic two-dimensional NS equations has been extensively studied. Applying the asymptotic compactness and the enstrophy equation, the $(H,H^1)$-attractor of the deterministic two-dimensional NS equations was proved on some unbounded domains in \cite{ju}. And then, \cite{cui21jde} proved the smoothing and finite fractal dimensionality of $(H,H^1)$-uniform attractors.  \cite{zhao19} proved $(H,H^1)$-trajectory attractors of the deterministic three-dimensional modified NS equations and  \cite{song,song1} also proved $(H^1,H^2)$-global attractor and $(H^1,H^2)$-uniform attractor of the deterministic three-dimensional NS equations with damping. For the stochastic fluid equations, the $(H,H^1)$-random attractors were studied by many authors, see, e.g. \cite{li,wang2024,wang2023}. Moreover, \cite{cui2025} proved the finite fractal dimensionality of $(H,H^2)$-random attractors of the random two-dimensional NS equations.

In this paper, we consider the random fractional three-dimensional NS equations. When $f\in H$, the global attractor $\mathcal{A}_0$ of the deterministic fractional three-dimensional NS equations \eqref{2} is an $(H,H^\frac52)$-global attractor. We overcome the main difficult lies in the $H^\frac54$ boundedness of the time derivatives $u_t$ of solutions of \eqref{2}, see, e.g. \cite{liu,song}. In this paper, we will show that under certain conditions, the random dynamical system generated by \eqref{1} has a $(H,H^\frac52)$-random attractor and it has a finite fractal dimensional in $H^\frac52$. Since the estimating of $u_t$ can not apply to stochastic equations, in order to prove the $H^\frac{5}{2}$ bounded absorbing set, we will construct by an indirect approach of estimating the $H^\frac{5}{2}$-distance   between the solutions  of the random fractional three-dimensional NS equations \eqref{1} and that of the corresponding deterministic  equations \eqref{2}. Here, \cite{wang2024} obtained an $(H,H^1)$-random attractor of the random hydrodynamical equations on  two-dimensional torus, but we study the $(H,H^\frac52)$-random attractor of the random fractional  NS equations on  three-dimensional torus with the same condition $f\in H$. Moreover, when $f\in H^{k-\frac54}$ for $k\geq\frac{5}{2}$, we also study the $(H,H^k)$-random attractor of the random fractional three-dimensional  NS equations. 

When $f\in H^{k-\frac{5}{2}}$ for $k\geq\frac{5}{2}$, the global attractor $\mathcal{A}_0$ of the deterministic fractional three-dimensional NS equations \eqref{2} is an $(H,H^k)$-global attractor, see, e.g. \cite{liu2023,song}. In this paper, we will show that under certain conditions, the random dynamical system generated by \eqref{1} has a $(H,H^k)$-random attractor and it has a finite fractal dimensional in $H^k$.

\subsection*{Main results. Main techniques}

 In this subsection, we now introduce the main results and main techniques.
Firstly, following Crauel  \& Flandoli \cite{crauel}  we reformulated the stochastic fractional three-dimensional NS   equations \eqref{1}   via a scalar Ornstein-Uhlenbeck process as the following random equations
\begin{align}  \label{1.3}
\frac{\d v}{\d t}+\nu A^{\frac54}v+B\big( v(t)+hz(\theta_t\omega)\big)=f(x)-\nu A^{\frac54}hz(\theta_t\omega)+hz(\theta_t\omega),
\end{align}
where $A =-P\Delta$, $B(u)=P((u\cdot \nabla)u)$ with $P:\LL^2\to H$ the Helmholtz-Leray orthogonal projection, and   $z$ is a tempered random variable.  The random equations are known as a conjugate system to \eqref{1} and generate a random dynamical system (RDS) $\phi$ in $H$, see Section \ref{sec2.1} for more detail.
The intensity  function of the noise $h(x)$ is supposed to satisfy  the following assumptions.
\begin{assumption} \label{assum}
Let    $ h \in H^\frac{15}{4}$, and
\[
  \|\nabla h\|_{L^\infty} <  {\sqrt \pi} \nu \lambda_1^\frac{5}{4} ,
\]
where  $\lambda_1 = 4\pi^2 /L^2$ is the first eigenvalue of the Stokes operator $A$.
  \end{assumption}
\begin{assumption} \label{assum1}
For $k\geq\frac{5}{2}$, we set $h\in H^{k+\frac54}$ and $f\in H^{k-\frac54}$.
  \end{assumption}
In Section \ref{sec4}, we can get an $H^{\frac{5}{2}}$-random absorbing set.
\begin{main}[$H^{\frac{5}{2}}$ random absorbing set]
Let Assumption \ref{assum} hold and $f\in H$. The  RDS $\phi$ has a  random absorbing set $\mathfrak{B}_{H^\frac52}$, defined by  a random $H^\frac52$ neighborhood of the global attractor $\mathcal{A}_0$ of the deterministic fractional three-dimensional  NS equations \eqref{2}:
\begin{align}
\mathfrak{B}_{H^\frac52}(\omega)=\left \{v\in H^\frac52: \, {\rm dist}_{H^\frac52}(v,\mathcal{A}_{0}) \leq \sqrt{\rho(\omega)} \, \right\}, \quad \omega\in\Omega, \nonumber
\end{align}
where $\rho(\cdot)$ is the tempered random variable defined by \eqref{rho}.  Moreover, the $\mathcal{D}_H$-random attractor $\mathcal{A}$ of the system \eqref{2.2} is a bounded and tempered random set in $H^\frac52$.
\end{main}

We extend the result of the Lemma \ref{lemma1.1} to random cases. Since the standard techniques of the deterministic fractional three-dimensional NS equations in proving the $H^{\frac52}$ regularity can not apply to the random equations \eqref{1.3}, we overcome the main difficulty to prove the main result as in Section \ref{sec4}. Since the Winner process is not derivable in time, we can not get the estimation the time-derivative $u_t$ of solutions of \eqref{1.3}. The method of the Lemma \ref{lemma1.1} can not directly apply the random equations \eqref{1.3}. In order to get the estimation of $H^\frac52$-norm of \eqref{1.3}, we multiply the random equations \eqref{1.3} by $A^\frac52v$ and need $f\in H^\frac54$, but we only require $f\in H$. By introducing a comparison approach, we get estimation of the difference $w:=v-u$, here $v$ is the solution of random equations \eqref{1.3} and $u$ is the solution of deterministic  equations \eqref{2} corresponding to an initial value lying in the global attractor $\mathcal{A}_0$, respectively. By virtue of the Lemma \ref{lemma4.4}, we shown that $w$ is asymptotically bounded in $H^\frac52$.

In Section \ref{sec6}, we can get the main result as the following.
\begin{main}[$(H,H^\frac{5}{2})$-random attractor]
Let Assumption \ref{assum} hold and $f\in H$. The  random attractor $\mathcal A $ of the random fractional three-dimensional NS equations  \eqref{1.3} has  a an $(H,H^\frac52)$-random attractor  of finite fractal dimension in $H^\frac52$.
\end{main}
According to the standard bi-spatial random attractor in \cite{cui18jdde}, in order to get an $(H,H^\frac{5}{2})$-random attractor, we only show the $(H,H^\frac{5}{2})$-asymptotic compactness of the random equations \eqref{1.3}. By proving the local $(H, H)$-Lipschtiz continuity, local $(H, H^\frac54)$-Lipschtiz continuity and local $(H^\frac54, H^\frac52)$-Lipschtiz continuity, we prove the local $(H, H^\frac52)$-Lipschtiz continuity of solutions as the following, see the Theorem \ref{theorem5.6}.
\begin{main}[Local $(H, H^\frac52)$-Lipschtiz continuity]
 Let Assumption \ref{assum} hold and $f\in H$. For   any tempered set $\mathfrak D \in \D_H$,  then there
  exist  random variables $T_{{\mathfrak D}} (\cdot)  $   and $ L_{\mathfrak D}(\cdot )$ such that   two solutions $v_1$ and $v_2$ of random fractional three-dimensional NS equations \eqref{1.3} corresponding to initial values   $v_{1,0},$ $ v_{2,0}$ in $\mathfrak D \left (\theta_{-T_{\mathfrak D}(\omega)}\omega \right)$, respectively, satisfy
  \begin{align*}
  &
  \big\|A^\frac54(v_1 \!  \big (T_{\mathfrak D}(\omega),\theta_{-T_{\mathfrak D}(\omega)}\omega,v_{1,0}\big)
  -v_2 \big (T_{\mathfrak D}(\omega),\theta_{-T_{\mathfrak D}(\omega)}\omega,v_{2,0} \big) \big) \|^2 \\[0.8ex]
&\quad \leq  L_{\mathfrak D} (\omega)\|v_{1,0}-v_{2,0}\|^2,\quad \omega\in \Omega.
\end{align*}
\end{main}

In Section \ref{sec7}, we also can get the main result as the following.
\begin{main}[$(H,H^k)$-random attractor]
Let $k\geq\frac{5}{2}$. Let Assumption \ref{assum1} hold and $f\in H^{k-\frac{5}{4}}$.  The  random attractor $\mathcal A $ of the random fractional three-dimensional NS equations  \eqref{1.3} has  a an $(H,H^k)$-random attractor  of finite fractal dimension in $H^k$.
\end{main}
Let $k\geq\frac{5}{2}$. In order to get an $(H,H^k)$-random attractor, we only show the $(H,H^k)$-asymptotic compactness of the random equations \eqref{1.3}. By proving the local $(H^{k-\frac{5}{4}}, H^k)$-Lipschtiz continuity, we also prove the local $(H, H^k)$-Lipschtiz continuity of solutions as the following, see the Theorem \ref{theorem7.4}.

\begin{main}[Local $(H, H^k)$-Lipschtiz continuity]
 Let Assumption \ref{assum1} hold and $f\in H^{k-\frac{5}{4}}$ for $k\geq\frac{5}{2}$. For   any tempered set $\mathfrak D \in \D_H$,  then there
  exist  random variables $T^*_{{\mathfrak D}} (\cdot)  $   and $ L^*_{\mathfrak D}(\cdot )$ such that   two solutions $v_1$ and $v_2$ of random fractional three-dimensional NS equations \eqref{1.3} corresponding to initial values   $v_{1,0},$ $ v_{2,0}$ in $\mathfrak D \left (\theta_{-T^*_{\mathfrak D}(\omega)}\omega \right)$, respectively, satisfy
  \begin{align*}
  &
  \big\|A^\frac k2(v_1 \!  \big (T^*_{\mathfrak D}(\omega),\theta_{-T^*_{\mathfrak D}(\omega)}\omega,v_{1,0}\big)
  -v_2 \big (T^*_{\mathfrak D}(\omega),\theta_{-T^*_{\mathfrak D}(\omega)}\omega,v_{2,0} \big) \big) \|^2 \\[0.8ex]
& \leq  L^*_{\mathfrak D} (\omega)\|v_{1,0}-v_{2,0}\|^2,\quad \omega\in \Omega.
\end{align*}
\end{main}

\subsection*{Organization of the paper}

 In Section \ref{sec2}, we introduce the basic concepts and $(X,Y)$-random attractor theory and their fractal dimensions.  In Section \ref{sec3}, we first  introduce the  functional spaces and basic result of the deterministic fractional three-dimensional NS equations, and then associate an RDS to the random fractional three-dimensional NS equations.    We construct  $H^\frac54$ random absorbing sets and the  $(H,H^\frac54)$-attractor of  the  solutions and has a finite fractal dimensional.     In Section \ref{sec4}, by an indirect approach of  estimating the $H^\frac52$-distance between the random and the deterministic solution trajectories  we construct an  $H^\frac52$ random absorbing set. In Section \ref{sec5}, we  prove the local  $(H,H)$-Lipschitz continuity, $(H,H^\frac54)$-Lipschitz continuity and $(H,H^\frac52)$-Lipschitz continuity  of the RDS. In Section \ref{sec6}, we show that the random attractor    in $H$ is in fact an $(H,H^\frac52)$-random attractor of finite  fractal dimension in $H^\frac52$. In Section \ref{sec7}, let $f\in H^{k-\frac54}$, by using iterative methods, we first prove $H^k$ random absorbing sets and $(H,H^k)$-Lipschitz continuity and the random attractor    in $H$ is in fact an $(H,H^k)$-random attractor of finite  fractal dimension in $H^k$ for $k\geq\frac{5}{2}$.

\section{Preliminaries:   $(X,Y)$-random attractors and  their fractal dimension}\label{sec2}

In this section, we will introduce the basic concepts and $(X,Y)$-random attractor theory and their fractal dimensions,  see, e.g. \cite{arnold,crauel,cui18jdde,cui,cui2025,langa2}.

Suppose that $(X,\|\cdot\|_X)$ is a separable Banach space. Let $(\Omega,\mathcal{F},\mathbb{P})$ be a probability space, with
$\mathcal{F}$ being possibly not $\mathbb{P}$-complete, endowed with a group $\{\theta_t\}_{t\in\mathbb{R}}$ of measure-preserving self-transformations on $\Omega$. Let $\mathcal{B}(\cdot)$ be the Borel $\sigma$-algebra of a metric space, and let $\R^+:=[0,\infty)$.

\begin{definition}
A mapping $\phi:\mathbb{R}^+\times\Omega\times X\rightarrow X$ is called a {\em random dynamical system (abbrev. RDS)}   in $X$, if
\begin{enumerate}[(i)]
\item $\phi$ is $(\mathcal{B}(\mathbb{R}^+)\times\mathcal{F}\times\mathcal{B}(X),\mathcal{B}(X))$-measurable;

\item  for any $\omega\in\Omega$, $\phi(0,\omega,\cdot)$ is the identity on $X$;

\item  it holds the  following cocycle property
\begin{align*}
\phi(t+s,\omega,x)=\phi(t,\theta_s\omega,\phi(s,\omega,x)), \quad \forall t,s\in \mathbb{R}^+,~\omega\in\Omega,~x\in X;
\end{align*}
\item the mapping $x\mapsto \phi(t,\omega, x)$ is continuous.
\end{enumerate}
\end{definition}

A  set-valued mapping $\mathfrak D$: $\Omega\mapsto 2^X\setminus \emptyset $, $ \omega\mapsto \mathfrak D(\omega) $, is said to be a \emph{random set}  in $X$ if it is   \emph{measurable} in the sense that the mapping $\omega\to \dist_X(x, \mathfrak D(\omega))$    is  $(\mathcal F,\mathcal B(\R))$-measurable for any $x\in X$. If each its image $\mathfrak D(\omega)$ is closed (or  bounded, compact, etc.) in $X$, then $\mathfrak D$ is called a  closed (or  bounded, compact, etc.) random set in $X$.

A random variable $\zeta:\Omega\mapsto X$ is defined {\em tempered} if
\[
  \lim_{t\to \infty} e^{-\varepsilon t} |\zeta(\theta_{-t}\omega)| =0, \quad \varepsilon >0.
\]
If $\zeta $ is a  tempered random variable, then so is $|\zeta|^k$ for any $k\in \mathbb N$.
A random set $\mathfrak D$ in $X$  is said to be {\emph{tempered}}  if there exists a   tempered random variable $\zeta$ such that $\|\mathfrak D(\omega)\|_X:=\sup_{x\in \mathfrak D(\omega)} \|x\| \leq \zeta(\omega)$ for any $\omega\in \Omega$.

Let $\mathcal{D}_X$  be the collection of all the tempered random sets in $X$. We introduce the basic concepts as the following.

 \begin{definition}
A    random set $\mathfrak{B}$ in $X$ is called a {\emph{(random)  absorbing set}} of an RDS $\phi$ if  it pullback absorbs every tempered random set in $X$, i.e., if for any  $\mathfrak D\in\mathcal{D}_X$   there exists a random variable $T_{\mathfrak D} (\cdot)$ such that
\begin{align*}
\phi(t,\theta_{-t}\omega, \mathfrak D(\theta_{-t}\omega))\subset \mathfrak{B}(\omega),\quad t\geq T_{\mathfrak D}(\omega),\  \omega\in\Omega.
\end{align*}
\end{definition}

\begin{definition}
A random set $\mathcal{A}$  is said  the \emph{random attractor}   of an RDS  $\phi$ in $X$ if
\begin{enumerate}[(i)]
\item  $\mathcal{A}$ is a tempered and compact random set in $X$;

 \item  $\mathcal{A}$ is invariant under $\phi$, i.e.,
\begin{align*}
\mathcal{A}(\theta_t\omega)=\phi(t,\omega,\mathcal{A}(\omega)), \quad t\geq0,~\omega\in\Omega;
\end{align*}

 \item $\mathcal{A}$ denotes pullback attracts each tempered set in $X$, i.e.,  for every $\mathfrak D\in  \mathcal{D}_X$,
\begin{align*}
\lim\limits_{t\rightarrow\infty} {\rm dist}_{X} \big(\phi(t,\theta_{-t}\omega,\mathfrak D(\theta_{-t}\omega)), \, \mathcal{A}(\omega) \big)=0,\quad \omega\in\Omega ,
\end{align*}
where $\dist _{X}(\cdot, \cdot) $  is  the Hausdorff semi-metric between  non-empty subsets  of  $ X$, i.e.,
\begin{align*}
 \dist_{X}(A,B)= \sup_{a\in A} \inf_{b\in B} \|a-b\|_X, \quad \forall A,B\subset X.
\end{align*}
\end{enumerate}
\end{definition}

Assume that $Y$ denotes a Banach space with continuous embedding $Y \hookrightarrow X$.
\begin{definition}
A random set $\mathcal{A}$   is said  the {\emph{$(X,Y)$-random attractor}}  of  an RDS $\phi$ in $X$,  if
\begin{enumerate}[(i)]
\item  $\mathcal{A}$ is a tempered and compact random set in $Y$;

 \item  $\mathcal{A}$ is invariant under $\phi$, i.e.,
\begin{align*}
\mathcal{A}(\theta_t\omega)=\phi(t,\omega,\mathcal{A}(\omega)), \quad t\geq0,~\omega\in\Omega;
\end{align*}

 \item $\mathcal{A}$ is pullback attracts each tempered set in $X$ in the metric of $Y$, i.e.,  for every $\mathfrak D\in  \mathcal{D}_X$,
 \[
\lim\limits_{t\rightarrow\infty} {\rm dist}_{Y} \big(\phi(t,\theta_{-t}\omega, \mathfrak D(\theta_{-t}\omega)), \, \mathcal{A}(\omega) \big)=0,\quad \omega\in\Omega .
 \]
\end{enumerate}
\end{definition}
Then we introduce the main criterion as following.
\begin{lemma}(\cite[Theorem 19]{cui18jdde}) \label{lem:cui18}
 Assume that $\phi$ is an RDS.  Suppose that
  \begin{enumerate}[(i)]
  \item  $\phi$ has a   random absorbing set $\mathfrak B$ which is a tempered and closed random set in $Y$;
  \item $\phi$  is $(X,Y)$-asymptotically compact, that is, for every $ \mathfrak D\in \D_X$ and $t_n\to \infty$,   any sequence
  $\{\phi(t_n, \theta_{-t_n} \omega, x_n)\}_{n\in \N}$ with $x_n \in \mathfrak D(\theta_{-t_n} \omega)$, $\omega\in \Omega$, has a convergent subsequence in $Y$.
  \end{enumerate}
    Then  $\phi$ has a unique   $(X,Y)$-random attractor $\A$ defined by
    \[
    \A(\omega) = \bigcap_{s\geq 0} \overline{ \bigcup_{t\geq  s} \phi(t,\theta_{-t} \omega, \mathfrak  B(\theta_{-t}\omega)) }^Y ,
    \quad \omega \in \Omega.
    \]
  \end{lemma}
Finally, we introduce the definition of the fractal dimension. Let $E$ be a precompact  set in $X$, and let $N(E, \varepsilon)$ be the minimum number of balls of radius $\varepsilon$ in $X$ required to cover $E$. Then the {\emph{fractal dimension of $E$}}  is defined as
\[
 d_f^X(E)=\limsup_{\varepsilon\to 0} \frac{ \log N(E,\varepsilon)}{ -\log \varepsilon} ,
\]
where the superscript ``$\ ^X$''   indicates the space to which  the fractal dimension is referred, see, e.g. \cite{langa2,robinson11}.

\section{The fractional three-dimensional Navier-Stokes equations and the   attractors in $H$}\label{sec3}
 In this section, we first  introduce the  functional spaces and basic result of the deterministic fractional three-dimensional NS equations, and then associate an RDS to the random fractional three-dimensional NS equations.    We construct  $H^\frac54$ random absorbing sets and the  $(H,H^\frac54)$-attractor of  the  solutions and has a finite fractal dimensional.

\subsection{Functional spaces}

In this subsection, we introduce the  functional spaces and Poincar\'e's inequality as follows,  see, e.g., \cite{robinson01,temam1,temam}.
  Let $(C_{per}^\infty(\mathbb T^3))^3 $ be the space of three component infinitely differentiable functions that are $L$-periodic in each direction, and by  $\LL^2$ the completion of $(C_{per}^\infty(\mathbb T^3))^3 $  with respect to the $(L^2(\TT^3))^3$ norm.

 The phase space $(H,\|\cdot\|)$ is a subspace of  $\LL^2$ of functions with zero mean and free divergence,  that is,
\begin{align*}
H= \left \{u\in \LL ^2:
\, \int_{\mathbb T^3}u  \d x=0,\  {\rm div}\, u=0 \right \},
\end{align*}
endowed with the norm of $\LL^2$, i.e.,
  $\|\cdot\|=\|\cdot\|_{\LL^2}$.   For any $p>2$,  we write $(L^p(\TT^3))^3$  simply as $L^p$ and its norm as $\|\cdot\|_{L^p}$.

Note that any $u\in \LL^2 $ has an expression $u=\sum_{j\in \mathbb{Z}^3} \hat u_j e^{{\rm i}j\cdot x}$, where  $ {\rm i}= \sqrt{-1}$ and $\hat u_j$ are Fourier coefficients.  Hence, $\int_{\mathbb T^3}    u  \d x=0$ is equivalent to $\hat u_0=0$. Then, any  $u\in H $ is in form as follows
\begin{align*}
 u =\sum_{j\in \mathbb Z^3\setminus\{0\}} \hat u_j e^{{\rm i}j\cdot x} .
\end{align*}
 Let $P: \LL ^2\rightarrow H$ denote the Helmholtz-Leray orthogonal projection operator. In  this periodic space, we define the stokes operator  $Au=-P\Delta u=-\Delta u$ for any $u\in D(A) $, here $D(A)$ is given as follows. In  this bounded domain, the projector $P$ can not commute with derivatives. Then the operator $A^{-1}$  is a self-adjoint positive-definite compact operator from $H$  to $H$.

   For $s >0 $,    the Sobolev space  $H^s:=D(A^{s/2})$  is given in a standard way by
\begin{align*}
H^s =\left \{u\in H:\|u\|^2_{H^s}=\sum\limits_{j\in \mathbb{Z}^{3}\backslash \{0\}}|j|^{2s}|\hat{u}_j|^2<\infty \right \},
\end{align*}
endowed with the norm $\|\cdot\|_{H^s}=\|A^{s/2}\cdot \|$.

For $u,v\in H^1$,  we introduce the following bilinear form
\begin{align*}
B(u,v)=P((u\cdot\nabla)v) ,
\end{align*}
and, in particular, $B(u):=B(u,u)$.

We introduce  the following Poincar\'e's inequality, see, e.g., \cite{Huo2016,robinson01,temam}.

\begin{lemma}[Poincar\'e's inequality]  If $u\in H^{\frac{5}{4}}$, then
\begin{equation} \label{poin}
 \lambda_1^{\frac{5}{4}} \|u\|^2 \leq \| u\|_{H^{\frac{5}{4}}}^2,
\end{equation}
here $\lambda_1$ is the first eigenvalue of $-\Delta$.
\end{lemma}
For later purpose we  denote by $\lambda:= \alpha \nu \lambda_1^{\frac{5}{4}}/4$. Assume that  the noise  intensity $h$  satisfy Assumption  \ref{assum}, i.e.,     $ h \in H^{\frac{15}{4}}$ and
\begin{align}\label{2.3}
  \|\nabla h\|_{L^\infty} <  {\sqrt \pi} \nu \lambda_1^{\frac{5}{4}} .
\end{align}
 (Especially, $\|\nabla h\|_{L^\infty} < \sqrt{\pi} \nu $  for the scale $L=2\pi$).
Obviously, the  tolerance of the noise  intensity     increases as the kinematic viscosity $\nu$ increases.

\subsection{Global attractor of the deterministic fractional three-dimensional Navier-Stokes equations}

In this subsection, we introduce  basic result of the deterministic fractional three-dimensional NS equations \eqref{2}. We   now  recall the global attractor of the deterministic fractional three-dimensional NS equations \eqref{2}, i.e. for $h(x)\equiv 0$,  which will be crucial for our analysis later. Under projection $P$, the equations \eqref{2}   are rewritten as
 \begin{align}\label{2.1}
   \partial_tu+ \nu A^{\frac54}u+B(u)  =f,\quad u(0)=u_{0}.
\end{align}
This equations   generate an autonomous dynamical system $S$ in phase space $H$ with a global attractor $\A_0$. Furthermore,    the following regularity result of the attractor $\A_0$ seems  optimal for $f\in H$  in the literature,  see, e.g., \cite{liu2023,liu,robinson01}.

\begin{lemma}\label{lem:det}
For $f\in H$, then the deterministic fractional three-dimensional NS equation \eqref{2.1}  generate a semigroup $S$ which possesses a global attractor $\A_0$ in $H$ and has an
  absorbing set which is bounded in $H^{\frac{5}{2}}$. In particular,    the global attractor $\A_0$   is   bounded in $H^{\frac{5}{2}}$.
\end{lemma}

By Lemma \ref{lem:det} and \cite{langa2,liu2023,robinson01}, we show that the global attractor $\mathcal A$ is not only bounded in $H^{\frac{5}{2}}$, but also has   finite fractal dimension (and thus compact) in $H^{\frac{5}{2}}$.

\begin{remark}  \label{rmk1}
When $f\in H$, we  can not get an estimation  of  $\|A^\frac54u\|$ directly and thus Lemma \ref{lem:det} was shown by estimating the time-derivatives $ \partial_t u $  of solutions    rather than estimating the solutions  $u$ themselves.       Then the   method  can not apply to the  random fractional three-dimensional NS  equations \eqref{2.2}  since the Wiener process is not derivable in time.  Therefore, in Section \ref{sec4} we will apply a comparison approach  to prove $H^\frac52$ random absorbing sets for \eqref{2.2}.
\end{remark}
\begin{remark}
Based on \cite{song,song1}, we will also prove $(H,H^2)$-global attractor and $(H,H^2)$-uniform attractor of the deterministic three-dimensional NS equations with damping.
\end{remark}

\subsection{Generation of an RDS}\label{sec2.1}

In this subsection, we   introduce  associate an RDS to the random fractional three-dimensional NS equations.
Firstly, we will transform the fractional stochastic three-dimensional Navier-Stokes equations \eqref{1}   to a fractional random three-dimensional NS equations  which    generate an RDS $\phi$ on $H$.
          Inspired by the references in \cite{arnold,crauel97jdde},    we set
 \[
z(\theta_t\omega)=-\int_{-\infty}^0e^\tau(\theta_t\omega)(\tau)\d \tau,\quad \forall t\in \mathbb{R},\ \omega\in\Omega.
\]
Then $z(\theta_t\omega)$ is   the one-dimensional Ornstein-Uhlenbeck process which solves the following equation
 \[
\d z(\theta_t\omega)+z(\theta_t\omega) \d t= \d \omega.
 \]
By virtue of \cite{arnold,wang2024,zhou}, there exists a $\theta_t $-invariant  set of full measure $\tilde \Omega\subset \Omega$  such that $z(\theta_t\omega)$ is continuous in $t$  for every $\omega \in \tilde\Omega$ and
\begin{align*}
\lim_{t\to \pm\infty}\frac{|z(\theta_t\omega)|}{|t|}=0,\quad {\rm and} \quad \lim_{t\to \pm\infty}\frac 1t \int^t_0 z(\theta_s\omega) \d s =0,
\end{align*}
and
\begin{align} \label{erg}
\lim\limits_{t\rightarrow \pm \infty}\frac{1}{t}\int_0^t|z(\theta_s\omega)|^m \d s=\mathbb{E}  \big (|z(\theta_t\omega)|^m
\big) =\frac{\Gamma  (\frac{1+m}{2} )}{\sqrt{\pi}}
,\quad \omega\in \tilde \Omega,
\end{align}
for  any $m\geq 1$, where $\Gamma$ is the Gamma function. Moreover, we have the random variable $|z(\cdot )|$ is tempered, namely, for each $\varepsilon>0$ it yields
\begin{align*}
   \lim_{t \to \infty} e^{-\varepsilon t} | z(\theta_{-t}\omega )| =0,\quad \forall  \omega\in \Omega .
\end{align*}
From now on, we will not distinguish $\tilde \Omega$ and $\Omega$. We next introduce the following transformation
 \[
u_h(t)=v(t)+hz(\theta_t\omega),\quad t>0,\ \omega\in \Omega.
 \]
Under projection $P$, the  system  \eqref{1} of $u_h$  is  rewritten as   the following  abstract evolution equations in $H$
\begin{align}\label{2.2}
\frac{\d v}{\d t}+ \nu A^{\frac54}v+B\big( v(t)+hz(\theta_t\omega)\big)=f(x)-  \nu A^{\frac54}hz(\theta_t\omega)+hz(\theta_t\omega),
\end{align}
with $v(0)=u_h(0)-hz(\omega)$.

By using the classic Galerkin method, we get the  existence and uniqueness of weak solutions of system \eqref{2.2} as stated below.
\begin{lemma}
Assume that  $f\in H$ and Assumption \ref{assum} hold.  For any $v(0)\in H$ and $\omega\in\Omega$, then there exists a unique weak solution
\begin{align*}
v (t) \in C_{loc}([0,\infty);H)\cap L^2_{loc} (0,\infty;H^{\frac{5}{4}} )
\end{align*}
satisfying \eqref{2.2} in distribution sense with $v|_{t=0}=v(0)$.  In addition, this solution is continuous in initial values in $H$.
\end{lemma}

Hence, by
\begin{align*}
 \phi(t,\omega, v_0) =v(t,\omega,v_0),\quad  \forall t\geq 0,\, \omega\in \Omega, \, v_0\in H,
\end{align*}
where $v$ is the solution of system \eqref{2.2},
we associated an RDS $\phi$ in $H$ to system \eqref{2.2}.

\subsection{The random attractor in $H$}
In this subsection, we    construct  $H^\frac54$ random absorbing sets and the  $(H,H^\frac54)$-attractor of  the  solutions and has a finite fractal dimensional. We first prove the random attractor $\A$ in $H$ for the RDS $\phi$.     Under a new   condition \eqref{2.3} in Assumption \ref{assum} and initial values $v(0)\in H$, by using the standard method, we will prove the bounds $H$  and $H^{\frac{5}{4}}$ of  solutions of system \eqref{2.2}.  In addition, we will prove the random attractor has a finite fractal dimensional in $H$, as the Theorem \ref{theorem4.6}. The main results are important for proving the following  $H^{\frac{5}{2}}$ random absorbing sets and the local $(H,H^{\frac{5}{2}})$-Lipschitz continuity of  $\phi$.

 Under the condition \eqref{2.3} in Assumption \ref{assum}, we first prove some uniform estimates on solutions  of system \eqref{2.2} corresponding to initial values $v(0)$ in $ H$.   In  this paper,  let $C$  be a positive constant whose value may vary  from line to line.

  We will frequently make use of  two  fundamental lemmas as follows,  see, e.g., \cite{temam1,temam}.

\begin{lemma}[{Gronwall's lemma}]
Assume that $x(t)$ is  a function from $\R$ to $\R^+$ such that
\begin{align*}
\dot x +a(t)x\leq b(t) .
\end{align*}
 Then for any $t\geq s $,
\begin{align*}
 x(t)\leq e^{ -\int^t_s a(\tau)  \d \tau} x(s)+\int_s^t e^{-\int^t_\eta a(\tau) \, \d \tau } b(\eta)  \d \eta.
\end{align*}
\end{lemma}

\begin{lemma}[Gagliardo-Nirenberg inequality]  If $u\in L^q(\mathbb{T}^3)$, $D^{m}u\in L^r(\mathbb{T}^3)$, $1\leq q,r\leq\infty$, then there exists a positive constant $C$ such that
\begin{align*}
 \|D^ju\|_{L^p} \leq  C \|D^mu\|_{L^r}^a\|u\|^{1-a}_{L^q},
\end{align*}
where
\begin{equation*}
\frac{1}{p}-\frac{j}{3}=a \left (\frac{1}{r}-\frac{m}{3} \right)+ \frac{1-a}{q}, \quad 1\leq p\leq\infty,~0\leq j\leq m,~\frac{j}{m}\leq a\leq1,
\end{equation*}  and
$C$ depends only on $\{m, \, j, \, a, \, q, \,r \}$.
\end{lemma}

\begin{lemma}[$H$ bound]\label{lemma4.1}
Assume that  $f\in H$ and Assumption \ref{assum} hold.
Then there exist random variables  $T_1(\omega)$ and $\zeta_1(\omega)$,   where $\zeta_1$ is tempered, such that any
  solution $v$  of the system \eqref{2.2} corresponding to initial  values  $v(0) \in H $ satisfies
  \begin{align}
 \|v(t,\theta_{-t} \omega, v(0))\|^2
 \leq e^{-\lambda t}\|v(0)\|^2
 + \zeta_1(\omega) ,\quad t\geq T_1(\omega).  \nonumber
\end{align}
\end{lemma}
\begin{proof}
By Assumption 1 and \eqref{2.3}, $  {\|\nabla h\|_{L^\infty} }/{\sqrt \pi} < \nu  \lambda_1^{\frac{5}{4}}$, then there exist  $  \alpha \in  (0,1]$ and $\beta>0$  such that
\begin{align}
  \frac{\|\nabla h\|_{L^\infty} }{\sqrt \pi}  &=  (1- \alpha) \nu  \lambda_1^{\frac{5}{4}},  \label{c1} \\
   \frac{  \|\nabla h\|_{L^\infty}  }{ \sqrt{\pi}}  \left( 1 +  \beta \right )  &=   (\nu -\frac  1 2 \alpha \nu )\lambda_1^{\frac{5}{4}}  .  \label{c2}
  \end{align}
 Taking the inner product of  the  system   \eqref{2.2} in $H$ by $v$ and  integration by parts we deduce
\begin{align*}
\frac12\frac{\d}{\d t}\|v\|^2+ \nu  \|A^{\frac58}v\|^2
 =-(B(v+hz(\theta_t\omega)) ,v)  + \big (f- \nu A^{\frac54}hz(\theta_t\omega) +hz(\theta_t\omega),v\big).
  \end{align*}
Applying Poincar\'e's inequality \eqref{poin}, it can deduce that
\begin{align}
 & \frac{\d}{\d t}\|v\|^2+ \left(2-\frac \alpha 2\right)  \nu  \lambda_1^{\frac{5}{4}}  \| v\|^2  + \frac {\alpha \nu } 2   \|A^{\frac58}v\|^2 \nonumber  \\
 &\quad \leq \frac{\d}{\d t}\|v\|^2  + 2\nu  \|A^{\frac58}v\|^2 \nonumber\\
&\quad  \leq 2 \big |\big(B(v+hz(\theta_t\omega)) ,v \big) \big|
+2\big| \big (f- \nu A^{\frac{5}{4}}hz(\theta_t\omega) +hz(\theta_t\omega),v \big)\big|  . \label{4.3}
\end{align}
By H\"{o}lder's inequality and Young's  inequality, it can deduce that
\begin{align}
   2\big | \big(B(v+hz(\theta_t\omega)) ,v \big) \big|
&\leq 2 |z(\theta_t\omega)|\int_{\mathbb{T}^3}|v|^2|\nabla h|\ \d x+ 2|z(\theta_t\omega)|^2\int_{\mathbb{T}^3}|v||h||\nabla h|\ \d x\nonumber\\
&\leq 2 |z(\theta_t\omega)|\|\nabla h\|_{L^\infty}\|v\|^2+ 2|z(\theta_t\omega)|^2\|\nabla h\|_{L^\infty}\|v\|\|h\|\nonumber\\
&\leq 2 |z(\theta_t\omega)|\|\nabla h\|_{L^\infty}\|v\|^2+  \frac{\alpha \nu \lambda_1^{\frac{5}{4}}} {8} \|v\|^2+ C|z(\theta_t\omega)|^4\|\nabla h\|_{L^\infty}^2\|h\|^2 , \nonumber
\end{align}
and
\begin{align}
 2\big| \big (f-\nu A^{\frac{5}{4}}hz(\theta_t\omega) +hz(\theta_t\omega),v \big)\big|  \leq \frac{\alpha \nu \lambda_1^{\frac{5}{4}}}{8}\|v\|^2+C\|f\|^2
 + C|z(\theta_t\omega)|^2 \left(\|A^{\frac{5}{4}}h\|^2 + \|h\|^2\right) .\nonumber
\end{align}
Inserting the above two inequalities   into \eqref{4.3} yields
\ben
&
\frac{\d}{\d t}\|v\|^2+ \left(2-\frac {3 \alpha} 4\right) \nu \lambda_1^{\frac{5}{4}}  \| v\|^2  + \frac {\alpha \nu }  2 \|A^{\frac58}v\|^2  \\
 &\quad \leq 2|z(\theta_t\omega)|\|\nabla h\|_{L^\infty}\|v\|^2+ C |z(\theta_t\omega)|^4\|\nabla h\|_{L^\infty}^2\|h\|^2\nonumber\\
&\qquad + C\|f\|^2+ C |z(\theta_t\omega)|^2 \left (\|h\|^2+\|A^{\frac{5}{4}}h\|^2 \right)\nonumber\\
&\quad \leq 2|z(\theta_t\omega)|\|\nabla h\|_{L^\infty}\|v\|^2+C \left (1+|z(\theta_t\omega)|^4 \right).
\ee
Hence,
\begin{align*}
\frac{\d}{\d t}\|v\|^2+ \left[ \left(2-\frac{3\alpha} 4\right ) \nu \lambda_1^{\frac{5}{4}} -2\|\nabla h\|_{L^\infty}|z(\theta_t\omega)| \right ]
\|v\|^2+\frac{ \alpha \nu } 2\|A^{\frac58}v\|^2
\leq C  \left (1+|z(\theta_t\omega)|^4 \right)  .
\end{align*}
By the  Gronwall lemma, we have
\begin{align}
&\|v(t)\|^2+\frac {\alpha \nu } 2\int_0^te^{ \left(2-\frac{3\alpha} 4\right )\nu  \lambda_1^{\frac{5}{4}}(s-t)+2\|\nabla h\|_{L^\infty}\int_{s}^t|z(\theta_\tau\omega)|\d \tau}\|A^{\frac58}v(s)\|^2 \d s\nonumber\\
&\quad \leq e^{-\left(2-\frac{3\alpha} 4\right ) \nu  \lambda_1^{\frac{5}{4}}t+2\|\nabla h\|_{L^\infty}\int_{0}^t|z(\theta_\tau\omega)| \d \tau}\|v(0)\|^2  \nonumber \\
& \qquad
+C\int_0^te^{ \left(2-\frac{3\alpha} 4\right ) \nu \lambda_1^{\frac{5}{4}}(s-t)+2\|\nabla h\|_{L^\infty}\int_{s}^t|z(\theta_\tau\omega)|\d \tau}  \left (1+|z(\theta_s\omega)|^4 \right) \d s, \nonumber
\end{align}
and then
\begin{align}
&\|v(t)\|^2+\frac {\alpha \nu} 2\int_0^te^{ \left(2-\frac{3\alpha} 4\right ) \nu \lambda_1^{\frac{5}{4}}(s-t) }\|A^{\frac58}v(s)\|^2 \d s\nonumber\\
&\quad \leq e^{-\left(2-\frac{3\alpha} 4\right ) \nu \lambda_1^{\frac{5}{4}}t+2\|\nabla h\|_{L^\infty}\int_{0}^t|z(\theta_\tau\omega)| \d \tau}\|v(0)\|^2  \nonumber \\
& \qquad
+C\int_0^te^{ \left(2-\frac{3\alpha} 4\right ) \nu \lambda_1^{\frac{5}{4}}(s-t)+2\|\nabla h\|_{L^\infty}\int_{s}^t|z(\theta_\tau\omega)| \d \tau}  \left (1+|z(\theta_s\omega)|^4 \right)  \d s. \label{sep6.5}
\end{align}

By the above equality \eqref{erg} for $m=1$, it can deduce that
\[
\lim\limits_{t\rightarrow \pm \infty}\frac{1}{t}\int_0^t|z(\theta_s\omega)| \, \d s=\mathbb{E}(|z(\theta_t\omega)| )=\frac{ 1}{\sqrt{\pi}}
, \quad \omega\in \Omega,
\]
then there exists a  random variable  $ T_1(\omega)\geq1$ such that, for any $t\geq T_1(\omega)$,
 \[
 \frac 1t \int_{0}^t|z(\theta_s\omega)|\, \d s
 \leq  \frac{  1 }{\sqrt{\pi}} + \frac{\beta}{\sqrt{\pi}}     ,
\]
and
\begin{align}
2\|\nabla h\|_{L^\infty}\int_{0}^t|z(\theta_s\omega)|\, \d s &\leq
   \frac{2\|\nabla h\|_{L^\infty}  }{ \sqrt{\pi}}  \left( 1 +  \beta  \right ) t  \nonumber \\
   &= \left (2 \nu  -  \alpha \nu   \right)\lambda_1^{\frac{5}{4}} t \quad \text{(by \eqref{c2})} .\label{erg1}
\end{align}
Inserting \eqref{erg1} into  \eqref{sep6.5}, it yields
\begin{align}
&\|v(t)\|^2+\frac {\alpha \nu }2\int_0^te^{ \left(2-\frac{3\alpha} 4\right )\nu  \lambda_1^{\frac{5}{4}}(s-t) }\|A^{\frac58}v(s)\|^2 \d s\nonumber\\
&\quad \leq e^{ - \frac{\alpha} 4  \nu  \lambda_1^{\frac{5}{4}}t }\|v(0)\|^2
+C\int_0^te^{ \left(2-\frac{3\alpha} 4\right )\nu  \lambda_1^{\frac{5}{4}}(s-t)+2\|\nabla h\|_{L^\infty}\int_{s}^t|z(\theta_\tau\omega)| \d \tau}  \left (1+|z(\theta_s\omega)|^4 \right) \d s. \nonumber
\end{align}
Let   $\lambda:= \alpha \nu    \lambda_1^{\frac{5}{4}} /4$,
 this estimate is given as
\begin{align}
&\|v(t)\|^2+\frac {\alpha \nu } 2\int_0^te^{ \left( \frac 8 \alpha -3  \right) \lambda (s-t)}\|A^{\frac58}v(s)\|^2 \d s\nonumber\\
 &\quad\leq e^{-\lambda t}\|v(0)\|^2
 +C\int_0^te^{ \left( \frac 8 \alpha -3  \right) \lambda (s-t)+2\|\nabla h\|_{L^\infty}\int_{s}^t|z(\theta_\tau\omega)| \d \tau}  \left (1+|z(\theta_s\omega)|^4 \right)  \d s. \nonumber
\end{align}
Replacing $\omega$ with $\theta_{-t}\omega$ yields
 \begin{align}
&\|v(t, \theta_{-t}\omega, v(0))\|^2+\frac {\alpha  \nu }2\int_0^te^{ \left( \frac 8 \alpha -3  \right) \lambda (s-t)}\|A^{\frac58}v(s, \theta_{-t}\omega, v(0))\|^2 \d s\nonumber\\
 &\quad\leq e^{-\lambda t}\|v(0)\|^2
 +C\int^0_{-t} e^{ \left( \frac 8 \alpha -3  \right) \lambda  s +2\|\nabla h\|_{L^\infty}\int_{s}^0 |z(\theta_\tau\omega)| \d \tau}  \left (1+|z(\theta_s\omega)|^4 \right) \d s. \nonumber
\end{align}
Let a random variable be defined  by
\ben
 \zeta_1(\omega) :=C\int_{-\infty}^0  e^{ \left( \frac 8 \alpha -3  \right) \lambda  s +2\|\nabla h\|_{L^\infty}\int_{s}^0 |z(\theta_\tau\omega)| \d \tau}  \left (1+|z(\theta_s\omega)|^4 \right)   \d s,\quad \omega\in \Omega ,
\ee
then $\zeta_1(\cdot)$ is a tempered random variable such that  $\zeta_1(\omega)\geq 1$. Moreover, we deduce that
\begin{align}
&\|v(t,\theta_{-t} \omega, v(0))\|^2+ \frac {\alpha \nu } 2\int_0^t e^{ \left ( \frac{8}{\alpha} -3  \right) \lambda (s-t)}\|A^{\frac58}v(s, \theta_{-t}\omega, v(0))\|^2\, \d s \nonumber \\
 &\quad \leq e^{-\lambda t}\|v(0)\|^2
 + \zeta_1(\omega) ,\quad t\geq T_1(\omega).  \label{4.29}
\end{align}
This completes the proof of Lemma \ref{lemma4.1}.
\end{proof}

\begin{lemma}[$H^{\frac54}$ bound] \label{lem:H1bound}
Assume that  $f\in H$ and Assumption \ref{assum} hold.
For any bounded set $ B$ in $ H$, then there exists a random variable $ T_B(\omega)>T_1(\omega)$ such that,  for any $t\geq T_B(\omega)$,
\[
\sup_{v(0)\in B} \left( \big\|A^{\frac58}v(t,\theta_{-t}\omega,v(0))  \big\|^2 + \int_{t-\frac 12 }^t \|A^{\frac54}v(s,\theta_{-t}\omega,v(0))\|^2   \d s\right)
\leq  \zeta_2(\omega),
\]
here $\zeta_2(\omega)$  is defined by \eqref{mar8.8} such that $\zeta_2(\omega) > \zeta_1(\omega)\geq 1$, $ \omega\in\Omega$.
\end{lemma}
\begin{proof}
Taking the inner product of the system \eqref{2.2} in $H$ by $A^{\frac54}v$ and  integration by parts we deduce
\begin{align}\label{4.6}
\frac12\frac{\d}{\d t}\|A^{\frac58}v\|^2+ \nu \|A^{\frac54}v\|^2&=-\big(B(v+hz(\theta_t\omega),v+hz(\theta_t\omega)), A^{\frac54}v \big)\nonumber\\
&\quad \ +\big(f-\nu A^{\frac54}hz(\theta_t\omega) +hz(\theta_t\omega),A^{\frac54}v \big).
\end{align}
 Applying the H\"{o}lder inequality,  the Sobolev embedding inequality,  and the Young inequality, we deduce
\begin{align}\label{4.4}
&\big| \big(B(v+hz(\theta_t\omega),v+hz(\theta_t\omega)), A^{\frac{5}{4}}v \big)\big|
 \nonumber\\
& \quad \leq \|A^{\frac{5}{4}}v\|\|v+hz(\theta_t\omega)\|_{L^{12}}
\|A^{\frac12}(v+hz(\theta_t\omega))\|_{L^{\frac{12}{5}}}\nonumber\\
& \quad \leq C\|A^{\frac{5}{4}}v\|\|A^{\frac{5}{8}}(v+hz(\theta_t\omega))\|\|A^{\frac{5}{8}}(v+hz(\theta_t\omega))\|\nonumber\\
& \quad \leq \frac{\nu}{4}\|A^{\frac{5}{4}}v\|^2+C\|A^{\frac{5}{8}}(v+hz(\theta_t\omega))\|^2\|A^{\frac{5}{8}}(v+hz(\theta_t\omega))\|^2
\nonumber\\
& \quad \leq \frac{\nu}{4}\|A^{\frac{5}{4}}v\|^2+C(\|A^{\frac{5}{8}}v\|^2+\|A^{\frac{5}{8}}h\|^2|z(\theta_t\omega)|^2)
(\|A^{\frac{5}{8}}v\|^2+\|A^{\frac{5}{8}}h\|^2|z(\theta_t\omega)|^2)\nonumber\\
& \quad \leq \frac{\nu}{4}\|A^{\frac{5}{4}}v\|^2+C(\|A^{\frac{5}{8}}v\|^2+|z(\theta_t\omega)|^2)
\|A^{\frac{5}{8}}v\|^2+C|z(\theta_t\omega)|^4,
\end{align}
and
\begin{align}
&\big (f-\nu A^{\frac{5}{4}}hz(\theta_t\omega) +hz(\theta_t\omega),A^{\frac{5}{4}}v \big)\nonumber\\
& \leq \frac  \nu 4 \|A^{\frac{5}{4}}v\|^2+C\|f\|^2+ C  \|A^{\frac{5}{4}}h\|^2|z(\theta_t\omega)|^2 + C  \|h\|^2|z(\theta_t\omega)|^2\nonumber\\
&\leq  \frac  \nu 4 \|A^{\frac{5}{4}}v\|^2+C\|f\|^2+ C  |z(\theta_t\omega)|^2 .\label{4.5}
\end{align}
Inserting \eqref{4.4} and \eqref{4.5} into \eqref{4.6} we deduce that
\begin{align}\label{4.27}
\frac{\d}{\d t}\|A^{\frac58}v\|^2 +  \nu  \|A^{\frac54}v\|^2
&\leq  C\left ( \|A^{\frac58}v\|^2+ |z(\theta_t\omega)|^2 \right)\|A^{\frac58}v\|^2  +
C \left(1 +|z(\theta_t\omega)|^4+\|f\|^2 \right)\nonumber \\
&\leq  C\left ( \|A^{\frac58}v\|^2+ |z(\theta_t\omega)|^2 \right)\|A^{\frac58}v\|^2  +
C \left(1 +|z(\theta_t\omega)|^4 \right).
\end{align}
By using Gronwall's lemma on $(\eta, t)$ with $\eta\in  (t-1,t )$, for $t >1$, it yields
\be
&\|A^{\frac58}v(t)\|^2
+ \nu \int_\eta ^te^{ \int_s^tC(\|A^{\frac{5}{8}}v\|^2+|z(\theta_\tau \omega)|^2) \d \tau}\|A^{\frac54} v(s)\|^2  \d s \nonumber\\
&\quad \leq e^{ \int_\eta^tC(\|A^{\frac{5}{8}}v\|^2+|z(\theta_\tau\omega)|^2) \d \tau}\|A^{\frac58}v( \eta)\|^2 +C\int_\eta^te^{ \int_s^t C(\|A^{\frac{5}{8}}v\|^2+ |z (\theta_\tau\omega)|^2) \d \tau} \left (1 +|z(\theta_s\omega)|^4 \right)\d s,
\ee
and then integrating over $\eta\in (t-1,t-\frac 12) $ yields
\be
& \frac12  \|A^{\frac58}v(t )\|^2
+\frac  \nu 2\int_{t-\frac 12 }^t e^{ \int_s^tC(\|A^{\frac{5}{8}}v\|^2+|z(\theta_\tau\omega)|^2) \d \tau}\|A^{\frac54}v(s)\|^2 \d s \nonumber\\
&\quad \leq \int_{t-1}^{t-\frac 12}  e^{ \int_\eta^tC(\|A^{\frac{5}{8}}v\|^2+|z(\theta_\tau\omega)|^2) \d \tau} \|A^{\frac58}v( \eta)\|^2  \d \eta
 \nonumber\\
&\qquad +C\int_{t-1}^te^{ \int_s^t C(\|A^{\frac{5}{8}}v\|^2+ |z (\theta_\tau\omega)|^2) \d \tau} \left (1 +|z(\theta_s\omega)|^4 \right) \d s \nonumber\\
&\quad \leq Ce^{   C \int_{ t-1}^t (\|A^{\frac{5}{8}}v\|^2+|z(\theta_\tau\omega)|^2)  \d \tau}  \left[  \int_{t-1}^{t  }  \|A^{\frac58}v( \eta)\|^2  \d \eta  + \int_{t-1}^t \left (1 +|z(\theta_s\omega)|^4 \right)\d s\right]  .
\ee
Replacing $\omega$ with $\theta_{-t}\omega$, it is easy to get
\be
&   \|A^{\frac58}v(t,\theta_{-t}\omega, v(0) )\|^2\nonumber\\
& +   \nu \int_{t-\frac 12 }^te^{  C\int_{s}^t\|A^{\frac{5}{8}}v(\tau ,\theta_{-t}\omega, v(0))\|^2\d \tau+C\int_{s-t}^0|z(\theta_\tau\omega)|^2 \d \tau}\|A^{\frac54} v(s ,\theta_{-t}\omega, v(0))\|^2 \d s\nonumber\\
&\leq  C e^{ C \int_{t-1}^t(\|A^{\frac{5}{8}}v(\tau ,\theta_{-t}\omega, v(0))\|^2+ |z (\theta_{\tau-t}\omega)|^2) \d \tau}\nonumber\\
&\cdot\left [ \int_{t-1}^t   \|A^{\frac58}v(s,\theta_{-t}\omega, v(0) )\|^2\d s+\int_{t-1}^t \left(1+ |z(\theta_{s-t} \omega)|^4  \right)  \d s\right]\nonumber\\
& =  C e^{ C \int_{t-1}^t\|A^{\frac{5}{8}}v(\tau ,\theta_{-t}\omega, v(0))\|^2\d \tau+C \int_{-1}^0 |z (\theta_{\tau}\omega)|^2 \d \tau}\nonumber\\
&\cdot\left [ \int_{t-1}^t   \|A^{\frac58}v(s,\theta_{-t}\omega, v(0) )\|^2\d s+\int_{-1}^0 \left(1+ |z(\theta_{s} \omega)|^4  \right)  \d s\right].
\ee
By  virtue of the Lemma   \ref{lemma4.1}, it can deduce that
\be
 \nu  \int_{t-1}^t  \|A^{\frac58}v(s,\theta_{-t}\omega, v(0) )\|^2  \d s
 & \leq   \nu e^{\left( \frac 8\alpha -3 \right)  \lambda }   \int_{t-1}^t  e^{\left( \frac 8\alpha -3 \right)  \lambda (s-t) } \|A^{\frac58}v(s,\theta_{-t}\omega, v(0) )\|^2  \d s
 \\
 & \leq  \frac{2}{\alpha }   e^{\left( \frac 8\alpha -3 \right)  \lambda }   \left( e^{-\lambda t}\|v(0)\|^2
 +   \zeta_1(\omega)\right) , \quad    t\geq T_1(\omega) . \nonumber
\ee
Hence, we deduce that for $t\geq T_1(\omega) $
\be
&   \|A^{\frac58}v(t,\theta_{-t}\omega, v(0) )\|^2
 +  \nu \int_{t-\frac 12 }^t  \|Av^{\frac54} (s ,\theta_{-t}\omega, v(0))\|^2 \d s \nonumber\\
 & \leq  \|A^{\frac58}v(t,\theta_{-t}\omega, v(0) )\|^2\nonumber\\
& +   \nu \int_{t-\frac 12 }^te^{  C\int_{s}^t\|A^{\frac{5}{8}}v(\tau ,\theta_{-t}\omega, v(0))\|^2\d \tau+C\int_{s-t}^0|z(\theta_\tau\omega)|^2 \d \tau}\|A^{\frac54} v(s ,\theta_{-t}\omega, v(0))\|^2 \d s \nonumber\\
& \leq   C e^{ C\int_{t-1}^t\|A^{\frac{5}{8}}v(\tau ,\theta_{-t}\omega, v(0))\|^2\d \tau+C \int_{-1}^0 |z (\theta_\tau\omega)|^2 \d \tau} \left(   e^{-\lambda t}\|v(0)\|^2
 +   \zeta_1(\omega)  +
\int_{-1}^0 |z(\theta_{s} \omega)|^4 \d s \right) \nonumber\\
& \leq   C e^{ C(e^{-\lambda t}\|v(0)\|^2
 +   \zeta_1(\omega))+C \int_{-1}^0 |z (\theta_\tau\omega)|^2 \d \tau} \left(   e^{-\lambda t}\|v(0)\|^2
 +   \zeta_1(\omega)  +
\int_{-1}^0 |z(\theta_{s} \omega)|^4 \d s \right) ,
\ee
here  we have used  $\zeta_1(\omega)\geq 1$. Let
\be \label{mar8.8}
 \zeta_2(\omega) :=   C e^{ C(\zeta_1(\omega) +\int_{-1}^0 |z (\theta_\tau\omega)|^2 \d \tau)} \left( \zeta_1(\omega)  +
\int_{-1}^0 |z(\theta_{s} \omega)|^4   \d s \right) ,
\quad \omega\in \Omega,
\ee
then $\zeta_2(\omega)$ is a tempered random variable such that $\zeta_2(\omega) > \zeta_1(\omega)\geq 1$.  This completes the proof of Lemma \ref{lem:H1bound}.

\end{proof}

Let a random set $\mathfrak{B}=\{\mathfrak{B}(\omega)\}_{\omega\in\Omega}$ in $H^{\frac54}$ be given by
\begin{align}\label{4.19}
\mathfrak{B}(\omega):= \left \{ v \in H^{\frac{5}{4}}:\|A^{\frac58}v\|^2\leq \zeta_2(\omega)
\right \}, \quad \omega\in\Omega,
\end{align}
where $\zeta_2(\cdot)$ is given by \eqref{mar8.8}. By the above Lemma \ref{lem:H1bound}, then  $\mathfrak{B}$ is a  random absorbing set of the RDS $\phi$ generated by the random fractional three-dimensional  NS equations  \eqref{2.2}.  Since  the compact embedding $H^{\frac{5}{4}} \hookrightarrow H$, we show the compactness of $\mathfrak B(\omega)$ in $H$, then the RDS $\phi$ has a    random attractor $\A$ in $H$. Inspired by the \cite{langa2}, we give the random attractor has a finite fractal dimension in $H$ as the following the theorem.

\begin{theorem}\label{theorem4.6}  Let Assumption \ref{assum} hold and $f\in H$. The  RDS $\phi$ generated by  the random fractional three-dimensional NS equations \eqref{2.2} has a  random absorbing set $\mathfrak{B}$ which is a tempered and bounded random set in $H^{\frac{5}{4}} $, and has also a   random attractor $\mathcal A $ in $H$. Moreover, $\mathcal A$ has a finite fractal dimension in $H$:
\[
d_f^H(\A(\omega))  \leq d,\quad \omega \in \Omega,
\]
for some  $d>0$.
\end{theorem}

And then, in Section \ref{sec4}, Section \ref{sec5} and Section \ref{sec6}, we will prove that this attractor $\A$ is in fact an $(H,H^\frac52)$-random attractor of $\phi$.

\section{Construction of  an $H^{\frac{5}{2}}$ random absorbing set}\label{sec4}

In this section, we will construct $H^{\frac{5}{2}}$-distance between random and deterministic trajectories and an  $H^\frac52$ random absorbing set of system \eqref{2.2}.   By estimating  the difference between the   solutions of the random  equations \eqref{2.2} and that of the  deterministic  equations \eqref{2.1} within the global attractor $\A_0$, we will get the main result, see Remark \ref{rmk1}.  The comparison approach seems first used in \cite{cui}.

\subsection{$H^{\frac{5}{2}}$-distance between random and deterministic trajectories}

By the Theorem \ref{theorem4.6},  we have that the random fractional three-dimensional NS equations \eqref{2.2}  have an $H^\frac54$ random absorbing set $\mathfrak B$. Meanwhile, by the Lemma \ref{lem:det}, we have that  the deterministic  fractional three-dimensional NS equations  \eqref{2.1} have a global attractor $\mathcal A_0$ bounded in $H^\frac52$, By constructing the connection between  the random absorbing set  $\mathfrak B$ and the  global attractor $\mathcal A_0$, we will prove the Lemma \ref{lemma4.3} and the Lemma \ref{lem:H2bd}.
  Note that    the $\omega$-dependence  of each  random time in the following estimates will be crucial for later analysis.

\begin{lemma}[$H^{\frac54}$-distance]
\label{lemma4.3}
 Let Assumption \ref{assum} hold and $f\in H$. There   exist   random variables $T_{\mathfrak B}(\cdot)$ and    $\zeta_4(\cdot)$ such that the solutions $v$ of the random fractional three-dimensional NS equations \eqref{2.2} and the solutions $u$  of the deterministic  fractional three-dimensional NS equations \eqref{2.1} satisfy
\begin{align}
 \big \|A^{\frac 58} v(t,\theta_{-t}\omega,v(0))-A^{\frac 58} u(t,u(0)) \big\| ^2 \leq \zeta_4(\omega), \nonumber
\end{align}
for any $t\geq T_{\mathfrak B}(\omega)$ and $\omega\in \Omega$,  $v(0)\in\mathfrak{B}(\theta_{-t}\omega)$
 and $u(0)\in\mathcal{A}_0$.
\end{lemma}
\begin{proof}
Assume that  $w=v-u$ is the difference between the  solutions $v$ of random system \eqref{2.2} and the  solutions $u$ of deterministic system \eqref{2.1}, respectively. Then we have
\begin{align}\label{4.8}
\frac{\d w}{\d t}+ \nu  A^{\frac 54}w+B( v+hz(\theta_t\omega))-B(u) = hz(\theta_t\omega) - \nu A^{\frac 54}hz(\theta_t\omega).
\end{align}
Taking the inner product of the system  \eqref{4.8} in $H$ by $A^{\frac 54}w$ and  integration by parts, it can deduce that
\begin{align}\label{4.11}
 &\frac12\frac{\d}{\d t}\|A^{\frac58}w\|^2+ \nu \|A^{\frac 54}w\|^2\nonumber \\
&= \big (hz(\theta_t\omega)-\nu A^{\frac 54}hz(\theta_t\omega) ,A^{\frac 54}w \big )  - \big(B(v+hz(\theta_t\omega))-B(u),A^{\frac 54}w\big) \nonumber \\
&= \big  (hz(\theta_t\omega)-\nu A^{\frac 54}hz(\theta_t\omega) ,A^{\frac 54}w \big ) - \big (B(w+hz(\theta_t\omega),u),A^{\frac 54}w \big ) \nonumber\\
& -  \big (B(v+hz(\theta_t\omega),w+hz(\theta_t\omega)),A^{\frac 54}w  \big ) .
\end{align}
For the first term on the right hand of \eqref{4.11}, by using H\"{o}lder's inequality and Young's inequality, it can deduce that
\begin{align}\label{4.10}
 \big  (hz(\theta_t\omega)-\nu A^{\frac 54}hz(\theta_t\omega) ,A^{\frac 54}w \big )   &\leq \nu\|A^{\frac 54}hz(\theta_t\omega)\|\|A^{\frac 54}w\|+\|hz(\theta_t\omega)\|\|A^{\frac 54}w\|\nonumber\\
&\leq \frac  \nu 4\|A^{\frac 54}w\|^2+C\|A^{\frac 54}h\|^2|z(\theta_t\omega)|^2+C\|h\|^2|z(\theta_t\omega)|^2\nonumber\\
&\leq \frac  \nu 4\|A^{\frac 54}w\|^2+C|z(\theta_t\omega)|^2.
\end{align}
By $u$ is a trajectory within the global attractor $\mathcal{A}_0$, we deduce $\|A^{\frac58}u(t)\|\leq C$ for any $t\geq0$. For the second and third terms on the right hand of \eqref{4.11}, by using H\"{o}lder's inequality, the Sobolev embedding inequality and Young's inequality, it can deduce that
\begin{align}
\big (B(w+hz(\theta_t\omega),u),A^{\frac 54}w \big )&\leq \|w+hz(\theta_t\omega)\|_{L^{12}}\|\nabla u\|_{L^{\frac{12}{5}}}\|A^{\frac54}w\|\nonumber\\
&\leq C\|A^{\frac58}(w+hz(\theta_t\omega))\|\|A^{\frac58} u\|\|A^{\frac54}w\|\nonumber\\
&\leq\frac{\nu}{8}\|A^{\frac54}w\|^2+C\|A^{\frac58} u\|^2\|A^{\frac58}(w+hz(\theta_t\omega))\|^2\nonumber\\
&\leq\frac{\nu}{8}\|A^{\frac54}w\|^2+C\|A^{\frac58} u\|^2\|A^{\frac58}w\|^2+C\|A^{\frac58}h\|^2|z(\theta_t\omega)|^2\|A^{\frac58} u\|^2\nonumber\\
&\leq\frac{\nu}{8}\|A^{\frac54}w\|^2+C\|A^{\frac58} u\|^2\|A^{\frac58}w\|^2+C|z(\theta_t\omega)|^2\|A^{\frac58} u\|^2\nonumber\\
&\leq\frac{\nu}{8}\|A^{\frac54}w\|^2+C\|A^{\frac58}w\|^2+C|z(\theta_t\omega)|^2,
\end{align}
and
\begin{align}\label{4.9}
&\big (B(v+hz(\theta_t\omega),w+hz(\theta_t\omega)),A^{\frac 54}w  \big ) \nonumber\\
&\leq \|v+hz(\theta_t\omega)\|_{L^{12}}\|\nabla(w+hz(\theta_t\omega))\|_{L^{\frac{12}{5}}}\|A^{\frac 54}w\|\nonumber\\
&\leq C\|A^{\frac 58}(v+hz(\theta_t\omega))\|\|A^{\frac 58}(w+hz(\theta_t\omega))\|\|A^{\frac 54}w\|\nonumber\\
&\leq \frac{\nu}{8}\|A^{\frac 54}w\|^2+C\|A^{\frac 58}(v+hz(\theta_t\omega))\|^2\|A^{\frac 58}(w+hz(\theta_t\omega))\|^2\nonumber\\
&\leq \frac{\nu}{8}\|A^{\frac 54}w\|^2+C(\|A^{\frac 58}v\|^2+\|A^{\frac 58}h\|^2|z(\theta_t\omega)|^2)\|A^{\frac 58}w\|^2\nonumber\\
&+C(\|A^{\frac 58}v\|^2+\|A^{\frac 58}h\|^2|z(\theta_t\omega)|^2)\|A^{\frac 58}h\|^2|z(\theta_t\omega)|^2\nonumber\\
&\leq \frac{\nu}{8}\|A^{\frac 54}w\|^2+C(\|A^{\frac 58}v\|^2+|z(\theta_t\omega)|^2)\|A^{\frac 58}w\|^2\nonumber\\
&+C(\|A^{\frac 58}v\|^2+|z(\theta_t\omega)|^2)|z(\theta_t\omega)|^2\nonumber\\
&\leq \frac{\nu}{8}\|A^{\frac 54}w\|^2+C(\|A^{\frac 58}v\|^2+|z(\theta_t\omega)|^2)\|A^{\frac 58}w\|^2
+C(\|A^{\frac 58}v\|^4+|z(\theta_t\omega)|^4).
\end{align}
Inserting \eqref{4.10}-\eqref{4.9} into \eqref{4.11}, we deduce that
\begin{align}
\frac{\d}{\d t}\|A^{\frac58}w\|^2+  \nu \|A^{\frac54}w\|^2
&\leq C \left (1+\|A^{\frac58}v\|^2+|z(\theta_t\omega)|^2 \right )\|A^{\frac58}w\|^2  \nonumber \\
&\quad
+C \left(1+\|A^{\frac58}v\|^4+|z(\theta_t\omega)|^4\right) .
\label{sep6.1}
\end{align}

For $t >1$ and $s\in[t-1,t]$, by Gronwall's lemma, we have that
\begin{align}
\|A^{\frac58}w(t)\|^2&\leq e^{\int_s^tC(1+\|A^{\frac58}v(\eta) \|^2+|z(\theta_\eta\omega)|^2) \d \eta}\|A^{\frac58}w(s)\|^2\nonumber\\
&+C\int_s^te^{\int_\tau^tC(1+\|A^{\frac58}v(\eta) \|^2+|z(\theta_\eta\omega)|^2) \d \eta} \left(1+\|A^{\frac58}v(\tau)\|^4+|z(\theta_\tau\omega)|^4 \right) \d \tau\nonumber\\
&\leq Ce^{\int_s^tC(\|A^{\frac58}v(\eta) \|^2+|z(\theta_\eta\omega)|^2) \d \eta} \nonumber  \\
&\times
 \left[\|A^{\frac58}w(s)\|^2
+\int_{t-1}^t    \left(1+\|A^{\frac58}v(\tau)\|^4+|z(\theta_\tau\omega)|^4\right) \d \tau \right]. \nonumber
\end{align}
Integrating w.r.t$.$  $s$ over $[t-1,t]$, it can deduce that
\begin{align*}
\|A^{\frac58}w(t)\|^2
&\leq Ce^{\int_{t-1}^tC(\|A^{\frac58}v(\eta)\|^2+|z(\theta_\eta\omega)|^2) \d \eta}\nonumber\\
&\quad \times \left[\int_{t-1}^t\|A^{\frac58}w(s)\|^2 \d s
+\int_{t-1}^t  \left(1+\|A^{\frac58}v(\tau)\|^4+|z(\theta_\tau \omega)|^4\right) \d \tau\right]. \nonumber
\end{align*}
 By virtue of  $ u$ is a trajectory within the global attractor, $\|A^{\frac 58} u\|^2$ is uniformly bounded. Hence, we have
\begin{align*}
\|A^{\frac 58} w(s)\|^2
 =\|A^{\frac 58} (v-u) (s)\|^2
  \leq 2\|A^{\frac 58} v (s)\|^2 + C\leq \|A^{\frac 58} v (s)\|^4+ C,
\end{align*}
and
\ben
\|A^{\frac58}w(t)\|^2
 \leq Ce^{\int_{t-1}^tC(\|A^{\frac58}v(\eta)\|^2+|z(\theta_\eta\omega)|^2) \d \eta}  \int_{t-1}^t  \left(1+\|A^{\frac58}v(\tau)\|^4+|z(\theta_\tau\omega)|^4 \right)\d \tau .
\ee
Replacing $\omega$ with $\theta_{-t}\omega$, it is easy to get for $t>1$
\begin{align}\label{4.14}
  \|A^{\frac58}w(t,\theta_{-t}\omega,w(0))\|^2
&\leq Ce^{C \int_{t-1}^t\|A^{\frac58}v(\eta,\theta_{-t}\omega,v(0))\|^2 \d \eta+  C\int_{t-1}^t|z(\theta_{\eta-t}\omega)|^2
  \d \eta}  \nonumber\\
 &
 \times \left(\int_{t-1}^t  \|A^{\frac58}v(s,\theta_{-t}\omega,v(0))\|^4  \d s + \int_{t-1}^t  |z(\theta_{s-t}\omega)|^4 \d s +1 \right) \nonumber\\
 &\leq Ce^{C \int_{t-1}^t\|A^{\frac58}v(\eta,\theta_{-t}\omega,v(0))\|^2 \d \eta+  C\int_{-1}^0|z(\theta_{\eta}\omega)|^2
  \d \eta}  \nonumber\\
 &
 \times \left(\int_{t-1}^t  \|A^{\frac58}v(s,\theta_{-t}\omega,v(0))\|^4  \d s + \int_{-1}^0  |z(\theta_{s}\omega)|^4 \d s +1 \right).
\end{align}
By \eqref{4.27}, we have
\begin{align}\label{4.20}
\frac{\d}{\d t}\|A^{\frac58}v\|^2 + \nu \|A^{\frac54}v\|^2
 \leq  C\left ( \|A^{\frac58}v\|^2 +  |z(\theta_t\omega)|^2 \right)  \|A^{\frac58}v\|^2  +
C \left(1 +|z(\theta_t\omega)|^4  \right).
\end{align}
For any $t>3$,  integrating \eqref{4.20} over $(\tau, t)$ with $\tau\in (t-3,t-2)$, we deduce that
\ben
&
\|A^{\frac58}v (t) \|^2 -\|A^{\frac58}v(\tau) \|^2  + \nu \int^t_{t-2} \|A^{\frac54} v(s)\|^2 \d s \\
&\quad
 \leq  C \int^t_{t-3} \left(\|A^{\frac58}v(s) \|^2+  |z(\theta_s\omega)|^2 \right) \|A^{\frac58}v(s) \|^2 \d s
+ C\int^t_{t-3} \left(1 +|z(\theta_s\omega)|^4 \right )  \d s,
\ee
and  integrating $\tau$ over $(t-3,t-2)$, it can deduce that for $t> 3$
\ben
\|A^{\frac58}v (t) \|^2+  \nu \int^t_{t-2} \|A^{\frac54} v(s)\|^2 \d s
& \leq   C \int^t_{t-3} \left ( \|A^{\frac58}v(s) \|^2+|z(\theta_s\omega)|^2+1\right)  \|A^{\frac58}v(s) \|^2  \d s \\
&\quad + C\int^t_{t-3} \left (1 +|z(\theta_s\omega)|^4 \right)  \d s.
\ee
For $\varepsilon\in  [0,2]$, replacing $\omega$ with $\theta_{-t-\varepsilon} \omega$, it yields
\begin{align}
 & \|A^{\frac58}v (t,\theta_{-t-\varepsilon} \omega, v(0)) \|^2 +  \nu \int_{t-2}^t  \|A^{\frac54}v (s,\theta_{-t-\varepsilon} \omega, v(0)) \|^2  \d s \nonumber \\
& \leq   C \int^t_{t-3} \left (\|A^{\frac58}v(s,\theta_{-t-\varepsilon} \omega, v(0)) \|^2 +|z(\theta_{s-t-\varepsilon} \omega)|^2+1\right)  \|A^{\frac58}v(s,\theta_{-t-\varepsilon} \omega, v(0)) \|^2  \d s\nonumber\\
 &+ C\int^{t}_{t-3} \left(1 +|z(\theta_{s-t-\varepsilon}\omega)|^4 \right ) \d s\nonumber\\
& \leq   C \int^t_{t-3} \left (|z(\theta_{s-t-\varepsilon} \omega)|^2+1\right)  \|A^{\frac58}v(s,\theta_{-t-\varepsilon} \omega, v(0)) \|^2  \d s\nonumber\\
&+ C \int^t_{t-3} \|A^{\frac58}v(s,\theta_{-t-\varepsilon} \omega, v(0)) \|^2  \|A^{\frac58}v(s,\theta_{-t-\varepsilon} \omega, v(0)) \|^2  \d s\nonumber\\
 &+ C\int^{-\varepsilon}_{-\varepsilon-3} \left(1 +|z(\theta_{s}\omega)|^4 \right ) \d s\nonumber\\
& \leq   C\left (  \sup_{\tau\in (-5,0)}|z(\theta_\tau\omega)|^2+1\right) \int^{t+\varepsilon}_{t+\varepsilon-5}  \|A^{\frac58}v(s,\theta_{-t-\varepsilon} \omega, v(0)) \|^2 \d s \nonumber\\
& +   C \int^{t+\varepsilon}_{t+\varepsilon-5} \|A^{\frac58}v(s,\theta_{-t-\varepsilon} \omega, v(0)) \|^2 \|A^{\frac58}v(s,\theta_{-t-\varepsilon} \omega, v(0)) \|^2 \d s \nonumber\\
&+ C\int^{0}_{-5}    \left(1 +|z(\theta_s\omega)|^4 \right)   \d s ,\quad t>5 . \label{mar8.1}
\end{align}
By \eqref{4.29} and Lemma \ref{lem:H1bound}, there exists $t\geq T_B(\omega)>T_1(\omega)$ such that
\begin{align*}
  \frac {\alpha  \nu }2\int_0^te^{ \left( \frac 8 \alpha -3  \right) \lambda (s-t)}\|A^{\frac58}v(s, \theta_{-t}\omega, v(0))\|^2\d s
 \leq e^{-\lambda t}\|v(0)\|^2
 + \zeta_1(\omega)
 \end{align*}
 and
\begin{align*}
 \|A^{\frac58}v(t, \theta_{-t}\omega, v(0))\|^2
 \leq  \zeta_2(\omega)
 \end{align*}
where $\zeta_2(\omega)>\zeta_1(\omega)\geq 1$ are a tempered random variable, respectively. Then we have for $t\geq T_B(\omega)+5 $ and $\varepsilon \in [0,2]$
 \begin{align}\label{4.1}
  \int^{t+\varepsilon}_{t+\varepsilon-5}  \|A^{\frac58}v(s,\theta_{-t-\varepsilon} \omega, v(0)) \|^2  \d s
  &\leq  C \int^{t+\varepsilon}_{t+\varepsilon-5}  e^{ \left( \frac 8 \alpha -3  \right) \lambda (s-t)}\|A^{\frac58}v(s,\theta_{-t-\varepsilon} \omega, v(0)) \|^2  \d s \nonumber\\
  & \leq Ce^{-\lambda t}\|v(0)\|^2
 + C \zeta_1(\omega) ,
 \end{align}
and
 \begin{align}\label{4.2}
  &\int^{t+\varepsilon}_{t+\varepsilon-5}  \|A^{\frac58}v(s,\theta_{-t-\varepsilon} \omega, v(0)) \|^4  \d s\nonumber\\
  &\leq  C\sup\limits_{s\in[t+\varepsilon-5,t+\varepsilon]} \|A^{\frac58}v(s,\theta_{-t-\varepsilon} \omega, v(0)) \|^2\int^{t+\varepsilon}_{t+\varepsilon-5} e^{ \left( \frac 8 \alpha -3  \right) \lambda (s-t)} \|A^{\frac58}v(s,\theta_{-t-\varepsilon} \omega, v(0)) \|^2  \d s \nonumber\\
  & \leq C\zeta_2(\omega)(e^{-\lambda t}\|v(0)\|^2
 +  \zeta_1(\omega)) .
 \end{align}
Inserting \eqref{4.1} and \eqref{4.2} into \eqref{mar8.1}, it can deduce that for $t\geq T_B(\omega)+5 $
\begin{align}
 &\sup\limits_{\varepsilon\in[0,2]} \left(\|A^{\frac58}v (t,\theta_{-t-\varepsilon} \omega, v(0)) \|^2 +  \nu \int_{t-2}^t  \|A^{\frac54}v (s,\theta_{-t-\varepsilon} \omega, v(0)) \|^2  \d s \right)\nonumber \\
& \leq   C\left (  \sup_{\tau\in (-5,0)}|z(\theta_\tau\omega)|^2+1\right)(e^{-\lambda t}\|v(0)\|^2
 +  \zeta_1(\omega))  \nonumber\\
& +   C \zeta_2(\omega)(e^{-\lambda t}\|v(0)\|^2
 +  \zeta_1(\omega))+ C\int^{0}_{-5}    \left(1 +|z(\theta_s\omega)|^4 \right)   \d s \nonumber \\
& \leq   C\left (  \sup_{\tau\in (-5,0)}|z(\theta_\tau\omega)|^4+1\right)(e^{-\lambda t}\|v(0)\|^2
 +  \zeta_1(\omega))  \nonumber\\
& +   C \zeta_2(\omega)(e^{-\lambda t}\|v(0)\|^2
 +  \zeta_1(\omega)) .
\end{align}
For $t\geq T_B(\omega)+7$ and $\eta\in[t-2,t]$ such that $\eta\geq T_B(\omega)+5 $. Then we have  for $\eta\geq t-2$
\begin{align}\label{4.7}
& \|A^{\frac58}v (\eta ,\theta_{-t } \omega, v(0)) \|^2 + \nu \int_{\eta-2}^\eta  \|A^{\frac54} v (s,\theta_{-t} \omega, v(0)) \|^2  \d s\nonumber\\
  &  =  \|A^{\frac58}v (\eta ,\theta_{-\eta-(t-\eta) } \omega, v(0)) \|^2
  + \nu  \int_{\eta-2}^\eta  \|A^{\frac54} v (s,\theta_{-\eta-(t-\eta)} \omega, v(0)) \|^2 \d s \nonumber\\
& \leq    C\left (  \sup_{\tau\in (-5,0)}|z(\theta_\tau\omega)|^4+1\right)\left(e^{-\lambda  \eta}\|v(0)\|^2
 +   \zeta_1(\omega)   \right)+C\zeta_2(\omega)\left(e^{-\lambda  \eta}\|v(0)\|^2
 +   \zeta_1(\omega)   \right)
\nonumber\\
& \leq   C\left (  \sup_{\tau\in (-5,0)}|z(\theta_\tau\omega)|^4+1\right)\left(e^{-\lambda t}\|v(0)\|^2
 +   \zeta_1(\omega)   \right)+ C\zeta_2(\omega)\left(e^{-\lambda  t}\|v(0)\|^2
 +   \zeta_1(\omega)   \right).
\end{align}
Since the initial value $v(0)\in \mathfrak B(\theta_{-t} \omega)$, then there exists a random variable $T_{\mathfrak B}(\omega) \geq T_B(\omega)+ 7$ such that for any $t\geq T_{\mathfrak B}(\omega)$
\begin{align}\label{timeB}
 \sup_{v(0)\in \mathfrak B(\theta_{-t} \omega)}  e^{-\lambda t} \|v(0)\|^2 \leq 1.
\end{align}
Inserting \eqref{timeB} into \eqref{4.7}, it yields for any $t\geq T_{\mathfrak B}(\omega)$
\begin{align}\label{mar8.3}
 &\sup_{\eta\in [t-2,t]}
  \left(
 \|A^{\frac58}v (\eta ,\theta_{-t } \omega, v(0)) \|^2
 + \nu \int_{\eta-2}^\eta  \|A^{\frac54} v (s,\theta_{-t} \omega, v(0)) \|^2  \d s\right) \nonumber \\
& \leq  C\left (  \sup_{\tau\in (-5,0)}|z(\theta_\tau\omega)|^4+1\right)\big(1
 +   \zeta_1(\omega)   \big)+C\zeta_2(\omega)\big(1
 +   \zeta_1(\omega)   \big)
 =:  \zeta_3(\omega)  ,
\end{align}
here   $ \zeta_3 $ is a tempered random variable. Moreover, it is easy to deduce that
\begin{align}\label{mar8.2}
  \int_{t-1}^t\|A^{\frac58}v(\eta,\theta_{-t}\omega,v(0))\|^4  \d \eta
&  \leq    |\zeta_3(\omega)|^2 ,  \quad t\geq T_{\mathfrak B} (\omega) .
\end{align}
By \eqref{4.29}, there exists $t\geq T_{\mathfrak B}(\omega)$ such that
\begin{align}\label{4.12}
  \int^{t}_{t-1}  \|A^{\frac58}v(s,\theta_{-t} \omega, v(0)) \|^2  \d s
  &\leq  C \int^{t}_{t-1}  e^{ \left( \frac 8 \alpha -3  \right) \lambda (s-t)}\|A^{\frac58}v(s,\theta_{-t} \omega, v(0)) \|^2  \d s \nonumber\\
  & \leq Ce^{-\lambda t}\|v(0)\|^2
 + C \zeta_1(\omega)\nonumber\\
 & \leq C(1
 +  \zeta_1(\omega))\nonumber\\
 & \leq |\zeta_3(\omega)|.
 \end{align}
Inserting \eqref{mar8.2} and \eqref{4.12} into \eqref{4.14}, we deduce that for any $t\geq T_{\mathfrak B}(\omega)$
\begin{align}  \label{mar8.4}
\|A^{\frac58}w(t,\theta_{-t}\omega,w(0))\|^2
&\leq Ce^{ |\zeta_3 (\omega)|+ C\int_{-1}^0|z(\theta_{\eta}\omega)|^2  \d \eta}  \left( |\zeta_3(\omega)|^2 + \int_{-1}^0 |z(\theta_{s}\omega)|^4   \d s  \right) \nonumber \\
&\leq Ce^{ 2|\zeta_3 (\omega)|^2+ C\int_{-1}^0|z(\theta_{\eta}\omega)|^4   \d \eta}
=: \zeta_4(\omega) ,
\nonumber
\end{align}
here $\zeta_4(\omega)$ is a tempered random variable.   This completes the proof of Lemma \ref{lemma4.3}.
\end{proof}

\begin{lemma}[$H^{\frac{5}{2}}$-distance] \label{lem:H2bd}  Let Assumption \ref{assum} hold and $f\in H$.
There exists a tempered random variable $\zeta_5(\cdot)$ such that the solutions $v$ of the random fractional three-dimensional NS equations \eqref{2.2} and the solutions $u$  of the deterministic fractional three-dimensional NS equations  \eqref{2.1} satisfy
\[
 \|A^{\frac54}v(t,\theta_{-t}\omega,v(0))-A^{\frac54}u(t,u(0))\|^2\leq \zeta_5(\omega),
\]
for  any $t\geq T_{\mathfrak B}(\omega)$, $v(0)\in\mathfrak{B}(\theta_{-t}\omega)$
 and $u(0)\in\mathcal{A}_0$,
where $T_{\mathfrak B}$ is defined by \eqref{timeB}.
\end{lemma}
\begin{proof}
Taking the inner product of the system \eqref{4.8} in $H$ by $A^{\frac{5}{2}}w$ and   integration by parts, we can get
\begin{align}\label{4.17}
 \frac12\frac{\d}{\d t}\|A^{\frac54}w\|^2+  \nu  \|A^{\frac{15}{8}}w\|^2
&= \left(hz(\theta_t\omega)-\nu A^{\frac54}hz(\theta_t\omega) ,A^{\frac52}w \right) \nonumber \\
&-\left(B(v+hz(\theta_t\omega)) -B(u),A^{\frac52}w \right)\nonumber\\
&=  \left(hz(\theta_t\omega)-\nu A^{\frac54}hz(\theta_t\omega) ,A^{\frac52}w \right)\nonumber\\
& -\left(B(v+hz(\theta_t\omega),w+hz(\theta_t\omega)),A^{\frac52}w \right )\nonumber\\
&
-\left(B(w+hz(\theta_t\omega),u),A^{\frac52}w \right ).
\end{align}
Now we estimate \eqref{4.17} term by term to obtain \eqref{sep6.3}. For the first term of the right hand side of \eqref{4.17},
by using H\"{o}lder's inequality and the Young inequality and $h\in H^\frac{15}{4}$,  we deduce that
\begin{align}
&\left(hz(\theta_t\omega)-\nu A^{\frac54}hz(\theta_t\omega) ,A^{\frac52}w \right)\nonumber\\
&\leq  \nu \|A^{\frac58}h\||z(\theta_t\omega)|\|A^{\frac{15}{8}}w\|+ \|A^{\frac{15}{8}}h\||z(\theta_t\omega)|\|A^{\frac{15}{8}}w\|\nonumber\\
&\leq \frac{ \nu }{8}\|A^{\frac{15}{8}}w\|^2+C\|A^{\frac58}h\|^2  |z(\theta_t\omega)|^2+C\|A^{\frac{15}{8}}h\|^2|z(\theta_t\omega)|^2\nonumber\\
&\leq \frac{ \nu }{8}\|A^{\frac{15}{8}}w\|^2+C |z(\theta_t\omega)|^2. \label{4.18}
\end{align}
For the second term of the right hand side of \eqref{4.17}, by integration by parts and applying H\"{o}lder's inequality, the Sobolev embedding inequality, the Gagliardo-Nirenberg inequality and the Young inequality, we deduce that for $h\in H^\frac{15}{4}$
\begin{align}
&\left|\left(B(v+hz(\theta_t\omega),w+hz(\theta_t\omega)),A^{\frac52}w \right )\right|\nonumber\\
&\leq C\|A^\frac58(v+hz(\theta_t\omega))\|\|\nabla(w+hz(\theta_t\omega))\|_{L^\infty}\|A^{\frac{15}{8}}w\|
\nonumber\\
&+C\|v+hz(\theta_t\omega)\|_{L^{12}}\|A^{\frac98}(w+hz(\theta_t\omega))\|_{L^{\frac{12}{5}}}\|A^{\frac{15}{8}}w\|
\nonumber\\
&\leq C\|A^\frac58(v+hz(\theta_t\omega))\|\|w+hz(\theta_t\omega)\|^{\frac13}\|A^{\frac{15}{8}}(w+hz(\theta_t\omega))\|^{\frac23}
\|A^{\frac{15}{8}}w\|
\nonumber\\
&+C\|A^{\frac58}(v+hz(\theta_t\omega))\|\|A^{\frac54}(w+hz(\theta_t\omega))\|\|A^{\frac{15}{8}}w\|
\nonumber\\
&\leq \frac{\nu}{8}\|A^{\frac{15}{8}}w\|^2
+C\|A^\frac58(v+hz(\theta_t\omega))\|^2\|A^{\frac{5}{4}}(w+hz(\theta_t\omega))\|^{\frac23}
\|A^{\frac{15}{8}}(w+hz(\theta_t\omega))\|^{\frac43}\nonumber\\
&+C\|A^{\frac58}(v+hz(\theta_t\omega))\|^2\|A^{\frac54}(w+hz(\theta_t\omega))\|^2\nonumber\\
&\leq \frac{\nu}{4}\|A^{\frac{15}{8}}w\|^2+C\|A^{\frac{15}{8}}h\|^2|z(\theta_t\omega)|^2
+C\|A^\frac58(v+hz(\theta_t\omega))\|^6\|A^{\frac{5}{4}}(w+hz(\theta_t\omega))\|^2
\nonumber\\
&+C\|A^{\frac58}(v+hz(\theta_t\omega))\|^2\|A^{\frac54}(w+hz(\theta_t\omega))\|^2\nonumber\\
&\leq \frac{\nu}{4}\|A^{\frac{15}{8}}w\|^2+C|z(\theta_t\omega)|^2\nonumber\\
&+C(\|A^\frac58v\|^6+\|A^\frac58h\|^6|z(\theta_t\omega)|^6+\|A^\frac58v\|^2+\|A^\frac58h\|^2|z(\theta_t\omega)|^2)
(\|A^{\frac{5}{4}}w\|^2+\|A^{\frac54}h\|^2|z(\theta_t\omega)|^2)\nonumber\\
&\leq \frac{\nu}{4}\|A^{\frac{15}{8}}w\|^2+C(1+\|A^\frac58v\|^6+|z(\theta_t\omega)|^6)
\|A^{\frac{5}{4}}w\|^2\nonumber\\
&+C(1+\|A^\frac58v\|^6+|z(\theta_t\omega)|^6)|z(\theta_t\omega)|^2\nonumber\\
&\leq \frac{\nu}{4}\|A^{\frac{15}{8}}w\|^2+C(1+\|A^\frac58v\|^6+|z(\theta_t\omega)|^6)
\|A^{\frac{5}{4}}w\|^2\nonumber\\
&+C(1+\|A^\frac58v\|^8+|z(\theta_t\omega)|^8).
\end{align}
For the last term of the right hand side of \eqref{4.17}, by integration by parts and similar method and $u$ is a bounded in $H^{\frac{5}{2}}$, we also deduce that for $h\in H^\frac{15}{4}$
\begin{align}\label{4.16}
&\left|\left(B(w+hz(\theta_t\omega),u),A^{\frac52}w \right )\right|\nonumber\\
&\leq C\|A^{\frac58}(w+hz(\theta_t\omega))\|_{L^{12}}\|\nabla u\|_{L^{\frac{12}{5}}}\|A^{\frac{15}{8}}w\|+C\|w+hz(\theta_t\omega)\|_{L^{12}}
\|A^{\frac{9}{8}}u\|_{L^{\frac{12}{5}}}\|A^{\frac{15}{8}}w\|
\nonumber\\
&\leq C\|A^{\frac54}(w+hz(\theta_t\omega))\|\|A^\frac58u\|
\|A^{\frac{15}{8}}w\|  +C\|A^\frac58(w+hz(\theta_t\omega))\|\|A^\frac54u\|\|A^{\frac{15}{8}}w\|\nonumber\\
&\leq C\|A^\frac54(w+hz(\theta_t\omega))\|\|A^\frac54u\|\|A^{\frac{15}{8}}w\|\nonumber\\
&\leq C\|A^\frac54(w+hz(\theta_t\omega))\|\|A^{\frac{15}{8}}w\|\nonumber\\
&\leq \frac \nu 8\|A^{\frac{15}{8}}w\|^2+C\left(\|A^{\frac54}w\|^2+\|A^{\frac54}h\|^2|z(\theta_t\omega)|^2 \right) \nonumber\\
&\leq \frac \nu 8\|A^{\frac{15}{8}}w\|^2+C\|A^{\frac54}w\|^2+C|z(\theta_t\omega)|^2.
\end{align}
Inserting \eqref{4.18}-\eqref{4.16} into \eqref{4.17},  we have
\begin{align}
 \frac{\d}{\d t}\|A^{\frac{5}{4}}w\|^2+  \nu  \|A^{\frac{15}{8}}w\|^2
&\leq C(1+\|A^\frac58v\|^6+|z(\theta_t\omega)|^6)
\|A^{\frac{5}{4}}w\|^2\nonumber\\
&+C(1+\|A^\frac58v\|^8+|z(\theta_t\omega)|^8) .  \label{sep6.3}
\end{align}

 For any $t >1$ and $s\in[t-1,t]$, applying Gronwall's lemma to \eqref{sep6.3},   we deduce that
\begin{align*}
\|A^{\frac{5}{4}}w(t)\|^2&\leq e^{C\int_s^t( 1+\|A^\frac58 v(\eta)\|^6 +|z(\theta_\eta\omega)|^6)\d \eta}\|A^{\frac{5}{4}}w(s)\|^2 \nonumber\\
& +C\int_s^t e^{C\int_\tau^t  (1+\|A^\frac58 v(\eta)\|^6+|z(\theta_\eta\omega)|^6) \d \eta}
 \Big (1+\|A^\frac58 v(\tau)\|^8  +|z(\theta_\tau\omega)|^8\Big ) \d \tau\nonumber\\
&\leq Ce^{C\int_{t-1}^t(1+\|A^\frac58 v(\eta)\|^6 +|z(\theta_\eta\omega)|^6)\d \eta} \nonumber \\
&\times \left (\|A^{\frac{5}{4}}w(s)\|^2+\int_{t-1}^t
(1+\|A^\frac58 v(\tau)\|^8  +|z(\theta_\tau\omega)|^8) \d \tau \right) \\
& \leq Ce^{C\int_{t-1}^t(1+\|A^\frac58 v(\eta)\|^8 +|z(\theta_\eta\omega)|^8) \d \eta} \Big(\|A^{\frac{5}{4}}w(s)\|^2+ 1 \Big)
. \nonumber
\end{align*}
Integrating  $s$ over   $  [t-1,t] $,  we also deduce that
\begin{align}
\|A^{\frac{5}{4}}w(t)\|^2
&\leq Ce^{C\int_{t-1}^t  (1+\|A^\frac58 v(\eta)\|^8+|z(\theta_\eta\omega)|^8) \d \eta}
 \left (\int_{t-1}^t\|A^{\frac{5}{4}}w(s)\|^2 \d s + 1\right). \nonumber
\end{align}
    Replacing $\omega$ with $\theta_{-t}\omega$, it yields
\begin{align}\label{4.22}
\|A^{\frac{5}{4}}w(t,\theta_{-t}\omega,w(0))\|^2
&\leq Ce^{C\int_{t-1}^t \|A^\frac58 v(\eta,\theta_{-t}\omega, v(0))\|^8 \d \eta + C\int^t_{t-1} |z(\theta_{\eta-t}\omega) |^8
 \d \eta}\nonumber\\
&\times \left(\int_{t-1}^t\|A^{\frac{5}{4}}w(s,\theta_{-t}\omega,w(0))\|^2 \d s +1  \right)\nonumber\\
& = Ce^{C\int_{t-1}^t \|A^\frac58 v(\eta,\theta_{-t}\omega, v(0))\|^8 \d \eta + C\int^0_{-1} |z(\theta_{\eta}\omega) |^8
 \d \eta}\nonumber\\
& \times \left(\int_{t-1}^t\|A^{\frac{5}{4}}w(s,\theta_{-t}\omega,w(0))\|^2 \d s +1  \right).
\end{align}
By virtue of \eqref{mar8.3} it yields
\begin{align} \label{4.23}
  \int_{t-1}^t  \|A^\frac58 v(s,\theta_{-t}\omega,v(0))\|^8 \d s \leq  |\zeta_3(\omega)|^4   ,  \quad t\geq T_{\mathfrak B}(\omega).
\end{align}
By $u$ is a   bounded  in $H^{\frac52}$ and   \eqref{mar8.3}, we can get that for any $t\geq T_{\mathfrak B}(\omega)$
\begin{align}
\int_{t-1}^t\|A^{\frac{5}{4}}w(s,\theta_{-t}\omega,w(0))\|^2 \d s
&\leq\int_{t-1}^t\|A^{\frac{5}{4}}v(s,\theta_{-t}\omega,v(0))\|^2 \d s+\int_{t-1}^t\|A^{\frac{5}{4}}u(s)\|^2 \d s\nonumber\\
&\leq  \int_{t-1}^t \|A^{\frac{5}{4}}v(s,\theta_{-t}\omega,v(0))\|^2 \d s  +  C \nonumber\\
&\leq \frac 1  \nu  \zeta_3(\omega) +C .  \label{mar8.7}
\end{align}
Hence, inserting \eqref{4.23} and \eqref{mar8.7} into \eqref{4.22} yields
\begin{align}
\|A^{\frac{5}{4}}w(t,\theta_{-t}\omega,w(0))\|^2
&\leq C e^{C|\zeta_3(\omega) |^4 + C\int^0_{-1} |z(\theta_{\eta}\omega) |^8
 \d \eta} \left(   \zeta_3(\omega) +1
\right)  \nonumber \\
&=:\zeta_5(\omega),  \quad t\geq T_{\mathfrak B}(\omega).  \nonumber
\end{align}
This completes the proof of Lemma \ref{lem:H2bd}.
\end{proof}

\subsection{$H^{\frac{5}{2}}$ random absorbing sets}
In this subsection, we will prove an  $H^\frac52$ random absorbing set of system \eqref{2.2}. By the above Lemma \ref{lem:H2bd}, we prove an $H^\frac52$ random absorbing set for the  random fractional three-dimensional NS equations \eqref{2.2}.  Then we  next introduce the following  stronger estimate than Lemma \ref{lem:H2bd}.
In order to get the local $(H,H^\frac52)$-Lipschitz continuity of the random fractional three-dimensional NS equations \eqref{2.2} in this Section \ref{sec5}, we need to prove the following main results as Lemma \ref{lemma4.4} and Theorem \ref{theorem4.1}.
\begin{lemma}\label{lemma4.4}
 Let Assumption \ref{assum} hold and $f\in H$. There exists a tempered random variable $\rho(\cdot)$ such that  the solutions $v$ of the random fractional three-dimensional NS equations \eqref{2.2} and the solutions $u$ of the deterministic fractional three-dimensional equations \eqref{2} satisfy
\[
  \sup_{\varepsilon\in [0,1]} \|A^{\frac{5}{4}}v(t,\theta_{-t-\varepsilon}\omega,  v(0))- A^{\frac{5}{4}}u(t,u_0)\|^2     \leq  \rho(\omega) ,
\]
for any $t\geq T_{\mathfrak B} (\omega)$, $v(0)\in \mathfrak B(\theta_{-t-\varepsilon} \omega)$ and $u(0)\in \mathcal A_0$, here $T_{\mathfrak B} (\cdot)$ is given  in \eqref{timeB}.
\end{lemma}

\begin{proof}
By \eqref{sep6.3}, we have
\begin{align}\label{4.15}
 \frac{\d}{\d t}\|A^\frac54w\|^2  &\leq C\left (1+\|A^\frac58 v\|^6 +|z(\theta_t\omega)|^6 \right )
\|A^\frac54w\|^2  \nonumber \\
& +C \left (1+\|A^\frac58 v\|^8 +|z(\theta_t\omega)|^8 \right) \nonumber \\
&\leq     C\left (1+\|A^\frac58 v\|^8 +|z(\theta_t\omega)|^8 \right )\left( \|A^\frac54w\|^2+1\right)   .
\end{align}
For any $t \geq 1$, we integrate \eqref{4.15} from $(s,t)$ for $s\in (t-1,t)$
\begin{align}
 \|A^\frac54w(t)\|^2
& \leq    \|A^\frac54w(s)\|^2\nonumber \\
&+C\int_{s}^t   \left (1 +\|A^\frac58v(\tau)\|^8+|z(\theta_\tau\omega)|^8 \right ) \left( \|A^\frac54w(\tau)\|^2 +1\right)\d \tau.
\end{align}
Integrating $s$ over $(t-1,t)$,  it yields
\begin{align}
 \|A^\frac54w(t)\|^2
& \leq
C\int_{t-1}^t   \left (1 +\|A^\frac58v(s)\|^8+|z(\theta_s\omega)|^8 \right ) \left( \|A^\frac54w(s)\|^2 +1\right)\d s.
\end{align}

For any $\varepsilon\in [0,1]$, replacing $\omega$ with $\theta_{-t-\varepsilon}\omega$, it can get
\begin{align*}
& \|A^\frac54w(t,\theta_{-t-\varepsilon}\omega, w(0))\|^2   \\
&\quad  \leq   C  \int_{t-1}^t   \left (1 +\|A^\frac58v(s,\theta_{-t-\varepsilon}\omega, v(0))\|^8+|z(\theta_{s-t-\varepsilon}\omega)|^8 \right ) \left( \|A^\frac54 w(s)\|^2 +1\right) \d s   \\
&\quad  \leq   C  \int_{t+\varepsilon-2}^{t+\varepsilon}   \left (1 +\|A^\frac58v(s,\theta_{-t-\varepsilon}\omega, v(0))\|^8+|z(\theta_{s-t-\varepsilon}\omega)|^8 \right ) \left( \|A^\frac54 w(s)\|^2 +1\right) \d s .
\end{align*}
By  \eqref{mar8.3},  it yields
\begin{align*}
\sup_{\eta\in (t-2,t)}
 \|A^{\frac58}v (\eta ,\theta_{-t } \omega, v(0)) \|^8
&\leq  |\zeta_3(\omega)|^4 ,\quad t\geq T_{\mathfrak B} (\omega) .
\end{align*}
Let
\begin{align*}
 \zeta_6(\omega):=   |\zeta_3(\omega)|^4 +\sup_{s\in(-2,0)} |z(\theta_s\omega)|^8+1,
\end{align*}
here, $\zeta_6(\omega)$ is a tempered random variable  satisfy
\begin{align*}
  \|A^\frac54w(t,\theta_{-t-\varepsilon}\omega, w(0))\|^2     \leq   C \zeta_6(\omega)  \int_{t+\varepsilon-2}^{t+\varepsilon}   \big( \|A^\frac54w(s,\theta_{-t-\varepsilon } \omega, w(0))\|^2 +1\big) \d s
\end{align*}
for any $t\geq T_{\mathfrak B} (\omega)$ uniformly for $\varepsilon\in [0, 1]$.
By virtue of  \eqref{mar8.7}, it yields
\begin{align*}
 \int_{t+\varepsilon-2}^{t+\varepsilon}   \|A^{\frac54}w(s,\theta_{-t-\varepsilon } \omega, w(0))\|^2  \d s \leq \frac 1 \nu  \zeta_3(\omega) +C
\end{align*}
for any $\varepsilon\in [0,1]$ and $t\geq  T_{\mathfrak B} (\omega) $. Hence it can deduce that for any $ t\geq T_{\mathfrak B} (\omega)$
\begin{align} \label{rho}
\sup_{\varepsilon\in [0,1]}
  \|A^\frac54w(t,\theta_{-t-\varepsilon}\omega, w(0))\|^2     \leq   C \zeta_6(\omega)\left(\zeta_3(\omega) +1\right)=: \rho(\omega).
\end{align}
This completes the proof of Lemma \ref{lemma4.4}.
\end{proof}

When $\varepsilon =0$, we can get an $H^\frac52$ random absorbing set.
\begin{theorem}[$H^\frac52$ absorbing set]
\label{theorem4.1}
 Let Assumption \ref{assum} hold and $f\in H$. The  RDS $\phi$ generated by the random fractional three-dimensional NS equations \eqref{2.2} has a  random absorbing set $\mathfrak{B}_{H^\frac52}$, given as a random $H^\frac52$ neighborhood of the global attractor $\mathcal{A}_0$ of the deterministic fractional three-dimensional  NS equations:
\begin{align}
\mathfrak{B}_{H^\frac52}(\omega)=\left \{v\in H^\frac52: \, {\rm dist}_{H^\frac52}(v,\mathcal{A}_{0}) \leq \sqrt{\rho(\omega)} \, \right\}, \quad \omega\in\Omega, \nonumber
\end{align}
where $\rho(\cdot)$ is the tempered random variable defined by \eqref{rho}.  Moreover, the $\mathcal{D}_H$-random attractor $\mathcal{A}$ of the system \eqref{2.2} is a bounded and tempered random set in $H^\frac52$.

\end{theorem}
\begin{proof}
Inspired by the  Lemma \ref{lemma4.4} and the Theorem 13 in \cite{cui}, this completes the proof of Theorem  \ref{theorem4.1}.
\end{proof}

\section{Local $(H,H^\frac52$)-Lipschitz continuity} \label{sec5}

In this section, we will prove a local $(H, H^\frac52)$-Lipschitz continuity in initial values of the solutions of the random fractional three-dimensional NS  equations \eqref{2.2}.   This will be done step-by-step by   showing the local $(H,H)$-Lipschitz continuity, the local $(H,H^\frac54)$-Lipschitz continuity and finally the local $(H,H^\frac52)$-Lipschitz continuity. Let
 \begin{align*}
 \bar v(t,\omega, \bar v(0)):= v_1(t,\omega, v_1(0))-v_2(t,\omega, v_2(0))
 \end{align*}
be  the difference between two solutions of the system \eqref{2.2}.

\subsection{Local $(H,H)$-Lipschitz continuity}
In this subsection, we will prove a local $(H,H)$-Lipschitz continuity of the solutions of the random fractional three-dimensional NS  equations \eqref{2.2}.

\begin{lemma}\label{lemma51}[Local $(H,H)$-Lipschitz]
 Let Assumption \ref{assum} hold and $f\in H$.
For  any tempered set $\mathfrak D\in \D_H $, then there exist random variables $t_{\mathfrak D}(\cdot)$ and $L_1(\mathfrak D,\cdot) $ such that
 two solutions $v_1$ and $v_2$ of random fractional three-dimensional NS equations \eqref{2.2} corresponding to initial values  $v_{1,0},$ $v_{2,0}$ in $ \mathfrak D( \theta_{-t_{\mathfrak D}(\omega)}\omega),$   respectively,  satisfy
\ben
 & \left \| v_1 \! \left ( t_{\mathfrak D} (\omega) ,\theta_{-t_{\mathfrak D}(\omega)}\omega,  v_{1,0}\right)
 - v_2 \! \left ( t_{\mathfrak D} (\omega) ,\theta_{-t_{\mathfrak D}(\omega)}\omega, v_{2,0}\right) \right \|^2  \\[0.8ex]
 &\quad
 \leq L_1({\mathfrak D}, \omega) \|{v}_{1,0}-v_{2,0}\|^2  ,\quad \omega \in \Omega.
\ee

\end{lemma}
\begin{proof}
Assume that $v_1$ and $v_2$ are two solutions of the system \eqref{2.2} with the initial values $v_{1,0}$ and $v_{2,0}$, respectively. Let $\bar{v}=v_1-v_2$, we introduce the following system
\begin{align}\label{5.1}
\frac{\d\bar{v}}{\d t}+  \nu  A^{\frac54}\bar{v}+B( v_1+hz(\theta_t\omega))-B( v_2+hz(\theta_t\omega))=0.
\end{align}
Taking the inner product of the  system  \eqref{5.1} in $H$ by $\bar{v}$ and  integration by parts,  we can deduce that
\ben
\frac{1}{2}\frac{\d}{\d t}\|\bar{v}\|^2+  \nu \|A^{\frac58}\bar{v}\|^2&=- \big (B(\bar{v},v_1+hz(\theta_t\omega)),\bar{v} \big )
-\big (B(v_2+hz(\theta_t\omega),\bar{v}),\bar{v}\big )\nonumber\\
&=-\big (B(\bar{v},v_1+hz(\theta_t\omega)),\bar{v} \big )\nonumber\\
&\leq \|\bar{v}\|_{L^{12}}\|\nabla(v_1+hz(\theta_t\omega))\|_{L^{\frac{12}{5}}}\|\bar{v}\|\nonumber\\
&\leq C\|\bar{v}\|\|A^{\frac{5}{8}}\bar{v}\|\|A^{\frac{5}{8}}(v_1+hz(\theta_t\omega))\|\nonumber\\
&\leq \frac{\nu}{2}\|A^{\frac{5}{8}}\bar{v}\|^2+C\left (\|A^\frac58v_1\|^2+\|A^\frac58h\|^2|z(\theta_t\omega)|^2 \right)\|\bar{v}\|^2\nonumber\\
&\leq \frac\nu2\|A^{\frac58}\bar{v}\|^2+C \left(\|A^\frac58v_1\|^2+|z(\theta_t\omega)|^2 \right )\|\bar{v}\|^2.
\ee
Moreover, it yields
\[
\frac{\d}{\d t}\|\bar{v}\|^2+  \nu \|A^{\frac58}\bar{v}\|^2
\leq C\left (\|A^\frac58v_1\|^2+|z(\theta_t\omega)|^2 \right) \|\bar{v}\|^2.
\]
By Gronwall's lemma, we have for any $t>0$
\begin{align} \label{mar9.1}
 & \|\bar{v}(t)\|^2+  \nu \int_0^te^{C\int_s^t(\|A^\frac58v_1(\tau) \|^2+|z(\theta_\tau\omega)|^2)\, \d \tau}\|A^{\frac58}\bar{v} (s) \|^2 \d s \nonumber \\
&
\leq e^{C\int_0^t(\|A^\frac58v_1(\tau)\|^2+|z(\theta_\tau\omega)|^2) \d \tau}\|\bar{v}(0)\|^2.
\end{align}

By \eqref{4.29}, it can get for any $t\geq T_1(\omega)$,
\ben
   \int_0^t \|A^\frac58v_1(\tau,\theta_{-t}\omega, v_{1,0})\|^2  \d \tau &\leq e^{ \left(\frac 8 \alpha -3  \right)\lambda t}  \int_0^t e^{ \left( \frac 8 \alpha -3  \right) \lambda (\tau-t)} \|A^\frac58v_1 (\tau,\theta_{-t}\omega, v_{1,0})\|^2   \d \tau  \\
 &   \leq \frac {2}{\nu\alpha} e^{ \left(\frac 8 \alpha -3  \right)\lambda t}  \left( e^{-\lambda t} \|v_{1,0}\|^2 +\zeta_1(\omega) \right) .
 \ee
Hence, by \eqref{mar9.1}, we can deduce for any $t\geq T_1(\omega)$
\ben
 \|\bar v(t,\theta_{-t}\omega, \bar v(0))\|^2
 & \leq  e^{C \int_0^t\left (\|A^\frac58v_1(\tau,\theta_{-t}\omega, v_{1,0})\|^2+|z(\theta_{\tau-t} \omega)|^2 \right)   \d \tau }
 \|\bar{v}(0)\|^2 \\
 &\leq e^{C  e^{ \left(\frac 8 \alpha -3  \right)\lambda t}  \left( e^{-\lambda t} \|v_{1,0}\|^2 +\zeta_1(\omega) \right) +C\int_{-t}^0|z(\theta_{\tau} \omega)|^2  \d \tau } \|\bar{v}(0)\|^2.
\ee
Since $v_{1,0} \in \mathfrak D(\theta_{-t}\omega) $ which is tempered, there exists a random variable $ t_{\mathfrak D} (\omega)  \geq T_1(\omega)$  such that
\[
 e^{-\lambda t_{\mathfrak D}(\omega) } \|v_{1,0}\|^2 \leq  e^{-\lambda t_{\mathfrak D}(\omega) } \left \|  \, \mathfrak D \!  \left (\theta_{-t_{\mathfrak D}(\omega)} \omega \right ) \right\|^2 \leq 1,\quad \omega\in \Omega.
\]
Since $\mathfrak D$ is pullback absorbed by the absorbing set $\mathfrak B$,  then there exists a large enough $t_{\mathfrak D}(\omega)$  such that
\be \label{mar19.2}
 \phi \left ( t_{\mathfrak D}(\omega) , \theta_{-t_{\mathfrak D}(\omega)}\omega, \mathfrak D(\theta_{- t_{\mathfrak D}(\omega) }\omega) \right ) \subset \mathfrak B(\omega), \quad \omega\in \Omega .
\ee
Let a random variable be defined by
\[
L_1({\mathfrak D}, \omega)=e^{C  e^{ \left(\frac 8 \alpha -3  \right)\lambda  t_{\mathfrak D}(\omega)}    ( 1+\zeta_1(\omega)  ) + C\int_{-t_{\mathfrak D}(\omega)}^0|z(\theta_{\tau} \omega)|^2 \, \d \tau } ,\quad \omega\in \Omega,
\]
then
\be \label{mar18.8}
\left \|\bar v  \! \left ( t_{\mathfrak D} (\omega) ,\theta_{-t_{\mathfrak D}(\omega)}\omega, \bar v(0) \right) \right \|^2
 \leq L_1({\mathfrak D}, \omega) \|\bar{v}(0)\|^2  ,\quad
\ee
where $v_{1,0},$ $v_{2,0}\in \mathfrak D \big ( \theta_{-t_{\mathfrak D}( \omega)}\omega \big),$  $\omega\in \Omega$. This completes the proof of Lemma \ref{lemma51}.
\end{proof}

\subsection{Local $(H,H^\frac54)$-Lipschitz continuity}

In this subsection, we will prove a local $(H,H^\frac54)$-Lipschitz continuity of the solutions of the random fractional three-dimensional NS  equations \eqref{2.2}. The main result will be obtained  by two steps. By   \eqref{4.19}, we can get an $H^\frac54$ random absorbing set $\mathfrak B$ with initial values. Moreover, we will prove the  initial values  in each tempered set $\mathfrak D $ in $\D_H$.

\begin{lemma}\label{5.10}[Local $(H,H^\frac54)$-Lipschitz on $\mathfrak B$]
 Let Assumption \ref{assum} hold and $f\in H$. Then  there exist random variables $\tau_\omega$ and $L_2(\mathfrak D,\omega) $ such that any
 two solutions $v_1$ and $v_2$ of the random fractional three-dimensional NS  equations \eqref{2.2} corresponding to initial values  $v_{1,0},$ $v_{2,0}$ in $ \mathfrak B( \theta_{- \tau_\omega}\omega),$   respectively,   satisfy
\be  \label{mar18.9}
 \left \|A^{\frac58}(v_1  (  \tau_\omega,\theta_{-\tau_\omega}\omega,  v_{1,0} )
 - v_2   ( \tau_\omega,\theta_{-\tau_\omega}\omega, v_{2,0} ) )\right \|^2
 \leq L_2(\mathfrak D, \omega) \|{v}_{1,0}-v_{2,0}\|^2  ,\quad \omega \in \Omega,
\ee
where $ \mathfrak B $ is defined by \eqref{4.19}.
\end{lemma}

\begin{proof}
Taking the inner product of the system  \eqref{5.1} in $H$ by $A^{\frac54}\bar{v}$ and  by integration by parts, we can deduce that
\begin{align}\label{5.2}
\frac{1}{2}\frac{\d}{\d t}\|A^\frac58\bar{v}\|^2+ \nu \|A^\frac54\bar{v}\|^2&=- \big (B(\bar{v},v_1+hz(\theta_t\omega)),A^\frac54\bar{v}\big )\nonumber\\
&-\big (B(v_2+hz(\theta_t\omega),\bar{v}),A^\frac54\bar{v}\big ) .
\end{align}
For the first term on the right hand side of the system \eqref{5.2}, applying the H\"{o}lder inequality, the Sobolev embedding inequality and the Young inequality, we have
\begin{align}\label{5.3}
 \big| \big (B(\bar{v},v_1+hz(\theta_t\omega)),A^\frac54\bar{v}\big )\big|
 &\leq \|\bar{v}\|_{L^{12}}\|\nabla(v_1+hz(\theta_t\omega))\|_{L^{\frac{12}{5}}}\|A^\frac54\bar{v}\|\nonumber\\
&\leq C\|A^\frac58\bar{v}\|\|A^\frac54\bar{v}\|\|A^\frac58(v_1+hz(\theta_t\omega))\|\nonumber\\
&\leq \frac{\nu}{4}\|A^\frac54\bar{v}\|^2+C\|A^\frac58(v_1+hz(\theta_t\omega))\|^2\|A^\frac58\bar{v}\|^2\nonumber\\
&\leq \frac{\nu}{4}\|A^\frac54\bar{v}\|^2+C(\|A^\frac58v_1\|^2+\|A^\frac58h\|^2|z(\theta_t\omega)|^2)\|A^\frac58\bar{v}\|^2\nonumber\\
&\leq \frac{\nu}{4}\|A^\frac54\bar{v}\|^2+C(\|A^\frac58v_1\|^2+|z(\theta_t\omega)|^2)\|A^\frac58\bar{v}\|^2
.
\end{align}
For the second term on the right hand side of the system \eqref{5.2}, by the similar method, we also have
\begin{align}\label{5.4}
 \big| \big (B(v_2+hz(\theta_t\omega),\bar{v}),A^\frac54\bar{v}\big )\big| &\leq \|\nabla\bar{v}\|_{L^{\frac{12}{5}}}\|v_2+hz(\theta_t\omega)\|_{L^{12}}\|A^\frac54\bar{v}\|\nonumber\\
&\leq C\|A^\frac58\bar{v}\|\|A^\frac58(v_2+hz(\theta_t\omega))\|\|A^\frac54\bar{v}\|\nonumber\\
&\leq \frac{\nu}{4}\|A^\frac54\bar{v}\|^2+C\|A^\frac58(v_2+hz(\theta_t\omega))\|^2\|A^\frac58\bar{v}\|^2\nonumber\\
&\leq \frac{\nu}{4}\|A^\frac54\bar{v}\|^2+C(\|A^\frac58v_2\|^2+\|A^\frac58h\|^2|z(\theta_t\omega)|^2)\|A^\frac58\bar{v}\|^2\nonumber\\
&\leq\frac{\nu}{4}\|A^\frac54\bar{v}\|^2+C(\|A^\frac58v_2\|^2+|z(\theta_t\omega)|^2)\|A^\frac58\bar{v}\|^2.
\end{align}
Inserting \eqref{5.3} and \eqref{5.4} into \eqref{5.2} yields
\begin{align} \label{mar9.2}
\frac{\d}{\d t}\|A^\frac58\bar{v}\|^2+  \nu \|A^\frac54\bar{v}\|^2
\leq C\left(\|A^\frac58v_1\|^2+\|A^\frac58v_2\|^2+|z(\theta_t\omega)|^2\right)\|A^\frac58\bar{v}\|^2.
\end{align}
Applying Gronwall's lemma to \eqref{mar9.2}, we have for $s\in(t-1, t-\frac 12)$, $t\geq 1$,
\begin{align*}
 & \|A^\frac58\bar{v}(t)\|^2
 +  \nu  \int_s^te^{C\int_\eta^t(\|A^\frac58v_1(\tau)\|^2+\|A^\frac58v_2(\tau)\|^2+|z(\theta_\tau\omega)|^2) \d \tau}\|A^\frac54\bar{v}(\eta) \|^2  \d \eta
\nonumber\\
&\quad \leq e^{ C\int_s^t(\|A^\frac58v_1(\tau) \|^2+\|A^\frac58v_2(\tau)\|^2+|z(\theta_\tau\omega)|^2) \d \tau}\|A^\frac58\bar{v}(s)\|^2, \nonumber
\end{align*}
and  integrating  $s$ over $(t-1,t-\frac 12)$  yields
\begin{align*}
 & \|A^\frac58\bar{v}(t)\|^2
 + \nu \int_{t-\frac 12} ^te^{C\int_\eta^t(\|A^\frac58v_1(\tau)\|^2+\|A^\frac58v_2(\tau)\|^2+|z(\theta_\tau\omega)|^2) \d \tau}\|A^\frac54\bar{v}(\eta) \|^2 \d \eta
\nonumber\\
&\quad \leq  2  \int_{t-1} ^{t-\frac 12}   e^{ C\int_s^t(\|A^\frac58v_1(\tau) \|^2+\|A^\frac58v_2(\tau)\|^2+|z(\theta_\tau\omega)|^2) \d \tau}\|A^\frac58\bar{v}(s)\|^2   \d s \\
& \quad \leq  2e^{ C\int_{t-1}^t(\|A^\frac58v_1(\tau) \|^2+\|A^\frac58v_2(\tau)\|^2+|z(\theta_\tau\omega)|^2)\d \tau}  \int_{t-1} ^{t-\frac 12} \|A^\frac58\bar{v}(s)\|^2 \d s
.
\end{align*}
By \eqref{mar9.1},  it deduces for $t\geq 1$
\begin{align}
 & \|A^\frac58\bar{v}(t)\|^2
 + \nu\int_{t-\frac 12} ^te^{C\int_\eta^t(\|A^\frac58v_1(\tau)\|^2+\|A^\frac58v_2(\tau)\|^2+|z(\theta_\tau\omega)|^2) \d \tau}\|A^\frac54\bar{v}(\eta) \|^2  \d \eta \nonumber  \\
  &\quad  \leq  Ce^{ C\int_{0}^t(\|A^\frac58v_1(\tau) \|^2+\|A^\frac58v_2(\tau)\|^2+|z(\theta_\tau\omega)|^2) \d \tau}
   \|\bar{v}(0)\|^2
. \label{mar19.1}
\end{align}

Since the absorbing set $\mathfrak B$ itself belongs to the attraction universe $\D_H$,  it pullback absorbs itself. Then  there exists a random variable $\tau_\omega \geq 1+t_{\mathfrak D}(\omega)$ such that
\begin{align}\label{mar18.2}
 \phi \big(\tau_\omega,\theta_{-\tau_\omega}\omega, \B(\theta_{-\tau_\omega}\omega)\big)
 \subset \B(\omega), \quad \omega\in \Omega,
\end{align}
here, $t_{\mathfrak D}(\omega)$ is defined by Lemma \ref{lemma51}. When $t=\tau_\omega$, we  replace $\omega$ with $\theta_{-\tau_\omega} \omega$ in \eqref{mar19.1}  deduce that
\begin{align}\label{mar18.1}
 & \big \|
 A^{\frac 58}  \bar {v}(\tau_\omega,\theta_{-\tau_\omega}\omega,  \bar v(0))  \big \| {^2}
 \nonumber \\
  &\quad  \leq  Ce^{ C\int_{0}^{\tau_\omega}  \sum_{i=1}^2\|A^\frac58v_i(s, \theta_{-\tau_\omega}\omega, v_i(0)) \|^2 \d s+ C\int^0_{-\tau_\omega} |z(\theta_s\omega)|^2 \d s}
   \|\bar v(0) \|^2 .
\end{align}
By \eqref{4.29}, it yields for any $\tau_\omega\geq T_1(\omega)$,
\begin{align}\label{5.8}
&\int_0^{\tau_\omega} \sum_{i=1}^2\|A^\frac58v_i(\tau,\theta_{-\tau_\omega}\omega, v_{i,0})\|^2  \d \tau \nonumber\\
&\leq \sum_{i=1}^2e^{ \left(\frac 8 \alpha -3  \right)\lambda \tau_\omega}  \int_0^{\tau_\omega} e^{ \left( \frac 8 \alpha -3  \right) \lambda (\tau-\tau_\omega)} \|A^\frac58v_i (\tau,\theta_{-\tau_\omega}\omega, v_{i,0})\|^2   \d \tau  \nonumber\\
 &   \leq \sum_{i=1}^2\frac {2}{\alpha\nu} e^{ \left(\frac 8 \alpha -3  \right)\lambda \tau_\omega}  \left( e^{-\lambda \tau_\omega} \|v_{i,0}\|^2 +2\zeta_1(\omega) \right).
\end{align}
Inserting \eqref{5.8} into \eqref{mar18.1}, then we also get
\begin{align}
 & \big \|
 A^{\frac 58}  \bar {v}(\tau_\omega,\theta_{-\tau_\omega}\omega,  \bar v(0))  \big \| {^2}
 \nonumber \\
  &  \leq  Ce^{ Ce^{ \left(\frac 8 \alpha -3  \right)\lambda \tau_\omega}  \left( e^{-\lambda \tau_\omega} (\|v_{1,0}\|^2+\|v_{2,0}\|^2) +2\zeta_1(\omega) \right)+ C\int^0_{-\tau_\omega} |z(\theta_s\omega)|^2 \d s}
   \|\bar v(0) \|^2 .
\end{align}
Since $v_{i,0} \in \mathfrak D(\theta_{-\tau_\omega}\omega) $ which is tempered, then there exists a random variable $ \tau_{\omega} (\omega)  \geq t_{\mathfrak D}(\omega)$  such that
\begin{align*}
 e^{-\lambda \tau_{\omega}(\omega) } (\|v_{1,0}\|^2+\|v_{2,0}\|^2) \leq  e^{-\lambda \tau_{\omega}(\omega) } 2\left \|   \mathfrak D \!  \left (\theta_{-\tau_{\omega}(\omega)} \omega \right ) \right\|^2 \leq 2,\quad \omega\in \Omega.
\end{align*}
Then, we get
\begin{align}
 & \big \|
 A^{\frac 58}  \bar {v}(\tau_\omega,\theta_{-\tau_\omega}\omega,  \bar v(0))  \big \| {^2}
 \nonumber \\
  &  \leq  Ce^{ Ce^{ C\tau_\omega}  \left( 1 +\zeta_1(\omega) \right)+ C\int^0_{-\tau_\omega} |z(\theta_s\omega)|^2 \d s}
   \|\bar v(0) \|^2 .
\end{align}
Let the random variable
\begin{align*}
  L_2(\mathfrak D, \omega):=  Ce^{ C e^{C\tau_\omega}  \zeta_1( \omega) }
,\quad \omega\in \Omega,
\end{align*}
This completes the proof of Lemma \ref{5.10}.
\end{proof}

Nextly, we will prove  the local $(H,H^{\frac54})$-Lipschitz continuity for initial values from any tempered set $\mathfrak D$ in $H$.
\begin{lemma}[Local $(H,H^{\frac54})$-Lipschitz]  \label{lem:H1}
 Let Assumption \ref{assum} hold and $f\in H$. For   any tempered set   $\mathfrak D\in \D_H$, then there exist random variables $ T_1(\mathfrak D ,\cdot )  $ and $L_3(\mathfrak D,\cdot) $ such that any
 two solutions $v_1$ and $v_2$ of random fractional three-dimensional NS equations \eqref{2.2} corresponding to initial values  $v_{1,0},$ $v_{2,0}$ in $ \mathfrak D( \theta_{-T_1(\mathfrak D,\omega)}\omega),$   respectively,  satisfy
\begin{align}\label{mar19.4}
  \left \| A^\frac58( v_1  (T_1 ,\theta_{-T_1}\omega,  v_{1,0} )
 - v_2   ( T_1 ,\theta_{-T_1}\omega, v_{2,0} ) )\right \|^2
 \leq L_3({\mathfrak D}, \omega) \|{v}_{1,0}-v_{2,0}\|^2  ,
\end{align}
where $T_1= T_1(\mathfrak D, \omega)$, $\omega\in \Omega$. Moreover, $ v (T_1 ,\theta_{-T_1}\omega,  v_{0}) \in \mathfrak B(\omega)$ for any $v_0\in \mathfrak D(\theta_{-T_1}\omega) $.
\end{lemma}

\begin{proof}
Since ${\mathfrak D} \in \D_H$   is pullback absorbed by $\mathfrak B$, by \eqref{mar19.2}, then  there exists  a random variable $  t_{\mathfrak D} (\cdot) $ such that
\begin{align}\label{mar19.5}
 \phi \left (  t_{\mathfrak D}(\omega) , \theta_{-  t_{\mathfrak D}(\omega)}\omega, \mathfrak D(\theta_{-  t_{\mathfrak D}(\omega)}\omega) \right) \subset \mathfrak B(\omega), \quad \omega\in \Omega.
\end{align}
Hence,    for any $ \omega\in \Omega $   we can deduce that for \eqref{mar18.8} and \eqref{mar18.9}
\begin{align*}
&   \left \| A^\frac58 \bar v  \!\left( \tau_\omega+   t_{\mathfrak D}(\theta_{-\tau_\omega} \omega)   , \, \theta_{-\tau_\omega  - t_{\mathfrak D}(\theta_{-\tau_\omega} \omega)}
 \omega, \, \bar v(0)\right)
   \right \|^2
    \nonumber \\
         &= \left \| A^\frac58\bar v  \! \left( \tau_\omega  , \, \theta_{-\tau_\omega   }
 \omega, \,  \bar v \big(   t_{\mathfrak D}(\theta_{-\tau_\omega} \omega) ,\,  \theta_{-\tau_\omega- t_{\mathfrak D}(\theta_{-\tau_\omega} \omega)  }
 \omega,  \, \bar v(0)\big) \right)
   \right \|^2 \nonumber
   \\
 &
   \leq  L_2  (\mathfrak D,   \omega  )   \left \|  \bar v \!  \left( t_{\mathfrak D}(\theta_{-\tau_\omega} \omega),\,  \theta_{-\tau_\omega- t_{\mathfrak D}(\theta_{-\tau_\omega} \omega)    }
 \omega, \,  \bar v(0) \right)
   \right \|^2   \nonumber \\
 & \leq     L_2  (\mathfrak D,   \omega  )   L_1  \big (\mathfrak D, \theta_{-\tau_\omega}\omega  \big)  \|\bar{v}(0)\|^2
 \nonumber
\end{align*}
for any $v_{1,0},$ $v_{2,0}\in \mathfrak D \left (  \theta_{-\tau_\omega  - t_{\mathfrak D}(\theta_{-\tau_\omega} \omega)} \omega \right)$.  Let
 \begin{align*}
 & T_1({\mathfrak D,\omega}) := \tau_\omega + t_{\mathfrak D}(\theta_{-\tau_\omega} \omega) ,  \\
 & L_3(\mathfrak D,\omega):=  L_2 \left (\mathfrak D,    \omega \right )   L_1 \left(\mathfrak D, \theta_{-\tau_\omega}\omega  \right) ,
\end{align*}
then we can get \eqref{mar19.4}.

By the definition \eqref{mar19.5}  of $ t_{\mathfrak D}(\theta_{-\tau_\omega} \omega)$, we can get for any $v_0\in \mathfrak D(\theta_{-T_1}\omega) $,
\begin{align*}
 y:= v \left(    t_{\mathfrak D}(\theta_{-\tau_\omega} \omega)   , \, \theta_{- t_{\mathfrak D}(\theta_{-\tau_\omega} \omega)} \circ \theta_{-\tau_\omega}
 \omega, \,  v(0)\right) \in \mathfrak B(\theta_{-\tau_\omega}\omega).
\end{align*}
Moreover, by the definition \eqref{mar18.2}  of $\tau_\omega$ we can deduce that
\[
 v(\tau_\omega, \theta_{-\tau_\omega} \omega, y) \in \mathfrak B(\omega).
\]
Hence, it yields for any $v_0\in \mathfrak D(\theta_{-T_1}\omega) $,
\begin{align*}
    v\left (T_1 ,\theta_{-T_1}\omega,  v_{0}\right)
    & =  v \left( \tau_\omega+   t_{\mathfrak D}(\theta_{-\tau_\omega} \omega)   , \, \theta_{- t_{\mathfrak D}(\theta_{-\tau_\omega} \omega)}(\theta_{-\tau_\omega}
 \omega), \,  v(0)\right)\\
 &= v(\tau_\omega, \theta_{-\tau_\omega} \omega, y)
  \in \mathfrak B(\omega).
\end{align*}
This completes the proof of Lemma \ref{lem:H1}.

 \end{proof}

\subsection{Local $(H , H^{\frac52})$-Lipschitz continuity}

In this subsection, we will prove a local $(H,H^\frac52)$-Lipschitz continuity of the solutions of the random fractional three-dimensional NS  equations \eqref{2.2}.  We first introduce the following useful estimate as  \eqref{mar9.4} and prove a local $(H^{\frac54}, H^{\frac52})$-Lipschitz on $\mathfrak B$. Finally, we will prove the main result of this section as Theorem \ref{theorem5.6}.

\begin{lemma}  Let Assumption \ref{assum} hold and $f\in H$,
  the solutions $v$ of random fractional three-dimensional NS equations \eqref{2.2} corresponding to initial values  $v(0)$  satisfy
\begin{align}
\sup_{s\in [t-1,t]}
 \|A^{\frac54}v(s,\theta_{-t}\omega,  v(0)) \|^2
  \leq     \rho(\omega) +\|\mathcal A_0\|_{H^\frac52}^2,
 \quad  t\geq T_{\mathfrak B} (\omega)+ 1,\label{mar9.4}
\end{align}
where  $v(0) \in\mathfrak B (\theta_{-t} \omega)$ and  $T_{\mathfrak B} $ is the random variable defined by   \eqref{timeB}.
\end{lemma}
\begin{proof}
By the above Lemma \ref{lemma4.3}, there exist  $ t\geq T_{\mathfrak B} (\omega)+1$ and $s\in (t-1,t)$ such that $s\geq T_{\mathfrak B} (\omega)$. Applying the above
 Lemma \ref{lemma4.4}, we can deduce that for $s\in [t-1,t]$ and $v(0) \in\mathfrak B (\theta_{-t} \omega)$
\begin{align*}
 \|A^{\frac54}v(s,\theta_{-t}\omega,  v(0)) \|^2    =
 \|A^{\frac54}v(s,\theta_{-s-(t-s)}\omega,  v(0)) \|^2
 \leq     \rho(\omega) +\|\mathcal A_0\|_{H^\frac52}^2.
\end{align*}
\end{proof}

\begin{lemma}[Local $(H^{\frac54}, H^{\frac52})$-Lipschitz on $\mathfrak B$]\label{lemma5.1}  Let Assumption \ref{assum} hold and $f\in H$,  there  exist  random variables $T_\omega$ and $ L_4(\mathfrak B, \omega )$ such that    two solutions $v_1$ and $v_2$ of random fractional three-dimensional NS equations \eqref{2.2} corresponding to initial values  $v_{1,0},$ $v_{2,0}$ in $\mathfrak{B}(\theta_{-T_\omega}\omega)$, respectively,   satisfy
\[
 \big \|A^{\frac54}(v_1(T_\omega,\theta_{-T_\omega}\omega,v_{1,0})-v_2(T_\omega,\theta_{-T_\omega}\omega,v_{2,0}) \big) \|{^2}\leq  L_4(\mathfrak B, \omega)\|A^{\frac58}(v_{1,0}-v_{2,0})\|^2, \quad \omega\in \Omega.
\]
\end{lemma}

\begin{proof}
Taking the inner product of the system  \eqref{5.1} in $H$ by $A^\frac52\bar{v}$ and integration by parts, we can deduce that
\begin{align}\label{5.5}
&\frac{1}{2}\frac{\d}{\d t}\|A^\frac54\bar{v}\|^2+ \nu  \|A^{\frac{15}{8}}\bar{v}\|^2\nonumber\\
&=- \big (B(v_1+hz(\theta_t\omega))-B(v_2+hz(\theta_t\omega)), A^\frac52\bar{v}  \big )\nonumber\\
&=- \big (B(\bar{v},v_1+hz(\theta_t\omega)), A^\frac52\bar{v}  \big )
- \big (B(v_2+hz(\theta_t\omega),\bar{v}), A^\frac52\bar{v} \big ).
\end{align}
For the first term on the right hand side of the system \eqref{5.5}, applying the integration by parts, the Sobolev embedding inequality, the H\"{o}lder inequality and the Young inequality, we can get
\begin{align}\label{5.6}
 \big| \big(B(\bar{v},v_1+hz(\theta_t\omega)), A^\frac52\bar{v}\big)
 \big| &\leq C\|A^\frac58\bar{v}\|_{L^{12}}\|\nabla(v_1+hz(\theta_t\omega))\|_{L^{\frac{12}{5}}}\|A^{\frac{15}{8}}\bar{v}\|  \nonumber\\
&
 +C\|\bar{v}\|_{L^{12}}\|A^\frac98(v_1+hz(\theta_t\omega))\|_{L^{\frac{12}{5}}}\|A^{\frac{15}{8}}\bar{v}\|
 \nonumber\\
&\leq C\|A^\frac54\bar{v}\|\|A^\frac58(v_1+hz(\theta_t\omega))\|\|A^{\frac{15}{8}}\bar{v}\|  \nonumber\\
&
 +C\|A^\frac58\bar{v}\|\|A^\frac54(v_1+hz(\theta_t\omega))\|\|A^{\frac{15}{8}}\bar{v}\|
 \nonumber\\
 &\leq \frac{ \nu}{4}\|A^{\frac{15}{8}}\bar{v}\|^2+C\|A^\frac58(v_1+hz(\theta_t\omega))\|^2\|A^\frac54\bar{v}\|^2\nonumber\\
 &+C\|A^\frac54(v_1+hz(\theta_t\omega))\|^2\|A^\frac58\bar{v}\|^2\nonumber\\
 &\leq \frac{ \nu}{4}\|A^{\frac{15}{8}}\bar{v}\|^2+C(\|A^\frac58v_1\|^2+\|A^\frac58h\|^2|z(\theta_t\omega)|^2)\|A^\frac54\bar{v}\|^2\nonumber\\
 &+C(\|A^\frac54v_1\|^2+\|A^\frac54h\|^2|z(\theta_t\omega)|^2)\|A^\frac58\bar{v}\|^2\nonumber\\
 &\leq \frac{\nu}{4}\|A^{\frac{15}{8}}\bar{v}\|^2+C(\|A^\frac58v_1\|^2+|z(\theta_t\omega)|^2)\|A^\frac54\bar{v}\|^2\nonumber\\
 &+C(\|A^\frac54v_1\|^2+|z(\theta_t\omega)|^2)\|A^\frac58\bar{v}\|^2.
\end{align}
For the second term on the right hand side of the system \eqref{5.5}, by the similar method, it can deduce that
\begin{align}\label{5.7}
\big| \big(B(v_2+hz(\theta_t\omega),\bar{v}),\, A^\frac52\bar{v}\big)\big|
&\leq C\|A^\frac58(v_2+hz(\theta_t\omega))\|_{L^{12}}\|\nabla\bar{v}\|_{L^\frac{12}{5}}\|A^\frac{15}{8}\bar{v}\| \nonumber\\
&
 +C\|v_2+hz(\theta_t\omega)\|_{L^{12}}\|A^\frac98\bar{v}\|_{L^{\frac{12}{5}}}\|A^\frac{15}{8}\bar{v}\|\nonumber\\
&\leq C\|A^\frac54(v_2+hz(\theta_t\omega))\|\|A^\frac58\bar{v}\|\|A^\frac{15}{8}\bar{v}\| \nonumber\\
&
 +C\|A^\frac58(v_2+hz(\theta_t\omega))\|\|A^\frac54\bar{v}\|\|A^\frac{15}{8}\bar{v}\|\nonumber\\
&\leq\frac{\nu}{4}\|A^{\frac{15}{8}}\bar{v}\|^2+C\|A^\frac54(v_2+hz(\theta_t\omega))\|^2\|A^\frac58\bar{v}\|^2\nonumber\\
&
 +C\|A^\frac58(v_2+hz(\theta_t\omega))\|^2\|A^\frac54\bar{v}\|^2\nonumber\\
&\leq\frac{\nu}{4}\|A^{\frac{15}{8}}\bar{v}\|^2+C(\|A^\frac58v_2\|^2+\|A^\frac58h\|^2|z(\theta_t\omega)|^2)
\|A^\frac54\bar{v}\|^2\nonumber\\
&
 +C(\|A^\frac54v_2\|^2+\|A^\frac54h\|^2|z(\theta_t\omega)|^2)\|A^\frac58\bar{v}\|^2\nonumber\\
&\leq\frac{\nu}{4}\|A^{\frac{15}{8}}\bar{v}\|^2+C(\|A^\frac58v_2\|^2+|z(\theta_t\omega)|^2)
\|A^\frac54\bar{v}\|^2\nonumber\\
&
 +C(\|A^\frac54v_2\|^2+|z(\theta_t\omega)|^2)\|A^\frac58\bar{v}\|^2.
\end{align}
Inserting \eqref{5.6} and \eqref{5.7} into \eqref{5.5}, it is easy to deduce that
\begin{align}
\frac{\d}{\d t}\|A^\frac54\bar{v}\|^2+\nu\|A^\frac{15}{8}\bar{v}\|^2
&\leq C\|A^\frac54\bar{v}\|^2 \left ( \sum_{j=1}^2 \|A^\frac58v_j\|^2 +|z(\theta_t\omega)|^2 \right ) \nonumber\\  &+C\|A^\frac58\bar{v}\|^2 \left (\sum_{j=1}^2\|A^\frac54v_j\|^2+|z(\theta_t\omega)|^2 \right), \nonumber
\end{align}
moreover,
\begin{align}
\frac{\d}{\d t}\|A^\frac54\bar{v}\|^2
\leq C\|A^\frac54\bar{v}\|^2 \left ( \sum_{j=1}^2 \|A^\frac58v_j\|^2 +|z(\theta_t\omega)|^2 \right ) +C\|A^\frac58\bar{v}\|^2 \left (\sum_{j=1}^2\|A^\frac54v_j\|^2+|z(\theta_t\omega)|^2 \right). \nonumber
\end{align}
For $s\in(t-1,t)$ and $t\geq 1$, by using Gronwall's lemma,  we can get
\begin{align}
&\|A^\frac54\bar{v}(t)\|^2
-e^{\int_s^tC( \sum_{i=1}^2 \|A^\frac58v_i\|^2 +|z(\theta_\tau\omega)|^2) \d \tau}\|A^\frac54\bar{v}(s)\|^2\nonumber\\
&\leq \int_s^tCe^{\int_\eta^tC( \sum_{i=1}^2 \|A^\frac58v_i\|^2+|z(\theta_\tau\omega)|^2) \d \tau}
\|A^\frac58\bar{v}\|^2 \left(\sum_{i=1}^2\|A^\frac54v_i\|^2+|z(\theta_\eta\omega)|^2\right) \d \eta\nonumber\\
& \leq Ce^{\int_{t-1}^tC( \sum_{i=1}^2 \|A^\frac58v_i\|^2+|z(\theta_\tau\omega)|^2) \d \tau}
\int_{t-1}^t\|A^\frac58\bar{v}\|^2\left( \sum_{i=1}^2\|A^\frac54v_i\|^2+|z(\theta_\eta\omega)|^2 \right) \d \eta. \nonumber
\end{align}
Then we integrate  $s$ on $(t-1,t)$ to deduce
\begin{align}\label{5.9}
&\|A^\frac54\bar{v}(t)\|^2
-\int_{t-1}^te^{\int_s^tC( \sum_{i=1}^2 \|A^\frac58v_i\|^2+|z(\theta_\tau\omega)|^2) \d \tau}\|A^\frac54\bar{v}(s)\|^2 \d s\nonumber\\
& \leq Ce^{\int_{t-1}^tC( \sum_{i=1}^2 \|A^\frac58v_i\|^2+|z(\theta_\tau\omega)|^2) \d \tau}
\int_{t-1}^t\|A^\frac58\bar{v}\|^2\left (\sum_{i=1}^2\|A^\frac54v_i\|^2+|z(\theta_\eta\omega)|^2\right) \d \eta.
\end{align}

By virtue of  the initial values $v_{1,0}$ and $v_{2,0}$   belong to the $H^\frac54$ random  absorbing set  $\mathfrak B$ and   using Gronwall's lemma to \eqref{mar9.2}, it is easy to deduce for $t\geq 1$
\begin{align}
 & \|A^\frac58\bar{v}(t)\|^2 + \nu \int_0^te^{\int_s^tC( \sum_{i=1}^2 \|A^\frac58v_i\|^2+|z(\theta_\tau\omega)|^2) \d \tau}\|A^\frac54\bar{v}(s) \|^2\d s
\nonumber\\
& \leq e^{\int_0^tC( \sum_{i=1}^2 \|A^\frac58v_i\|^2+|z(\theta_\tau\omega)|^2)\d \tau}\|A^\frac58\bar{v}(0)\|^2 . \nonumber
\end{align}
Then we can get the  following  estimates
\begin{align}\label{5.11}
& \nu \int_{t-1}^te^{\int_s^tC( \sum_{i=1}^2 \|A^\frac58v_i\|^2+|z(\theta_\tau\omega)|^2)\d \tau}\|A^\frac54\bar{v}(s)\|^2 \d s\nonumber\\
& \leq e^{\int_{0}^tC( \sum_{i=1}^2 \|A^\frac58v_i\|^2+|z(\theta_\tau \omega)|^2) \d \tau}\|A^\frac58\bar{v}(0)\|^2,
\end{align}
and
\begin{align}\label{5.12}
&
\int_{t-1}^t\|A^\frac58\bar{v}\|^2 \left (\sum_{i=1}^2\|A^\frac54v_i\|^2+|z(\theta_\eta\omega)|^2\right) \d \eta\nonumber\\
&\quad \leq e^{\int_0^tC( \sum_{i=1}^2 \|A^\frac58v_i\|^2+|z(\theta_\tau\omega)|^2) \d \tau}\|A^\frac58\bar{v}(0)\|^2
\int_{t-1}^t   \left (\sum_{i=1}^2\|A^\frac54v_i\|^2+|z(\theta_\eta\omega)|^2\right) \d \eta.
\end{align}
Inserting \eqref{5.11}-\eqref{5.12}  into \eqref{5.9}, we can deduce that for $  t\geq 1$
\begin{align}
 \|A^\frac54\bar{v}(t)\|^2
&\leq Ce^{\int_0^tC( \sum_{i=1}^2 \|A^\frac58v_i\|^2+|z(\theta_\tau\omega)|^2) \d \tau}\|A^\frac58\bar{v}(0)\|^2
\int_{t-1}^t  \left (1+\sum_{i=1}^2\|A^\frac54v_i\|^2+|z(\theta_\eta\omega)|^2\right)  \d \eta.   \nonumber
\end{align}
We replace $\omega$ by $\theta_{-t}\omega$ to deduce   for any $  t\geq 1$
\begin{align} \label{mar11.1}
 \|A^\frac54\bar{v}(t,\theta_{-t}\omega,\bar{v}(0))\|^2
&\leq Ce^{\int_0^tC( \sum_{i=1}^2\|A^\frac58v_i(\tau,\theta_{-t}\omega,v_i(0))\|^2
+|z(\theta_{\tau-t}\omega)|^2)\d \tau}\|A^\frac58\bar{v}(0)\|^2\nonumber\\
&  \times \int_{t-1}^t  \! \left ( \sum_{i=1}^2\|A^\frac54v_i(\eta,\theta_{-t}\omega,v_{i,0})\|^2+|z(\theta_{\eta-t}\omega)|^2+1\right) \d \eta.
\end{align}

By \eqref{mar9.4}, it yields for $i=1,2$
  \[
\sup_{s\in [t-1,t]}
 \|A^\frac54v_i(s,\theta_{-t}\omega,  v_i(0)) \|^2
  \leq     \rho(\omega) +\|\mathcal A_0\|_{H^\frac52}^2,
 \quad  t\geq T_{\mathfrak B} (\omega)+1.
\]
Let
\[
T_\omega :=T_{\mathfrak B} (\omega)+1, \quad \omega\in \Omega,
\]
it can deduce that
\begin{align}
 & \int_{T_\omega -1}^{T_\omega }  \left ( \sum_{i=1}^2\|A^\frac54v_i(\eta,\theta_{-T_\omega }\omega,v_{i,0})\|^2+|z(\theta_{\eta-T_\omega }\omega)|^2\right)  \d \eta  \nonumber\\
 &
 \leq  2\left( \rho(\omega) +\|\mathcal A_0\|_{H^\frac{5}{2}}^2 \right) + \int_{T_\omega -1}^{T_\omega } |z(\theta_{\eta-T_\omega }\omega)|^2 \d \eta\nonumber\\
 &
 =  2\left( \rho(\omega) +\|\mathcal A_0\|_{H^\frac{5}{2}}^2 \right) + \int_{-1}^0 |z(\theta_{\eta}\omega)|^2 \d \eta =: \zeta_7(\omega)
  . \label{mar11.2}
 \end{align}
By \eqref{4.29} and Lemma \ref{lem:H1bound}, it yields for $t\geq T_{B}(\omega)\geq T_{1}(\omega)$
\begin{align}
 & \int_{0}^{t }  \left ( \sum_{i=1}^2\|A^\frac58v_i(\eta,\theta_{-t }\omega,v_{i,0})\|^2+|z(\theta_{\eta-t }\omega)|^2\right)  \d \eta  \nonumber\\
 &\leq  e^{(\frac{8}{\alpha}-3)\lambda t}\int_{0}^{t } e^{(\frac{8}{\alpha}-3)\lambda(\eta- t)}  \sum_{i=1}^2\|A^\frac58v_i(\eta,\theta_{-t }\omega,v_{i,0})\|^2\d \eta+\int_{-t}^{0 }|z(\theta_{\eta }\omega)|^2 \d \eta  \nonumber\\
 &\leq \frac{2}{\alpha\nu}e^{(\frac{8}{\alpha}-3)\lambda t}\left(e^{-\lambda t}\sum_{i=1}^2\|v_{i,0}\|^2+2\zeta_1(\omega)\right)+e^{C\int_{-t}^{0 }|z(\theta_{\eta }\omega)|^2 \d \eta}\nonumber\\
 &\leq Ce^{Ct+C\int_{-t}^{0 }|z(\theta_{\eta }\omega)|^2 \d \eta}\left(1+e^{-\lambda t}\sum_{i=1}^2\|v_{i,0}\|^2+2\zeta_1(\omega)\right)\nonumber\\
 &\leq Ce^{Ct+C\int_{-t}^{0 }|z(\theta_{\eta }\omega)|^2 \d \eta}\left (1+\zeta_2(\theta_{-t }\omega)\right).
 \end{align}
Let $t=T_\omega$, we have
 \begin{align}  \label{mar11.3}
 & \int_0^{T_\omega}
  \left ( \sum_{i=1}^2\|A^\frac58v_i(\eta,\theta_{-T_\omega }\omega,v_{i,0})\|^2+|z(\theta_{\eta-T_\omega }\omega)|^2\right)  \d \eta  \nonumber \\
 &     \leq  C\big( \zeta_2(\theta_{-{T_\omega}}\omega)  +1 \big)  e^{ C {T_\omega}+C \int^0_{-{T_\omega}}  |z(\theta_{\tau} \omega)|^2 \d \tau} =:\zeta_8(\omega).
\end{align}

Inserting     \eqref{mar11.2}  and \eqref{mar11.3}  into \eqref{mar11.1},  we deduce at $t=T_\omega$
\begin{align}
 \|A^\frac54\bar{v}(T_\omega ,\theta_{-T_\omega }\omega,\bar{v}(0))\|^2
&\leq Ce^{ C \zeta_8(\omega) } \zeta_7(\omega) \|A^\frac58\bar{v}(0)\|^2. \nonumber
\end{align}
Let
\[
 L_4 (\mathfrak B, \omega) :=   Ce^{ C \zeta_8(\omega) } \zeta_7(\omega) ,\quad \omega\in \Omega,
\]
this completes the proof of Lemma \ref{lemma5.1}.
\end{proof}

Finally, we introduce the following main result of this section.
\begin{theorem}[Local $(H, H^\frac52)$-Lipschtiz] \label{theorem5.6}
 Let Assumption \ref{assum} hold and $f\in H$. For   any tempered set $\mathfrak D \in \D_H$,  then there
  exist  random variables $T_{{\mathfrak D}} (\cdot)  $   and $ L_{\mathfrak D}(\cdot )$ such that   two solutions $v_1$ and $v_2$ of random fractional three-dimensional NS equations \eqref{2.2} corresponding to initial values   $v_{1,0},$ $ v_{2,0}$ in $\mathfrak D \left (\theta_{-T_{\mathfrak D}(\omega)}\omega \right)$, respectively, satisfy
  \begin{align}\label{sep5.1}
  &
  \big\|A^\frac54(v_1 \!  \big (T_{\mathfrak D}(\omega),\theta_{-T_{\mathfrak D}(\omega)}\omega,v_{1,0}\big)
  -v_2 \big (T_{\mathfrak D}(\omega),\theta_{-T_{\mathfrak D}(\omega)}\omega,v_{2,0} \big) \big) \|^2 \nonumber\\
&\quad \leq  L_{\mathfrak D} (\omega)\|v_{1,0}-v_{2,0}\|^2,\quad \omega\in \Omega.
\end{align}
\end{theorem}
\begin{proof}
By the above Lemma \ref{lem:H1}, for any tempered set   $\mathfrak D\in \D_H$, then there exist random variables $ T_1(\mathfrak D ,\cdot )  $ and $L_3(\mathfrak D,\cdot) $ such that
\be   \label{mar19.6}
  \left \| A^\frac58\bar v\left (T_1 ,\theta_{-T_1}\omega,  \bar v(0)\right)
  \right \|^2
 \leq L_3({\mathfrak D}, \omega) \|\bar {v} (0)\|^2  ,
\ee
where $T_1= T_1(\mathfrak D, \omega)$. Moreover, $ v\left (T_1 ,\theta_{-T_1}\omega,  v_{0}\right) \in \mathfrak B(\omega)$ for any $v_0\in \mathfrak D(\theta_{-T_1}\omega) $, $\omega\in \Omega$.
 By the above Lemma \ref{lemma5.1} and \eqref{mar19.6}, it yields
\ben
  \left \| A^\frac54\bar v \big( T_\omega +T_1, \theta_{-T_\omega-T_1}
 \omega, \bar v(0) \big)
   \right \|^2
 &  =  \left \| A^\frac54\bar v \big( T_\omega  , \theta_{-T_\omega }
 \omega, \bar v  ( T_1, \theta_{-T_\omega-T_1}
 \omega, \bar v(0)  )\big)
   \right \|^2  \\
   &\leq  L_4(\mathfrak B, \omega) \left \| A^\frac58 \bar v \big  ( T_1, \theta_{-T_\omega-T_1}
 \omega, \bar v(0)  \big)
   \right \|^2  \  \text{(by Lemma \ref{lemma5.1})}\\
   &\leq L_4(\mathfrak B,\omega)  L_3 \big (\mathfrak D, \theta_{-T_\omega}\omega \big) \|\bar v(0)\|^2  \  \text{(by \eqref{mar19.6})} ,
\ee
where $T_1=T_1(\theta_{-T_\omega}  \omega)$,
uniformly for $v_{1,0},$ $ v_{2,0}\in \B(\theta_{-T_\omega-T_1(\theta_{-T_\omega} \omega)}\omega)$, $  \omega\in \Omega$. Let the random variables defined by
\ben
  & T_{\mathfrak D}( \omega) := T_\omega +T_1(\theta_{-T_\omega} \omega) , \\
  &L_{\mathfrak D}(\omega) :=  L_4(\mathfrak B,\omega)   L_3 \big (\mathfrak D, \theta_{-T_\omega}\omega \big),
\ee
this completes the proof of Lemma \ref{theorem5.6}.
  \end{proof}

\section{The $(H,H^\frac{5}{2})$-random attractor}\label{sec6}

In this subsection, we will prove a $(H,H^\frac{5}{2})$-random attractor of the random fractional three-dimensional NS  equations \eqref{2.2}. Moreover, the  random attractor  $\mathcal A$ of the random fractional three-dimensional NS  equations \eqref{2.2} is also a finite-dimensional  $(H,H^\frac{5}{2})$-random attractor. As in \ref{theorem4.6}, we introduce the following main result.

\begin{theorem} \label{theorem5.1}
 Let Assumption \ref{assum} hold and $f\in H$. The  RDS $\phi$ generated by the random fractional three-dimensional NS equations  \eqref{2.2} has  a tempered $(H,H^\frac52)$-random attractor $\mathcal A $. Moreover, $\mathcal A$ has a finite fractal dimension in $H^\frac52$: there exists    $ d>0$ such that
\[
d_f^{H^\frac52} \! \big ( \A(\omega) \big)  \leq d, \quad \omega \in \Omega.
\]
\end{theorem}
\begin{proof}
  By the above Theorem \ref{theorem4.1},  we can get an $H^\frac{5}{2}$ random absorbing set $\mathfrak B_{H^\frac{5}{2}}$, which is tempered and closed in $H^\frac{5}{2}$.  Then the  $(H,H^\frac{5}{2})$-smoothing property \eqref{sep5.1}   implies the $(H,H^\frac{5}{2})$-asymptotic compactness of $\phi$. Thence  by the above Lemma \ref{lem:cui18}, we show that $\A$ is an $(H,H^\frac{5}{2})$-random attractor of $\phi$.
 By the above Theorem \ref{theorem4.6}, we have  the fractal dimension in $H$ of  $\A$ is finite. By the Lemma 5 in \cite{cui} and the $(H,H^\frac{5}{2})$-smoothing property \eqref{sep5.1}, we prove the   finite-dimensionality  in $H^\frac{5}{2}$. This completes the proof of Theorem \ref{theorem5.1}.
\end{proof}

 \begin{remark}
 The local $(H,H^\frac{5}{2})$-Lipschitz continuity \eqref{sep5.1}  and the finite fractal dimension in $H^\frac{5}{2}$ of the global attractor    are new even for the deterministic fractional three-dimensional NS equations.
 \end{remark}

\section{The $(H,H^k)$-random attractor}\label{sec7}

In this section, let $f\in H^{k-\frac{5}{4}}$ and $k\geq\frac52$, we will prove  an  $H^\frac52$ random absorbing set and  an  $H^k$ random absorbing set of system \eqref{2.2}. Then we will prove a local $(H, H^k)$-Lipschitz continuity in initial values of the solutions of the random fractional three-dimensional NS  equations \eqref{2.2}. Finally, the $(H,H^k)$-random attractor  $\mathcal A$ of the random fractional three-dimensional NS  equations \eqref{2.2} is also a finite-dimensional  $(H,H^k)$-random attractor.

\subsection{Estimate in $H^{\frac{5}{2}}$}
In this subsection, let $f\in H^{\frac{5}{4}}$ and $h\in H^{\frac{15}{4}}$,   an  $H^\frac52$ random absorbing set  of system \eqref{2.2} are proved. Let $k>\frac{5}{2}$, in order to prove $H^k$ random absorbing set  of system \eqref{2.2}, we use the iterative methods and need satisfy the following Lemma \ref{lemma7.1}.
\begin{lemma}[$H^{\frac52}$ bound] \label{lemma7.1}
Assume that  $f\in H^\frac54$ and Assumption \ref{assum1} hold.
For any bounded set $ B_1$ in $ H$,  there exists a random variable $ T_{B_1}(\omega)>T_B(\omega)$ such that,  for any $t\geq T_{B_1}(\omega)$,
\[
\sup_{v(0)\in B_1} \left( \big\|A^{\frac54}v(t,\theta_{-t}\omega,v(0))  \big\|^2 + \int_{t-\frac 14 }^t \|A^{\frac{15}{8}}v(s,\theta_{-t}\omega,v(0))\|^2   \d s\right)
\leq  \zeta_9(\omega),
\]
here $\zeta_9(\omega)$  is defined by \eqref{7.4} such that $\zeta_9(\omega) > \zeta_2(\omega)> 1$, $ \omega\in\Omega$.
\end{lemma}
\begin{proof}
Taking the inner product of the system \eqref{2.2} in $H$ by $A^{\frac52}v$ and  integration by parts, we have
\begin{align}\label{7.1}
\frac12\frac{\d}{\d t}\|A^{\frac54}v\|^2+ \nu \|A^{\frac{15}{8}}v\|^2&=-\big(B(v+hz(\theta_t\omega),v+hz(\theta_t\omega)), A^{\frac52}v \big)\nonumber\\
& +\big(f-\nu A^{\frac54}hz(\theta_t\omega) +hz(\theta_t\omega),A^{\frac52}v \big).
\end{align}
For the first term on the right hand side of \eqref{7.1}, by integration by parts and
 applying the H\"{o}lder inequality,  the Sobolev embedding inequality  and the Young inequality, we can deduce that
\begin{align}\label{7.2}
&\big| \big(B(v+hz(\theta_t\omega),v+hz(\theta_t\omega)), A^{\frac{5}{2}}v \big)\big|
 \nonumber\\
&  \leq C\|A^{\frac{15}{8}}v\|\|A^{\frac{5}{8}}(v+hz(\theta_t\omega))\|_{L^{12}}
\|A^{\frac12}(v+hz(\theta_t\omega))\|_{L^{\frac{12}{5}}}\nonumber\\
&+C\|A^{\frac{15}{8}}v\|\|v+hz(\theta_t\omega)\|_{L^{12}}
\|A^{\frac98}(v+hz(\theta_t\omega))\|_{L^{\frac{12}{5}}}\nonumber\\
&  \leq C\|A^{\frac{15}{8}}v\|\|A^{\frac{5}{4}}(v+hz(\theta_t\omega))\|\|A^{\frac{5}{8}}(v+hz(\theta_t\omega))\|\nonumber\\
& \leq \frac{\nu}{4}\|A^{\frac{15}{8}}v\|^2+C\|A^{\frac{5}{8}}(v+hz(\theta_t\omega))\|^2\|A^{\frac{5}{4}}(v+hz(\theta_t\omega))\|^2
\nonumber\\
&  \leq \frac{\nu}{4}\|A^{\frac{15}{8}}v\|^2+C(\|A^{\frac{5}{8}}v\|^2+\|A^{\frac{5}{8}}h\|^2|z(\theta_t\omega)|^2)
(\|A^{\frac{5}{4}}v\|^2+\|A^{\frac{5}{4}}h\|^2|z(\theta_t\omega)|^2)\nonumber\\
&  \leq \frac{\nu}{4}\|A^{\frac{15}{8}}v\|^2+C(\|A^{\frac{5}{8}}v\|^2+|z(\theta_t\omega)|^2)
\|A^{\frac{5}{4}}v\|^2+C|z(\theta_t\omega)|^4.
\end{align}
For the second term on the right hand side of \eqref{7.1}, by integration by parts and
 applying the similar method, we also deduce that
\begin{align}
&\big (f-\nu A^{\frac{5}{4}}hz(\theta_t\omega) +hz(\theta_t\omega),A^{\frac{5}{2}}v \big)\nonumber\\
& \leq \|A^{\frac{5}{8}}f\| \|A^{\frac{15}{8}}v\|+  \nu\|A^{\frac{15}{8}}h\||z(\theta_t\omega)|\|A^{\frac{15}{8}}v\|+ C  \|A^{\frac{5}{8}}h\||z(\theta_t\omega)|\|A^{\frac{15}{8}}v\|\nonumber\\
& \leq \frac  \nu 4 \|A^{\frac{15}{8}}v\|^2+C\|A^{\frac{5}{8}}f\|^2+ C  \|A^{\frac{15}{8}}h\|^2|z(\theta_t\omega)|^2 + C  \|A^{\frac{5}{8}}h\|^2|z(\theta_t\omega)|^2\nonumber\\
&\leq  \frac  \nu 4 \|A^{\frac{15}{8}}v\|^2+C(\|A^{\frac{5}{8}}f\|^2+ \|A^{\frac{15}{8}}h\|^2 |z(\theta_t\omega)|^2 )\nonumber\\
&\leq  \frac  \nu 4 \|A^{\frac{15}{8}}v\|^2+C(\|A^{\frac{5}{8}}f\|^2+ |z(\theta_t\omega)|^2).\label{7.3}
\end{align}
Inserting \eqref{7.2} and \eqref{7.3} into \eqref{7.1}, it can get that
\begin{align}
\frac{\d}{\d t}\|A^{\frac54}v\|^2 +  \nu  \|A^{\frac{15}{8}}v\|^2
&\leq  C\left ( \|A^{\frac58}v\|^2+ |z(\theta_t\omega)|^2 \right)\|A^{\frac54}v\|^2  +
C \left(1 +|z(\theta_t\omega)|^4+\|A^{\frac{5}{8}}f\|^2 \right)\nonumber \\
&\leq  C\left ( \|A^{\frac58}v\|^2+ |z(\theta_t\omega)|^2 \right)\|A^{\frac54}v\|^2  +
C \left(1 +|z(\theta_t\omega)|^4 \right).
\end{align}
For $\eta\in  (t-\frac{1}{2},t-\frac{1}{4})$, applying Gronwall's lemma on $(\eta, t)$  for $t >1$, it can deduce that
\begin{align*}
&\|A^{\frac54}v(t)\|^2
+ \nu \int_\eta ^te^{ C\int_s^t(\|A^{\frac{5}{8}}v\|^2+|z(\theta_\tau \omega)|^2) \d \tau}\|A^{\frac{15}{8}} v(s)\|^2  \d s \nonumber\\
&\leq e^{ C\int_\eta^t(\|A^{\frac{5}{8}}v\|^2+|z(\theta_\tau\omega)|^2) \d \tau}\|A^{\frac54}v( \eta)\|^2 +C\int_\eta^te^{C \int_s^t (\|A^{\frac{5}{8}}v\|^2+ |z (\theta_\tau\omega)|^2) \d \tau} \left (1 +|z(\theta_s\omega)|^4 \right)\d s,
\end{align*}
and  integrating over $\eta\in (t-\frac{1}{2},t-\frac 14) $ yields
\begin{align*}
& \frac14  \|A^{\frac54}v(t )\|^2
+\frac  \nu 4\int_{t-\frac 14 }^t e^{ C\int_s^t(\|A^{\frac{5}{8}}v\|^2+|z(\theta_\tau\omega)|^2) \d \tau}\|A^{\frac{15}{8}}v(s)\|^2 \d s \nonumber\\
& \leq \int_{t-\frac 12}^{t-\frac 14}  e^{ C\int_\eta^t(\|A^{\frac{5}{8}}v\|^2+|z(\theta_\tau\omega)|^2) \d \tau} \|A^{\frac54}v( \eta)\|^2  \d \eta
 \nonumber\\
& +C\int_{t-\frac{1}{2}}^te^{ C\int_s^t (\|A^{\frac{5}{8}}v\|^2+ |z (\theta_\tau\omega)|^2) \d \tau} \left (1 +|z(\theta_s\omega)|^4 \right) \d s \nonumber\\
&\leq Ce^{   C \int_{ t-\frac{1}{2}}^t (\|A^{\frac{5}{8}}v\|^2+|z(\theta_\tau\omega)|^2)  \d \tau}  \left[  \int_{t-\frac{1}{2}}^{t  }  \|A^{\frac54}v( \eta)\|^2  \d \eta  + \int_{t-\frac{1}{2}}^t \left (1 +|z(\theta_s\omega)|^4 \right)\d s\right]  .
\end{align*}
Replacing $\omega$ with $\theta_{-t}\omega$, then we can get
\begin{align*}
&   \|A^{\frac54}v(t,\theta_{-t}\omega, v(0) )\|^2\nonumber\\
& +   \nu \int_{t-\frac 14 }^te^{  C\int_{s}^t\|A^{\frac{5}{8}}v(\tau ,\theta_{-t}\omega, v(0))\|^2\d \tau+C\int_{s-t}^0|z(\theta_\tau\omega)|^2 \d \tau}\|A^{\frac{15}{8}} v(s ,\theta_{-t}\omega, v(0))\|^2 \d s\nonumber\\
&\leq  C e^{ C \int_{t-\frac{1}{2}}^t(\|A^{\frac{5}{8}}v(\tau ,\theta_{-t}\omega, v(0))\|^2+ |z (\theta_{\tau-t}\omega)|^2) \d \tau}\nonumber\\
&\cdot\left [ \int_{t-\frac{1}{2}}^t   \|A^{\frac54}v(s,\theta_{-t}\omega, v(0) )\|^2\d s+\int_{t-\frac{1}{2}}^t \left(1+ |z(\theta_{s-t} \omega)|^4  \right)  \d s\right]\nonumber\\
& =  C e^{ C \int_{t-\frac{1}{2}}^t\|A^{\frac{5}{8}}v(\tau ,\theta_{-t}\omega, v(0))\|^2\d \tau+C \int_{-\frac{1}{2}}^0 |z (\theta_{\tau}\omega)|^2 \d \tau}\nonumber\\
&\cdot\left [ \int_{t-\frac{1}{2}}^t   \|A^{\frac54}v(s,\theta_{-t}\omega, v(0) )\|^2\d s+\int_{-\frac{1}{2}}^0 \left(1+ |z(\theta_{s} \omega)|^4  \right)  \d s\right].
\end{align*}
By  virtue of the Lemma   \ref{lemma4.1}, it can deduce that
\begin{align*}
 \nu  \int_{t-\frac{1}{2}}^t  \|A^{\frac58}v(s,\theta_{-t}\omega, v(0) )\|^2  \d s
 & \leq   \nu e^{\left( \frac 8\alpha -3 \right)  \lambda }   \int_{t-\frac{1}{2}}^t  e^{\left( \frac 8\alpha -3 \right)  \lambda (s-t) } \|A^{\frac58}v(s,\theta_{-t}\omega, v(0) )\|^2  \d s
 \\
 & \leq   \nu e^{\left( \frac 8\alpha -3 \right)  \lambda }   \int_{0}^t  e^{\left( \frac 8\alpha -3 \right)  \lambda (s-t) } \|A^{\frac58}v(s,\theta_{-t}\omega, v(0) )\|^2  \d s
 \\
 & \leq  \frac{2}{\alpha }   e^{\left( \frac 8\alpha -3 \right)  \lambda }   \left( e^{-\lambda t}\|v(0)\|^2
 +   \zeta_1(\omega)\right) , \quad    t\geq T_1(\omega)>1 . \nonumber
\end{align*}
Hence, it is easy to get for $t\geq T_{B_1}(\omega):=T_{B}(\omega)+1$
\begin{align*}
&   \|A^{\frac54}v(t,\theta_{-t}\omega, v(0) )\|^2 +   \nu \int_{t-\frac 14 }^t\|A^{\frac{15}{8}} v(s ,\theta_{-t}\omega, v(0))\|^2 \d s\nonumber\\
&   \leq\|A^{\frac54}v(t,\theta_{-t}\omega, v(0) )\|^2\nonumber\\
& +   \nu \int_{t-\frac 14 }^te^{  C\int_{s}^t\|A^{\frac{5}{8}}v(\tau ,\theta_{-t}\omega, v(0))\|^2\d \tau+C\int_{s-t}^0|z(\theta_\tau\omega)|^2 \d \tau}\|A^{\frac{15}{8}} v(s ,\theta_{-t}\omega, v(0))\|^2 \d s\nonumber\\
& \leq  C e^{ C \int_{t-\frac{1}{2}}^t\|A^{\frac{5}{8}}v(\tau ,\theta_{-t}\omega, v(0))\|^2\d \tau+C \int_{-\frac{1}{2}}^0 |z (\theta_{\tau}\omega)|^2 \d \tau}\nonumber\\
&\cdot\left [ \int_{t-\frac{1}{2}}^t   \|A^{\frac54}v(s,\theta_{-t}\omega, v(0) )\|^2\d s+\int_{-\frac{1}{2}}^0 \left(1+ |z(\theta_{s} \omega)|^4  \right)  \d s\right]\nonumber\\
&\leq Ce^{C\zeta_1(\omega)+C \int_{-\frac{1}{2}}^0 |z (\theta_{\tau}\omega)|^2 \d \tau}\left(\zeta_2(\omega)+\int_{-\frac{1}{2}}^0 (1+ |z(\theta_{s} \omega)|^4  )  \d s\right)\nonumber\\
&\leq Ce^{C\zeta_1(\omega)+C \int_{-\frac{1}{2}}^0 |z (\theta_{\tau}\omega)|^2 \d \tau}\left(\zeta_2(\omega)+\int_{-\frac{1}{2}}^0 |z(\theta_{s} \omega)|^4    \d s\right),
\end{align*}
where  we have used  $\zeta_2(\omega)>\zeta_1(\omega)> 1$. Let
\begin{align}\label{7.4}
 \zeta_9(\omega) :=  Ce^{C\zeta_1(\omega)+C \int_{-\frac{1}{2}}^0 |z (\theta_{\tau}\omega)|^2 \d \tau}\left(\zeta_2(\omega)+\int_{-\frac{1}{2}}^0 |z(\theta_{s} \omega)|^4    \d s\right)  ,
\quad \omega\in \Omega,
\end{align}
then $\zeta_9(\omega)$ is a tempered random variable such that $\zeta_9(\omega) > \zeta_2(\omega)> 1$.  This completes the proof of Lemma \ref{lemma7.1}.

\end{proof}

\subsection{Estimate in $H^{k}$ for $k>\frac{5}{2}$}
In this subsection, let $f\in H^{k-\frac{5}{4}}$ and $k>\frac{5}{2}$ and $h\in H^{k+\frac{5}{4}}$,     an  $H^k$ random absorbing set of system \eqref{2.2} is proved.
\begin{lemma}[$H^{k}$ bound for $k>\frac{5}{2}$] \label{lemma7.2}
Assume that  $f\in H^{k-\frac54}$ and Assumption \ref{assum1} hold.
Then  there exists a random variable $ T_{k}(\omega)>T_{k-1}(\omega)$ such that,  for any $t\geq T_{k}(\omega)$,
\[
  \big\|A^{\frac k2}v(t,\theta_{-t}\omega,v(0))  \big\|^2 + \int_{t-\frac{1}{2^{k-\frac{1}{2}}} }^t \|A^{\frac{k}{2}+\frac{5}{8}}v(s,\theta_{-t}\omega,v(0))\|^2   \d s
\leq  \zeta_k(\omega),
\]
here $\zeta_k(\omega)$  is defined by \eqref{7.10} such that $\zeta_k(\omega) > \zeta_{k-1}(\omega)> 1$, $ \omega\in\Omega$.
\end{lemma}

\begin{proof}
Taking the inner product of the system \eqref{2.2} in $H$ by $A^{k}v$ and  integration by parts, we have
\begin{align}\label{7.5}
\frac 12\frac{\d}{\d t}\|A^{\frac k2}v\|^2+ \nu \|A^{\frac{k}{2}+\frac{5}{8}}v\|^2&=-\big(B(v+hz(\theta_t\omega),v+hz(\theta_t\omega)), A^{k}v \big)\nonumber\\
& +\big(f-\nu A^{\frac54}hz(\theta_t\omega) +hz(\theta_t\omega),A^{k}v \big).
\end{align}
For the first term on the right hand side of \eqref{7.5}, by integration by parts and
 applying the H\"{o}lder inequality,  the Sobolev embedding inequality  and the Young inequality, we can deduce that
\begin{align}\label{7.6}
&\big| \big(B(v+hz(\theta_t\omega),v+hz(\theta_t\omega)), A^{k}v \big)\big|
 \nonumber\\
&  \leq C\|A^{\frac{k}{2}+\frac{5}{8}}v\|\|A^{\frac{k}{2}-\frac{5}{8}}(v+hz(\theta_t\omega))\|_{L^{12}}
\|A^{\frac12}(v+hz(\theta_t\omega))\|_{L^{\frac{12}{5}}}\nonumber\\
&+C\|A^{\frac{k}{2}+\frac{5}{8}}v\|\|v+hz(\theta_t\omega)\|_{L^{12}}
\|A^{\frac{k}{2}-\frac{1}{8}}(v+hz(\theta_t\omega))\|_{L^{\frac{12}{5}}}\nonumber\\
&  \leq C\|A^{\frac{k}{2}+\frac{5}{8}}v\|\|A^{\frac{k}{2}}(v+hz(\theta_t\omega))\|
\|A^{\frac{5}{8}}(v+hz(\theta_t\omega))\|\nonumber\\
& \leq \frac{\nu}{4}\|A^{\frac{k}{2}+\frac{5}{8}}v\|^2+C\|A^{\frac{5}{8}}(v+hz(\theta_t\omega))\|^2
\|A^{\frac{k}{2}}(v+hz(\theta_t\omega))\|^2
\nonumber\\
&  \leq \frac{\nu}{4}\|A^{\frac{k}{2}+\frac{5}{8}}v\|^2+C(\|A^{\frac{5}{8}}v\|^2+\|A^{\frac{5}{8}}h\|^2|z(\theta_t\omega)|^2)
(\|A^{\frac{k}{2}}v\|^2+\|A^{\frac{k}{2}}h\|^2|z(\theta_t\omega)|^2)\nonumber\\
&  \leq \frac{\nu}{4}\|A^{\frac{k}{2}+\frac{5}{8}}v\|^2+C(\|A^{\frac{5}{8}}v\|^2+|z(\theta_t\omega)|^2)
(\|A^{\frac{k}{2}}v\|^2+|z(\theta_t\omega)|^2)\nonumber\\
&  \leq \frac{\nu}{4}\|A^{\frac{k}{2}+\frac{5}{8}}v\|^2+C(\|A^{\frac{5}{8}}v\|^2+|z(\theta_t\omega)|^2)
\|A^{\frac{k}{2}}v\|^2+C|z(\theta_t\omega)|^4.
\end{align}
For the second term on the right hand side of \eqref{7.5}, by integration by parts and
 applying the similar method, we also deduce that
\begin{align}
&\big (f-\nu A^{\frac{5}{4}}hz(\theta_t\omega) +hz(\theta_t\omega),A^{k}v \big)\nonumber\\
& \leq \|A^{\frac{k}{2}-\frac{5}{8}}f\| \|A^{\frac{k}{2}+\frac{5}{8}}v\|+  \nu\|A^{\frac{k}{2}+\frac{5}{8}}h\||z(\theta_t\omega)|\|A^{\frac{k}{2}+\frac{5}{8}}v\|+ C  \|A^{\frac{k}{2}-\frac{5}{8}}h\||z(\theta_t\omega)|\|A^{\frac{k}{2}+\frac{5}{8}}v\|\nonumber\\
& \leq \frac  \nu 4 \|A^{\frac{k}{2}+\frac{5}{8}}v\|^2+C\|A^{\frac{k}{2}-\frac{5}{8}}f\|^2+ C  \|A^{\frac{k}{2}+\frac{5}{8}}h\|^2|z(\theta_t\omega)|^2 + C  \|A^{\frac{k}{2}-\frac{5}{8}}h\|^2|z(\theta_t\omega)|^2\nonumber\\
&\leq  \frac  \nu 4 \|A^{\frac{k}{2}+\frac{5}{8}}v\|^2+C(\|A^{\frac{k}{2}-\frac{5}{8}}f\|^2+ \|A^{\frac{k}{2}+\frac{5}{8}}h\|^2 |z(\theta_t\omega)|^2 )\nonumber\\
&\leq  \frac  \nu 4 \|A^{\frac{k}{2}+\frac{5}{8}}v\|^2+C(\|A^{\frac{k}{2}-\frac{5}{8}}f\|^2+ |z(\theta_t\omega)|^2).\label{7.7}
\end{align}
Inserting \eqref{7.6} and \eqref{7.7} into \eqref{7.5}, it is easy to deduce that
\begin{align}\label{7.20}
&\frac{\d}{\d t}\|A^{\frac k2}v\|^2 +  \nu  \|A^{\frac{k}{2}+\frac{5}{8}}v\|^2\nonumber \\
&\leq  C\left ( \|A^{\frac58}v\|^2+ |z(\theta_t\omega)|^2 \right)\|A^{\frac k2}v\|^2  +
C \left(1 +|z(\theta_t\omega)|^4+\|A^{\frac{k}{2}-\frac{5}{8}}f\|^2 \right)\nonumber \\
&\leq  C\left ( \|A^{\frac58}v\|^2+ |z(\theta_t\omega)|^2 \right)\|A^{\frac k2}v\|^2  +
C \left(1 +|z(\theta_t\omega)|^4 \right).
\end{align}
For $\eta\in  (t-\frac{1}{2^{k-\frac{3}{2}}},t-\frac{1}{2^{k-\frac{1}{2}}})$, applying Gronwall's lemma on $(\eta, t)$  for $t >1$, then we have
\begin{align*}
&\|A^{\frac k2}v(t)\|^2
+ \nu \int_\eta ^te^{ C\int_s^t(\|A^{\frac{5}{8}}v\|^2+|z(\theta_\tau \omega)|^2) \d \tau}\|A^{\frac{k}{2}+\frac{5}{8}} v(s)\|^2  \d s \nonumber\\
&\leq e^{ C\int_\eta^t(\|A^{\frac{5}{8}}v\|^2+|z(\theta_\tau\omega)|^2) \d \tau}\|A^{\frac k2}v( \eta)\|^2 +C\int_\eta^te^{C \int_s^t (\|A^{\frac{5}{8}}v\|^2+ |z (\theta_\tau\omega)|^2) \d \tau} \left (1 +|z(\theta_s\omega)|^4 \right)\d s,
\end{align*}
and  integrating over $\eta\in (t-\frac{1}{2^{k-\frac{3}{2}}},t-\frac{1}{2^{k-\frac{1}{2}}}) $ yields
\begin{align*}
& \frac{1}{2^{k-\frac{1}{2}}}  \|A^{\frac k2}v(t )\|^2
+  \frac{\nu}{2^{k-\frac{1}{2}}}\int_{t-\frac{1}{2^{k-\frac{1}{2}}}}^t e^{ C\int_s^t(\|A^{\frac{5}{8}}v\|^2+|z(\theta_\tau\omega)|^2) \d \tau}\|A^{\frac{k}{2}+\frac{5}{8}}v(s)\|^2 \d s \nonumber\\
& \leq \int_{t-\frac{1}{2^{k-\frac{3}{2}}}}^{t-\frac{1}{2^{k-\frac{1}{2}}}}  e^{ C\int_\eta^t(\|A^{\frac{5}{8}}v\|^2+|z(\theta_\tau\omega)|^2) \d \tau} \|A^{\frac k2}v( \eta)\|^2  \d \eta
 \nonumber\\
& +C\int_{t-\frac{1}{2^{k-\frac{3}{2}}}}^te^{ C\int_s^t (\|A^{\frac{5}{8}}v\|^2+ |z (\theta_\tau\omega)|^2) \d \tau} \left (1 +|z(\theta_s\omega)|^4 \right) \d s \nonumber\\
&\leq Ce^{   C \int_{ t-\frac{1}{2^{k-\frac{3}{2}}}}^t (\|A^{\frac{5}{8}}v\|^2+|z(\theta_\tau\omega)|^2)  \d \tau}  \left[  \int_{t-\frac{1}{2^{k-\frac{3}{2}}}}^{t  }  \|A^{\frac k2}v( \eta)\|^2  \d \eta  + \int_{t-\frac{1}{2^{k-\frac{3}{2}}}}^t \left (1 +|z(\theta_s\omega)|^4 \right)\d s\right]  .
\end{align*}
Replacing $\omega$ with $\theta_{-t}\omega$,  we can deduce that
\begin{align*}
&   \|A^{\frac k2}v(t,\theta_{-t}\omega, v(0) )\|^2\nonumber\\
& +   \nu \int_{t-\frac{1}{2^{k-\frac{1}{2}}}}^te^{  C\int_{s}^t\|A^{\frac{5}{8}}v(\tau ,\theta_{-t}\omega, v(0))\|^2\d \tau+C\int_{s-t}^0|z(\theta_\tau\omega)|^2 \d \tau}\|A^{\frac{k}{2}+\frac{5}{8}} v(s ,\theta_{-t}\omega, v(0))\|^2 \d s\nonumber\\
&\leq  C e^{ C \int_{t-\frac{1}{2^{k-\frac{3}{2}}}}^t(\|A^{\frac{5}{8}}v(\tau ,\theta_{-t}\omega, v(0))\|^2+ |z (\theta_{\tau-t}\omega)|^2) \d \tau}\nonumber\\
&\cdot\left [ \int_{t-\frac{1}{2^{k-\frac{3}{2}}}}^t   \|A^{\frac k2}v(s,\theta_{-t}\omega, v(0) )\|^2\d s+\int_{t-\frac{1}{2^{k-\frac{3}{2}}}}^t \left(1+ |z(\theta_{s-t} \omega)|^4  \right)  \d s\right]\nonumber\\
& =  C e^{ C \int_{t-\frac{1}{2^{k-\frac{3}{2}}}}^t\|A^{\frac{5}{8}}v(\tau ,\theta_{-t}\omega, v(0))\|^2\d \tau+C \int_{-\frac{1}{2^{k-\frac{1}{2}}}}^0 |z (\theta_{\tau}\omega)|^2 \d \tau}\nonumber\\
&\cdot\left [ \int_{t-\frac{1}{2^{k-\frac{3}{2}}}}^t   \|A^{\frac k2}v(s,\theta_{-t}\omega, v(0) )\|^2\d s+\int_{-\frac{1}{2^{k-\frac{3}{2}}}}^0 \left(1+ |z(\theta_{s} \omega)|^4  \right)  \d s\right].
\end{align*}
Applying the above Lemma \ref{lemma4.1}, we have
\begin{align*}
&\int_{t-\frac{1}{2^{k-\frac{3}{2}}}}^t\|A^{\frac{5}{8}}v(\tau ,\theta_{-t}\omega, v(0))\|^2\d \tau\nonumber\\
&\leq e^{(\frac{8}{\alpha}-3)\lambda}\int_{t-\frac{1}{2^{k-\frac{3}{2}}}}^te^{(\frac{8}{\alpha}-3)\lambda(\tau-t)}
\|A^{\frac{5}{8}}v(\tau ,\theta_{-t}\omega, v(0))\|^2\d \tau\nonumber\\
&\leq e^{(\frac{8}{\alpha}-3)\lambda}\int_{0}^te^{(\frac{8}{\alpha}-3)\lambda(\tau-t)}
\|A^{\frac{5}{8}}v(\tau ,\theta_{-t}\omega, v(0))\|^2\d \tau\nonumber\\
&\leq \frac{2}{\alpha\nu}e^{(\frac{8}{\alpha}-3)\lambda}(e^{-\lambda t}\|v_0\|^2+\zeta_1(\omega)),~~t\geq T_1(\omega)>1.
\end{align*}
Moreover, it yields
\begin{align}
&   \|A^{\frac k 2}v(t,\theta_{-t}\omega, v(0) )\|^2 +   \nu \int_{t-\frac{1}{2^{k-\frac{1}{2}}} }^t\|A^{\frac{k}{2}+\frac{5}{8}} v(s ,\theta_{-t}\omega, v(0))\|^2 \d s\nonumber\\
&   \leq\|A^{\frac k2}v(t,\theta_{-t}\omega, v(0) )\|^2\nonumber\\
& +   \nu \int_{t-\frac{1}{2^{k-\frac{1}{2}}} }^te^{  C\int_{s}^t\|A^{\frac{5}{8}}v(\tau ,\theta_{-t}\omega, v(0))\|^2\d \tau+C\int_{s-t}^0|z(\theta_\tau\omega)|^2 \d \tau}\|A^{\frac{k}{2}+\frac{5}{8}} v(s ,\theta_{-t}\omega, v(0))\|^2 \d s\nonumber\\
& \leq  C e^{ C \int_{t-\frac{1}{2^{k-\frac{3}{2}}}}^t\|A^{\frac{5}{8}}v(\tau ,\theta_{-t}\omega, v(0))\|^2\d \tau+C \int_{-\frac{1}{2^{k-\frac{3}{2}}}}^0 |z (\theta_{\tau}\omega)|^2 \d \tau}\nonumber\\
&\cdot\left [ \int_{t-\frac{1}{2^{k-\frac{3}{2}}}}^t   \|A^{\frac k2}v(s,\theta_{-t}\omega, v(0) )\|^2\d s+\int_{-\frac{1}{2^{k-\frac{3}{2}}}}^0 \left(1+ |z(\theta_{s} \omega)|^4  \right)  \d s\right]\nonumber\\
&\leq Ce^{C\zeta_1(\omega)+C \int_{-\frac{1}{2^{k-\frac{3}{2}}}}^0 |z (\theta_{\tau}\omega)|^2 \d \tau}\nonumber\\
&\cdot\left [ \int_{t-\frac{1}{2^{k-\frac{3}{2}}}}^t   \|A^{\frac k2}v(s,\theta_{-t}\omega, v(0) )\|^2\d s+\int_{-\frac{1}{2^{k-\frac{3}{2}}}}^0 \left(1+ |z(\theta_{s} \omega)|^4  \right)  \d s\right],\label{7.9}
\end{align}
where  we have used  $\zeta_1(\omega)> 1$. In order to prove the estimation of $H^k$ of solutions of the system \eqref{2.2}, we let the estimation of $H^{k-\frac54}$ of solutions of the system \eqref{2.2} satisfy the following inequality \eqref{7.8}, then there exist a random variables $\zeta_{k-1}(\omega)$ and $T_{k-1}(\omega)$ such that
\begin{align}\label{7.8}
  \big\|A^{\frac{k}{2}-\frac{5}{8}}v(t,\theta_{-t}\omega,v(0))  \big\|^2 + \int_{t-\frac{1}{2^{k-\frac{3}{2}}} }^t \|A^{\frac{k}{2}}v(s,\theta_{-t}\omega,v(0))\|^2   \d s
\leq  \zeta_{k-1}(\omega),
\end{align}
here $\zeta_{k-1}(\omega)$ is a tempered random variable such that $\zeta_{k-1}(\omega)>\zeta_{9}(\omega)>1$. Inserting \eqref{7.8} into \eqref{7.9} yields
\begin{align}
&   \|A^{\frac k 2}v(t,\theta_{-t}\omega, v(0) )\|^2 +   \nu \int_{t-\frac{1}{2^{k-\frac{1}{2}}} }^t\|A^{\frac{k}{2}+\frac{5}{8}} v(s ,\theta_{-t}\omega, v(0))\|^2 \d s\nonumber\\
&\leq Ce^{C\zeta_1(\omega)+C \int_{-\frac{1}{2^{k-\frac{3}{2}}}}^0 |z (\theta_{\tau}\omega)|^2 \d \tau}\cdot\left [ \zeta_{k-1}(\omega)+\int_{-\frac{1}{2^{k-\frac{3}{2}}}}^0 \left(1+ |z(\theta_{s} \omega)|^4  \right)  \d s\right]\nonumber\\
&\leq Ce^{C\zeta_1(\omega)+C \int_{-\frac{1}{2^{k-\frac{3}{2}}}}^0 |z (\theta_{\tau}\omega)|^2 \d \tau}\cdot\left [ \zeta_{k-1}(\omega)+\int_{-\frac{1}{2^{k-\frac{3}{2}}}}^0  |z(\theta_{s} \omega)|^4    \d s\right].
\end{align}
Let
\begin{align}\label{7.10}
 \zeta_k(\omega) :=  Ce^{C\zeta_1(\omega)+C \int_{-\frac{1}{2^{k-\frac{3}{2}}}}^0 |z (\theta_{\tau}\omega)|^2 \d \tau}\left(\zeta_{k-1}(\omega)+\int_{-\frac{1}{2^{k-\frac{3}{2}}}}^0 |z(\theta_{s} \omega)|^4    \d s\right)  ,
\quad \omega\in \Omega,
\end{align}
here $\zeta_k(\omega)$ is a tempered random variable such that $\zeta_k(\omega) > \zeta_{k-1}(\omega)> 1$.  This completes the proof of Lemma \ref{lemma7.2}.

\end{proof}

\subsection{Local $(H , H^{k})$-Lipschitz continuity}
In this section, let $k\geq\frac{5}{2}$, we will prove a local $(H, H^k)$-Lipschitz continuity in initial values of the solutions of the random fractional three-dimensional NS  equations \eqref{2.2}.   This will be done step-by-step by   showing the local $(H,H^{k-\frac{5}{4}})$-Lipschitz continuity, the local $(H^{k-\frac{5}{4}},H^k)$-Lipschitz continuity and finally the local $(H,H^k)$-Lipschitz continuity.

Let $k\geq\frac{5}{2}$. The random set $\mathfrak{B}^k=\{\mathfrak{B}^k(\omega)\}_{\omega\in\Omega}$ in $H^{k}$ is defined by
\begin{align}
\mathfrak{B}^k(\omega):= \left \{ v \in H^{k}:\|A^{\frac k2}v\|^2\leq \zeta_k(\omega)
\right \}, \quad \omega\in\Omega,
\end{align}
where $\zeta_k(\cdot)$ is a tempered random variable is given by \eqref{7.10}.

In order to prove the  $(H , H^{k})$-Lipschitz continuity of the system \eqref{5.1}, we set the  $(H , H^{k-1})$-Lipschitz continuity of the system \eqref{5.1} and get the following main result.

\begin{theorem}[Local $(H, H^{k-\frac{5}{4}})$-Lipschtiz] \label{theorem7.2}
 Let Assumption \ref{assum} hold and $f\in H^{k-\frac{5}{4}}$. Then there
  exist  a random variables $T^*=T_{k-1} (\omega)  $   and $ L_{k-1}(\mathfrak D,\omega )$ such that   two solutions $v_1$ and $v_2$ of random fractional three-dimensional NS equations of \eqref{2.2} corresponding to initial values   $v_{1,0},$ $ v_{2,0}$ in $\mathfrak D \left (\theta_{-T^*(\omega)}\omega \right)$, respectively, satisfy
\begin{align}
  &
  \big\|A^{\frac k2-\frac{5}{8}}(v_1 \!  \big (T^*(\omega),\theta_{-T^*(\omega)}\omega,v_{1,0}\big)
  -v_2 \big (T^*(\omega),\theta_{-T^*(\omega)}\omega,v_{2,0} \big) \big) \|^2 \nonumber\\
& \leq  L_{k-1} (\mathfrak D,\omega)\|v_{1,0}-v_{2,0}\|^2,\quad \omega\in \Omega.
\end{align}
\end{theorem}
Then we will prove the following $(H^{k-\frac{5}{4}},H^k)$-Lipschitz continuity.

\begin{lemma}[Local $(H^{k-\frac54}, H^{k})$-Lipschitz on $\mathfrak B^{k-\frac{5}{4}}$]\label{lemma7.3}  Let Assumption \ref{assum} hold and $f\in H^{k-\frac{5}{4}}$ for $k\geq\frac{5}{2}$, then there  exist  random variables $T_k(\omega)$ and $ L_k(\mathfrak B^{k-\frac{5}{4}}, \omega )$ such that    two solutions $v_1$ and $v_2$ of random fractional three-dimensional NS equations \eqref{2.2} corresponding to initial values  $v_{1,0},$ $v_{2,0}$ in $\mathfrak{B}^{k-\frac{5}{4}}(\theta_{-T_\omega}\omega)$, respectively,   satisfy
\[
 \big \|A^{\frac k2}(v_1(T_k\omega,\theta_{-T_k}\omega,v_{1,0})-v_2(T_k\omega,\theta_{-T_k}\omega,v_{2,0}) \big) \|{^2}\leq  L_k(\mathfrak B^{k-1}, \omega)\|A^{\frac{k}{2}-\frac58}(v_{1,0}-v_{2,0})\|^2,
 \]
 for $\omega\in \Omega$.
\end{lemma}

\begin{proof}
Let $k\geq\frac{5}{2}$.   Taking the inner product of the system  \eqref{5.1} in $H$ by $A^k\bar{v}$ and integration by parts, we can deduce that
\begin{align}\label{7.11}
&\frac{1}{2}\frac{\d}{\d t}\|A^\frac k2\bar{v}\|^2+ \nu  \|A^{\frac{k}{2}+\frac{5}{8}}\bar{v}\|^2\nonumber\\
&=- \big (B(v_1+hz(\theta_t\omega))-B(v_2+hz(\theta_t\omega)), A^k\bar{v}  \big )\nonumber\\
&=- \big (B(\bar{v},v_1+hz(\theta_t\omega)), A^k\bar{v}  \big )
- \big (B(v_2+hz(\theta_t\omega),\bar{v}), A^k\bar{v} \big ).
\end{align}
For the first term on the right hand side of the system \eqref{7.11}, applying the integration by parts, the H\"{o}lder inequality, the Sobolev embedding inequality  and the Young inequality, we can deduce that
\begin{align}\label{7.12}
 \big| \big(B(\bar{v},v_1+hz(\theta_t\omega)), A^k\bar{v}\big)
 \big| &\leq C\|A^{\frac{k}{2}-\frac{5}{8}}\bar{v}\|_{L^{12}}\|\nabla(v_1+hz(\theta_t\omega))\|_{L^{\frac{12}{5}}}\|A^{\frac{k}{2}+\frac{5}{8}}\bar{v}\|  \nonumber\\
&
 +C\|\bar{v}\|_{L^{12}}\|A^{\frac{k}{2}-\frac{1}{8}}(v_1+hz(\theta_t\omega))\|_{L^{\frac{12}{5}}}\|A^{\frac{k}{2}+\frac{5}{8}}\bar{v}\|
 \nonumber\\
&\leq C\|A^\frac k2\bar{v}\|\|A^\frac58(v_1+hz(\theta_t\omega))\|\|A^{\frac{k}{2}+\frac{5}{8}}\bar{v}\|  \nonumber\\
&
 +C\|A^\frac58\bar{v}\|\|A^\frac k2(v_1+hz(\theta_t\omega))\|\|A^{\frac{k}{2}+\frac{5}{8}}\bar{v}\|
 \nonumber\\
 &\leq \frac{ \nu}{4}\|A^{\frac{k}{2}+\frac{5}{8}}\bar{v}\|^2+C\|A^\frac58(v_1+hz(\theta_t\omega))\|^2\|A^\frac k2\bar{v}\|^2\nonumber\\
 &+C\|A^\frac k2(v_1+hz(\theta_t\omega))\|^2\|A^\frac58\bar{v}\|^2\nonumber\\
 &\leq \frac{ \nu}{4}\|A^{\frac{k}{2}+\frac{5}{8}}\bar{v}\|^2+C(\|A^\frac58v_1\|^2+\|A^\frac58h\|^2|z(\theta_t\omega)|^2)\|A^\frac k2\bar{v}\|^2\nonumber\\
 &+C(\|A^\frac k2v_1\|^2+\|A^\frac k2h\|^2|z(\theta_t\omega)|^2)\|A^\frac58\bar{v}\|^2\nonumber\\
 &\leq \frac{\nu}{4}\|A^{\frac{k}{2}+\frac{5}{8}}\bar{v}\|^2+C(\|A^\frac58v_1\|^2+|z(\theta_t\omega)|^2)\|A^\frac k2\bar{v}\|^2\nonumber\\
 &+C(\|A^\frac k2 v_1\|^2+|z(\theta_t\omega)|^2)\|A^\frac58\bar{v}\|^2.
\end{align}
For the second term on the right hand side of the system \eqref{7.11}, by the similar method, it can get that
\begin{align}\label{7.13}
\big| \big(B(v_2+hz(\theta_t\omega),\bar{v}), A^k\bar{v}\big)\big|
&\leq C\|A^{\frac{k}{2}-\frac{5}{8}}(v_2+hz(\theta_t\omega))\|_{L^{12}}\|\nabla\bar{v}\|_{L^\frac{12}{5}}
\|A^{\frac{k}{2}+\frac{5}{8}}\bar{v}\| \nonumber\\
&
 +C\|v_2+hz(\theta_t\omega)\|_{L^{12}}\|A^{\frac{k}{2}-\frac{1}{8}}\bar{v}\|_{L^{\frac{12}{5}}}
 \|A^{\frac{k}{2}+\frac{5}{8}}\bar{v}\|\nonumber\\
&\leq C\|A^\frac k2(v_2+hz(\theta_t\omega))\|\|A^\frac58\bar{v}\|\|A^{\frac{k}{2}+\frac{5}{8}}\bar{v}\| \nonumber\\
&
 +C\|A^\frac58(v_2+hz(\theta_t\omega))\|\|A^\frac k2\bar{v}\|\|A^{\frac{k}{2}+\frac{5}{8}}\bar{v}\|\nonumber\\
&\leq\frac{\nu}{4}\|A^{\frac{k}{2}+\frac{5}{8}}\bar{v}\|^2+C\|A^\frac k2(v_2+hz(\theta_t\omega))\|^2\|A^\frac58\bar{v}\|^2\nonumber\\
&
 +C\|A^\frac58(v_2+hz(\theta_t\omega))\|^2\|A^\frac k2\bar{v}\|^2\nonumber\\
&\leq\frac{\nu}{4}\|A^{\frac{k}{2}+\frac{5}{8}}\bar{v}\|^2+C(\|A^\frac k2v_2\|^2+\|A^\frac k2h\|^2|z(\theta_t\omega)|^2)
\|A^\frac58\bar{v}\|^2\nonumber\\
&
 +C(\|A^\frac58v_2\|^2+\|A^\frac58h\|^2|z(\theta_t\omega)|^2)\|A^\frac k2\bar{v}\|^2\nonumber\\
&\leq\frac{\nu}{4}\|A^{\frac{k}{2}+\frac{5}{8}}\bar{v}\|^2+C(\|A^\frac58v_2\|^2+|z(\theta_t\omega)|^2)
\|A^\frac k2\bar{v}\|^2\nonumber\\
&
 +C(\|A^\frac k2v_2\|^2+|z(\theta_t\omega)|^2)\|A^\frac58\bar{v}\|^2.
\end{align}
Inserting \eqref{7.12} and \eqref{7.13} into \eqref{7.11} yields
\begin{align}\label{7.16}
\frac{\d}{\d t}\|A^\frac k2\bar{v}\|^2+\nu\|A^{\frac{k}{2}+\frac{5}{8}}\bar{v}\|^2
&\leq C\|A^\frac k2\bar{v}\|^2 \left ( \sum_{j=1}^2 \|A^\frac58v_j\|^2 +|z(\theta_t\omega)|^2 \right ) \nonumber\\  &+C\|A^\frac58\bar{v}\|^2 \left (\sum_{j=1}^2\|A^\frac k2v_j\|^2+|z(\theta_t\omega)|^2 \right),
\end{align}
moreover,
\begin{align}
\frac{\d}{\d t}\|A^\frac k2\bar{v}\|^2
\leq C\|A^\frac k2\bar{v}\|^2 \left ( \sum_{j=1}^2 \|A^\frac58v_j\|^2 +|z(\theta_t\omega)|^2 \right ) +C\|A^\frac58\bar{v}\|^2 \left (\sum_{j=1}^2\|A^\frac k2v_j\|^2+|z(\theta_t\omega)|^2 \right). \nonumber
\end{align}
For any $s\in(t-\frac{1}{2^{k-\frac{3}{2}}},t)$ and $t\geq 1$, applying Gronwall's lemma on $(s,t)$, then we have
\begin{align}
&\|A^\frac k2\bar{v}(t)\|^2
-e^{\int_s^tC( \sum_{i=1}^2 \|A^\frac58v_i\|^2 +|z(\theta_\tau\omega)|^2) \d \tau}\|A^\frac k2\bar{v}(s)\|^2\nonumber\\
&\leq \int_s^tCe^{C\int_\eta^t( \sum_{i=1}^2 \|A^\frac58v_i\|^2+|z(\theta_\tau\omega)|^2) \d \tau}
\|A^\frac58\bar{v}\|^2 \left(\sum_{i=1}^2\|A^\frac k2v_i\|^2+|z(\theta_\eta\omega)|^2\right) \d \eta\nonumber\\
& \leq Ce^{\int_{t-\frac{1}{2^{k-\frac{3}{2}}}}^tC( \sum_{i=1}^2 \|A^\frac58v_i\|^2+|z(\theta_\tau\omega)|^2) \d \tau}
\int_{t-\frac{1}{2^{k-\frac{3}{2}}}}^t\|A^\frac58\bar{v}\|^2\left( \sum_{i=1}^2\|A^\frac k2v_i\|^2+|z(\theta_\eta\omega)|^2 \right) \d \eta. \nonumber
\end{align}
Then we integrate  $s$ on $(t-\frac{1}{2^{k-\frac{3}{2}}},t)$ to get
\begin{align}\label{7.14}
&\|A^\frac k2\bar{v}(t)\|^2
-2^{k-\frac{3}{2}}\int_{t-\frac{1}{2^{k-\frac{3}{2}}}}^te^{\int_s^tC( \sum_{i=1}^2 \|A^\frac58v_i\|^2+|z(\theta_\tau\omega)|^2) \d \tau}\|A^\frac k2\bar{v}(s)\|^2 \d s\nonumber\\
& \leq Ce^{\int_{t-\frac{1}{2^{k-\frac{3}{2}}}}^tC( \sum_{i=1}^2 \|A^\frac58v_i\|^2+|z(\theta_\tau\omega)|^2) \d \tau}
\int_{t-\frac{1}{2^{k-\frac{3}{2}}}}^t\|A^\frac58\bar{v}\|^2\left (\sum_{i=1}^2\|A^\frac k2v_i\|^2+|z(\theta_\eta\omega)|^2\right) \d \eta.
\end{align}

By virtue of  the initial values $v_{1,0}$ and $v_{2,0}$   belong to the $H^\frac54$ random  absorbing set  $\mathfrak B$ and   applying Gronwall's lemma to \eqref{mar9.2}, we can deduce that for $t\geq 1$
\begin{align}
 & \|A^\frac58\bar{v}(t)\|^2 + \nu \int_0^te^{C\int_s^t( \sum_{i=1}^2 \|A^\frac58v_i\|^2+|z(\theta_\tau\omega)|^2) \d \tau}\|A^\frac54\bar{v}(s) \|^2\d s
\nonumber\\
& \leq e^{C\int_0^t( \sum_{i=1}^2 \|A^\frac58v_i\|^2+|z(\theta_\tau\omega)|^2)\d \tau}\|A^\frac58\bar{v}(0)\|^2 . \nonumber
\end{align}
Then we can get the  following  estimate
\begin{align}\label{7.15}
&
\int_{t-\frac{1}{2^{k-\frac{3}{2}}}}^t\|A^\frac58\bar{v}\|^2 \left (\sum_{i=1}^2\|A^\frac k2v_i\|^2+|z(\theta_\eta\omega)|^2\right) \d \eta\nonumber\\
& \leq e^{\int_0^tC( \sum_{i=1}^2 \|A^\frac58v_i\|^2+|z(\theta_\tau\omega)|^2) \d \tau}\|A^\frac58\bar{v}(0)\|^2
\int_{t-\frac{1}{2^{k-\frac{3}{2}}}}^t   \left (\sum_{i=1}^2\|A^\frac k2v_i\|^2+|z(\theta_\eta\omega)|^2\right) \d \eta.
\end{align}
Inserting \eqref{7.15}  into \eqref{7.14} yields
\begin{align}\label{7.19}
 \|A^\frac k2\bar{v}(t)\|^2
&\leq Ce^{C\int_0^t( \sum_{i=1}^2 \|A^\frac58v_i\|^2+|z(\theta_\tau\omega)|^2) \d \tau}\nonumber\\
&\cdot\left (\int_{t-\frac{1}{2^{k-\frac{3}{2}}}}^t\|A^\frac k2\bar{v}(s)\|^2 \d s + \|A^\frac58\bar{v}(0)\|^2
\int_{t-\frac{1}{2^{k-\frac{3}{2}}}}^t   \left (\sum_{i=1}^2\|A^\frac k2v_i\|^2+|z(\theta_\eta\omega)|^2\right) \d \eta\right) .
\end{align}
For $k\geq\frac{5}{2}$, taking the inner product of the system \eqref{5.1} in $H$ by $A^{k-\frac{5}{4}}v$ and  integration by parts, by the similar method and \eqref{7.16}, we also get
\begin{align}\label{7.17}
\frac{\d}{\d t}\|A^{\frac k2-\frac{5}{8}}\bar{v}\|^2+\nu\|A^{\frac{k}{2}}\bar{v}\|^2
&\leq C\|A^{\frac k2-\frac{5}{8}}\bar{v}\|^2 \left ( \sum_{j=1}^2 \|A^\frac58v_j\|^2 +|z(\theta_t\omega)|^2 \right ) \nonumber\\  &+C\|A^\frac58\bar{v}\|^2 \left (\sum_{j=1}^2\|A^{\frac k2-\frac{5}{8}}v_j\|^2+|z(\theta_t\omega)|^2 \right).
\end{align}
Applying Gronwall's lemma to \eqref{7.17}, it yields
\begin{align}
&\|A^{\frac k2-\frac{5}{8}}\bar{v}(t)\|^2
+\nu\int_0^t\|A^\frac k2\bar{v}(s)\|^2\d s\nonumber\\
&\leq\|A^{\frac k2-\frac{5}{8}}\bar{v}(t)\|^2
+\nu\int_0^te^{C\int_\eta^t( \sum_{i=1}^2 \|A^\frac58v_i\|^2+|z(\theta_\tau\omega)|^2) \d \tau}
\|A^\frac k2\bar{v}(\eta)\|^2  \d \eta\nonumber\\
&\leq e^{C\int_0^t( \sum_{i=1}^2 \|A^\frac58v_i\|^2 +|z(\theta_\tau\omega)|^2) \d \tau}\|A^{\frac k2-\frac{5}{8}}\bar{v}(0)\|^2\nonumber\\
&+C\int_0^te^{C\int_\eta^t( \sum_{i=1}^2 \|A^\frac58v_i\|^2+|z(\theta_\tau\omega)|^2) \d \tau}
\|A^\frac58\bar{v}\|^2 \left(\sum_{i=1}^2\|A^{\frac k2-\frac{5}{8}}v_i\|^2+|z(\theta_\eta\omega)|^2\right) \d \eta\nonumber\\
& \leq Ce^{C\int_{0}^t( \sum_{i=1}^2 \|A^\frac58v_i\|^2+|z(\theta_\tau\omega)|^2) \d \tau}(\|A^{\frac k2-\frac{5}{8}}\bar{v}(0)\|^2\nonumber\\
&+\int_{0}^t\|A^\frac58\bar{v}\|^2\left( \sum_{i=1}^2\|A^{\frac k2-\frac{5}{8}}v_i\|^2+|z(\theta_\eta\omega)|^2 \right) \d \eta),\nonumber\\
\end{align}
moreover
\begin{align}\label{7.18}
&\int_{t-\frac{1}{2^{k-\frac{3}{2}}}}^t\|A^\frac k2\bar{v}(s)\|^2 \d s
\leq\int_{0}^t\|A^\frac k2\bar{v}(s)\|^2 \d s\nonumber\\
& \leq Ce^{C\int_{0}^t( \sum_{i=1}^2 \|A^\frac58v_i\|^2+|z(\theta_\tau\omega)|^2) \d \tau}\nonumber\\
&\left(\|A^{\frac k2-\frac{5}{8}}\bar{v}(0)\|^2+\int_{0}^t\|A^\frac58\bar{v}\|^2\left( \sum_{i=1}^2\|A^{\frac k2-\frac{5}{8}}v_i\|^2+|z(\theta_\eta\omega)|^2 \right) \d \eta\right)\nonumber\\
&\leq Ce^{C\int_{0}^t( \sum_{i=1}^2 \|A^\frac58v_i\|^2+|z(\theta_\tau\omega)|^2) \d \tau}\nonumber\\
&\left(\|A^{\frac k2-\frac{5}{8}}\bar{v}(0)\|^2+\|A^\frac58\bar{v}(0)\|^2\int_{0}^t\left( \sum_{i=1}^2\|A^{\frac k2-\frac{5}{8}}v_i\|^2+|z(\theta_\eta\omega)|^2 \right) \d \eta\right)\nonumber\\
&\leq Ce^{C\int_{0}^t( \sum_{i=1}^2 \|A^\frac58v_i\|^2+|z(\theta_\tau\omega)|^2) \d \tau}\|A^{\frac k2-\frac{5}{8}}\bar{v}(0)\|^2\nonumber\\
&\left(1+\int_{0}^t\left( \sum_{i=1}^2\|A^{\frac k2-\frac{5}{8}}v_i\|^2+|z(\theta_\eta\omega)|^2 \right) \d \eta\right).
\end{align}
Inserting \eqref{7.18} into \eqref{7.19} yields
\begin{align}
 \|A^\frac k2\bar{v}(t)\|^2
&\leq Ce^{C\int_0^t( \sum_{i=1}^2 \|A^\frac58v_i\|^2+|z(\theta_\tau\omega)|^2) \d \tau}\|A^{\frac k2-\frac{5}{8}}\bar{v}(0)\|^2\nonumber\\
&\cdot\left (1 +
\int_{0}^t   \left (\sum_{i=1}^2\|A^\frac k2v_i\|^2+|z(\theta_\eta\omega)|^2\right) \d \eta\right) .
\end{align}
Replacing $\omega$ by $\theta_{-t}\omega$,  it is easy to get
\begin{align}
 \|A^\frac k2\bar{v}(t,\theta_{-t}\omega,\bar{v}(0))\|^2
&\leq Ce^{C\int_0^t( \sum_{i=1}^2\|A^\frac58v_i(\tau,\theta_{-t}\omega,v_i(0))\|^2
+|z(\theta_{\tau-t}\omega)|^2)\d \tau}\|A^{\frac k2-\frac58}\bar{v}(0)\|^2\nonumber\\
&  \times \left (1 +
\int_{0}^t   \left (\sum_{i=1}^2\|A^\frac k2v_i(\eta,\theta_{-t}\omega,v_{i,0})\|^2+|z(\theta_{\eta-t}\omega)|^2\right) \d \eta\right).
\end{align}
Taking the inner product of the system \eqref{2.2} in $H$ by $A^{k-\frac{5}{4}}v$,  by the similar method and \eqref{7.20}, it yields
\begin{align}\label{7.21}
&\frac{\d}{\d t}\|A^{\frac{k}{2}-\frac{5}{8}}v\|^2 +  \nu  \|A^{\frac{k}{2}}v\|^2\nonumber \\
&\leq  C\left ( \|A^{\frac58}v\|^2+ |z(\theta_t\omega)|^2 \right)\|A^{\frac{k}{2}-\frac{5}{8}}v\|^2  +
C \left(1 +|z(\theta_t\omega)|^4 \right).
\end{align}
Applying Gronwall's lemma on \eqref{7.21} gives
\begin{align}\label{7.22}
&\|A^{\frac{k}{2}-\frac{5}{8}}v(t)\|^2
+ \nu \int_0 ^te^{ C\int_s^t(\|A^{\frac{5}{8}}v\|^2+|z(\theta_\tau \omega)|^2) \d \tau}\|A^{\frac{k}{2}} v(s)\|^2  \d s \nonumber\\
&\leq e^{ C\int_0^t(\|A^{\frac{5}{8}}v\|^2+|z(\theta_\tau\omega)|^2) \d \tau}\|A^{\frac{k}{2}-\frac{5}{8}}v( 0)\|^2 \nonumber\\ &+C\int_0^te^{C \int_s^t (\|A^{\frac{5}{8}}v\|^2+ |z (\theta_\tau\omega)|^2) \d \tau} \left (1 +|z(\theta_s\omega)|^4 \right)\d s\nonumber\\
&\leq Ce^{ C\int_0^t(\|A^{\frac{5}{8}}v\|^2+|z(\theta_\tau\omega)|^2) \d \tau} \left (\|A^{\frac{k}{2}-\frac{5}{8}}v( 0)\|^2
+\int_0^t\left (1 +|z(\theta_s\omega)|^4 \right)\d s\right).
\end{align}
Since
\begin{align*}
\int_0^t\left (1 +|z(\theta_s\omega)|^4 \right)\d s\leq e^{\int_0^t\left (1 +|z(\theta_s\omega)|^4 \right)\d s},
\end{align*}
the above estimate \eqref{7.22} is simplified to
\begin{align*}
&\|A^{\frac{k}{2}-\frac{5}{8}}v(t)\|^2
+ \nu \int_0 ^t\|A^{\frac{k}{2}} v(s)\|^2  \d s \nonumber\\
&\leq\|A^{\frac{k}{2}-\frac{5}{8}}v(t)\|^2
+ \nu \int_0 ^te^{ C\int_s^t(\|A^{\frac{5}{8}}v\|^2+|z(\theta_\tau \omega)|^2) \d \tau}\|A^{\frac{k}{2}} v(s)\|^2  \d s \nonumber\\
&\leq Ce^{ C\int_0^t\|A^{\frac{5}{8}}v\|^2\d \tau+C\int_0^t(1+|z(\theta_\tau\omega)|^4) \d \tau} \left (\|A^{\frac{k}{2}-\frac{5}{8}}v( 0)\|^2
+1\right).
\end{align*}
Replacing $\omega$ by $\theta_{-t}\omega$ yields
\begin{align}\label{7.24}
&\|A^{\frac{k}{2}-\frac{5}{8}}v(t,\theta_{-t}\omega,v(0))\|^2
+ \nu \int_0 ^t\|A^{\frac{k}{2}} v(s,\theta_{-t}\omega,v(0))\|^2  \d s \nonumber\\
&\leq Ce^{ C\int_0^t\|A^{\frac{5}{8}}v(\tau,\theta_{-t}\omega,\bar{v}(0))\|^2\d \tau+C\int_0^t(1+|z(\theta_{\tau-t}\omega)|^4) \d \tau} \left (\|A^{\frac{k}{2}-\frac{5}{8}}v( 0)\|^2
+1\right)\nonumber\\
&=Ce^{ C\int_0^t\|A^{\frac{5}{8}}v(\tau,\theta_{-t}\omega,\bar{v}(0))\|^2\d \tau+C\int_{-t}^0(1+|z(\theta_{\tau}\omega)|^4) \d \tau} \left (\|A^{\frac{k}{2}-\frac{5}{8}}v( 0)\|^2
+1\right).
\end{align}
By the above Lemma \ref{lemma4.1}, we have for $ t\geq T_1(\omega)$
\begin{align}\label{7.23}
\int_0^t\|A^{\frac{5}{8}}v(\tau,\theta_{-t}\omega,v(0))\|^2\d \tau \leq
C(e^{-\lambda t}\|v(0)\|^2
 + \zeta_1(\omega) ).
\end{align}
Inserting \eqref{7.23} into \eqref{7.24}, we have for $ t\geq T_1(\omega)$
\begin{align}
&\|A^{\frac{k}{2}-\frac{5}{8}}v(t,\theta_{-t}\omega,v(0))\|^2
+ \nu \int_0 ^t\|A^{\frac{k}{2}} v(s,\theta_{-t}\omega,v(0))\|^2  \d s \nonumber\\
&\leq Ce^{ C\zeta_1(\omega)+C\int_{-t}^0(1+|z(\theta_{\tau}\omega)|^4) \d \tau} \left (\|A^{\frac{k}{2}-\frac{5}{8}}v( 0)\|^2
+1\right):=\zeta_1^*(\omega),
\end{align}
here $\zeta_1^*(\omega)$ is a tempered random variable such that $\zeta_1^*(\omega)>1$.
Moreover, we can deduce that
\begin{align}
 \|A^\frac k2\bar{v}(t,\theta_{-t}\omega,\bar{v}(0))\|^2
&\leq Ce^{C\zeta_{1}(\omega)+C\zeta_1^*(\omega)}\|A^{\frac k2-\frac58}\bar{v}(0)\|^2 \left (1+\zeta_1^*(\omega) +
\int_{-t}^0   |z(\theta_{\eta}\omega)|^2 \d \eta\right)\nonumber\\
&\leq Ce^{C\zeta_{1}(\omega)+C\zeta_1^*(\omega)}\|A^{\frac k2-\frac58}\bar{v}(0)\|^2.
\end{align}
Let $t=T_k(\omega)>T^*=T_{k-1}(\omega)$,   we deduce at $t=T_k(\omega)$
\begin{align}
 \|A^\frac k2\bar{v}(T_k\omega ,\theta_{-T_k }\omega,\bar{v}(0))\|^2
&\leq Ce^{C\zeta_{1}(\omega)+C\zeta_1^*(\omega)} \|A^{\frac k2-\frac58}\bar{v}(0)\|^2. \nonumber
\end{align}
Let
\[
 L_k (\mathfrak B^{k-\frac{5}{4}}, \omega) :=  Ce^{C\zeta_{1}(\omega)+C\zeta_1^*(\omega)} ,\quad \omega\in \Omega,
\]
this completes the proof of Lemma \ref{lemma7.3}.
\end{proof}
In this section, we next introduce the following main result.

\begin{theorem}[Local $(H, H^k)$-Lipschtiz] \label{theorem7.4}
 Let Assumption \ref{assum} hold and $f\in H^{k-\frac{5}{4}}$ for $k\geq\frac{5}{2}$. For   any tempered set $\mathfrak D \in \D_H$,  then there
  exist  random variables $T^*_{{\mathfrak D}} (\cdot)  $   and $ L^*_{\mathfrak D}(\cdot )$ such that   two solutions $v_1$ and $v_2$ of random fractional three-dimensional NS equations \eqref{2.2} corresponding to initial values   $v_{1,0},$ $ v_{2,0}$ in $\mathfrak D \left (\theta_{-T^*_{\mathfrak D}(\omega)}\omega \right)$, respectively, satisfy
  \begin{align} \label{7.26}
  &
  \big\|A^\frac k2(v_1 \!  \big (T^*_{\mathfrak D}(\omega),\theta_{-T^*_{\mathfrak D}(\omega)}\omega,v_{1,0}\big)
  -v_2 \big (T^*_{\mathfrak D}(\omega),\theta_{-T^*_{\mathfrak D}(\omega)}\omega,v_{2,0} \big) \big) \|^2 \nonumber\\
& \leq  L^*_{\mathfrak D} (\omega)\|v_{1,0}-v_{2,0}\|^2,\quad \omega\in \Omega.
\end{align}
\end{theorem}
\begin{proof}
By the above Lemma \ref{theorem7.2}, for any tempered set   $\mathfrak D\in \D_H$, then there exist random variables $ T_{k-1}(\mathfrak D ,\cdot )  $ and $L_{k-1}(\mathfrak D,\cdot) $ such that
\be   \label{7.25}
  \left \| A^{\frac{k}{2}-\frac58}\bar v\left (T^* ,\theta_{-T^*}\omega,  \bar v(0)\right)
  \right \|^2
 \leq L_{k-1}({\mathfrak D}, \omega) \|\bar {v} (0)\|^2  ,
\ee
where $T^*= T_{k-1}(\mathfrak D, \omega)$. Moreover, $ v\left (T^* ,\theta_{-T^*}\omega,  v_{0}\right) \in \mathfrak B^{k-\frac{5}{4}}(\omega)$ for any $v_0\in \mathfrak D(\theta_{-T^*}\omega) $, $\omega\in \Omega$. Hence,
 by the above Lemma \ref{lemma7.3} and \eqref{7.25}, we can deduce that
\ben
  \left \| A^\frac k2\bar v \big( T_k +T^*, \theta_{-T_k -T^*}
 \omega, \bar v(0) \big)
   \right \|^2
 &  =  \left \| A^\frac k2\bar v \big( T_k  , \theta_{-T_k }
 \omega, \bar v  ( T^*, \theta_{-T_k-T^*}
 \omega, \bar v(0)  )\big)
   \right \|^2  \\
   &\leq  L_k(\mathfrak B^{k-\frac{5}{4}}, \omega) \left \| A^{\frac{k}{2}-\frac58} \bar v \big  ( T^*, \theta_{-T_k-T^*}
 \omega, \bar v(0)  \big)
   \right \|^2  \\
   &\leq L_k(\mathfrak B^{k-\frac{5}{4}},\omega)  L_{k-1} \big (\mathfrak D, \theta_{-T_k}\omega \big) \|\bar v(0)\|^2 ,
\ee
here $T^*=T^*(\theta_{-T_k}  \omega)$,
uniformly for $v_{1,0}$, $ v_{2,0}\in \B^{k-\frac{5}{4}}(\theta_{-T_k-T^*(\theta_{-T_k} \omega)}\omega)$, $  \omega\in \Omega$. Let the random variables defined by
\ben
  & T_{\mathfrak D}^*( \omega) := T_k +T^*(\theta_{-T_k} \omega) , \\
  &L_{\mathfrak D}^*(\omega) :=  L_k(\mathfrak B^{k-1},\omega)   L_{k-1} \big (\mathfrak D, \theta_{-T_k}\omega \big),
\ee
this completes the proof of Theorem \ref{theorem7.4}.
  \end{proof}

\subsection{The $(H,H^k)$-random attractor}

\begin{theorem}\label{theorem7.1}
 Let $k\geq\frac{5}{2}$. Let Assumption \ref{assum1} hold and $f\in H^{k-\frac{5}{4}}$. The  RDS $\phi$ generated by the random fractional three-dimensional NS equations  \eqref{2.2} has  a tempered $(H,H^k)$-random attractor $\mathcal A $. Moreover, $\mathcal A$ has a finite fractal dimension in $H^k$: there exists    $ d>0$ such that
\[
d_f^{H^k} \! \big ( \A(\omega) \big)  \leq d, \quad \omega \in \Omega.
\]
\end{theorem}
\begin{proof}
  Let $k\geq\frac{5}{2}$. In Lemma \ref{lemma7.2},  we have proved an $H^k$ random absorbing set $\mathfrak B_{H^k}$, which is tempered and closed in $H^k$.  Then the  $(H,H^k)$-smoothing property \eqref{7.26}   implies that the $(H,H^k)$-asymptotic compactness of $\phi$. Hence  by the above Lemma \ref{lem:cui18}, we show that $\A$ is an $(H,H^k)$-random attractor of $\phi$. By the above Theorem \ref{theorem4.6}, we prove  the fractal dimension in $H$ of  $\A$ is finite. By the Lemma 5 in \cite{cui} and the $(H,H^k)$-smoothing property \eqref{7.26}, we show the   finite-dimensionality  in $H^k$. This completes the proof of Theorem \ref{theorem7.1}.
\end{proof}

 \begin{remark}
 The local $(H,H^k)$-Lipschitz continuity \eqref{7.26}  and the finite fractal dimension in $H^k$ of the global attractor    are new even for deterministic fractional three-dimensional NS equations.
 \end{remark}

\section*{Acknowledgements}
  H. Liu was supported by the National Natural Science Foundation of China (No. 12271293), Natural Science Foundation of Shandong Province (No.  ZR2024MA069) and the project of Youth Innovation Team of Universities of Shandong Province (No. 2023KJ204). C.F. Sun was supported by the the National Natural Science Foundation of China (No. 11701269) and Natural Science Foundation of Jiangsu Province (No.  BK20231301). J. Xin was  supported  by the Natural Science Foundation of Shandong Province (No. ZR2023MA002).\\

\section*{Declarations}

\noindent{\bf Conflict of interest} \\

\noindent The authors declare that they have no conflict of interest.



\begin{thebibliography}{99}

\bibitem{ali}
 Z.I. Ali, {\em Stochastic generalized magnetohydrodynamics equations: well-posedness}. Appl. Anal., 98 (2019), 2464-2485.
\bibitem{Anh}
 C.T. Anh, L. Tinh,  {\em Regularity and attractors for the three-dimensional generalized Boussinesq system}, Math. Methods Appl. Sci., 46 (2023), 15526-15556.




 \bibitem{arnold}
 L. Arnold,   {\em Random Dynamical Systems}, Springer-Verlag, Berlin, 1998.

 \bibitem{Brzezniak13}
 Z. Brze\'{z}niak, T. Caraballo, J.A. Langa, Y. Li, G. {\L}ukaszewicz, J. Real,  {\em Random attractors for stochastic 2D-Navier-Stokes equations in some unbounded domains}, J. Differential Equations, 255 (2013), 3897-3919.


\bibitem{Brzezniak15}
Z. Brze\'{z}niak, S. Cerrai, M. Freidlin, Quasipotential and exit time for 2D stochastic Navier-Stokes equations driven by space time white noise, Probab. Theory Related Fields, 162 (2015), 739-793.

\bibitem{Brzezniak2015}
Z. Brze\'{z}niak, B. Goldys, Q.T. Le Gia,   {\em Random dynamical systems generated by stochastic Navier-Stokes equations on a rotating sphere}, J. Math. Anal. Appl., 426 (2015), 505-545.

 \bibitem{Brzezniak2018}
 Z. Brze\'{z}niak, B. Goldys, Q.T. Le Gia,  {\em Random attractors for the stochastic Navier-Stokes equations on the 2D unit sphere}, J. Math. Fluid Mech., 20 (2018), 227-253.



 \bibitem{Brzezniak}
 Z. Brze\'{z}niak, Y.H. Li,   {\em Asymptotic compactness and absorbing sets for 2D stochastic Navier-Stokes equations on some unbounded domains}, Trans. Amer. Math. Soc., 358 (2006), 5587-5629.











\bibitem{Caraballo} T. Caraballo, Z. Chen, D.D. Yang,  {\em Stochastic 3D globally modified Navier-Stokes equations: weak attractors, invariant measures and large deviations},  Appl. Math. Optim., 88 (2023).


 \bibitem{constantin}
 P. Constantin, C. Foias,   {\em Global Lyapunov exponents, Kaplan-Yorke formulas and the dimension of the attractors for 2D Navier-Stokes equations},  Comm. Pure Appl. Math., 38 (1985), 1-27.



 \bibitem{crauel97jdde}
 H. Crauel, A. Debussche, F. Flandoli,   {\em Random Attractors},  J. Dynam. Differential Equations,  9 (1997), 307-341.





 \bibitem{crauel}
 H. Crauel, F. Flandoli,   {\em Attractors for random dynamical systems},  Probab. Theory Related Fields, 100 (1994), 365-393.

\bibitem{cui21jde}
H.Y. Cui, A.N. Carvalho, A.C. Cunha, J.A. Langa,   {\em Smoothing and finite-dimensionality of uniform attractors in Banach spaces}, J. Differential Equations, 285 (2021) 383-428.


\bibitem{cui24siads}
H.Y.  Cui, R.N. Figueroa-Lopez, J.A. Langa, M.J.D. Nascimento, {\em Forward attraction of nonautonomous dynamical systems and applications to Navier-Stokes equations}, SIAM J. Appl. Dyn. Syst., 23 (2024), 2407-2443.


\bibitem{cui24ma}
H.Y. Cui,   R.N. Figueroa-Lopez,  H.L. López-Lázaro,  J. Simsen, {\em Multi-valued dynamical systems on time-dependent metric spaces with applications to Navier-Stokes equations}. Math. Ann., 390 (2024), 5415-5470.

\bibitem{cui18jdde}
H.Y. Cui, J.A. Langa, Y.R. Li,   {\em Measurability of random attractors for quasi strong-to-weak continuous random dynamical systems}, J. Dynam. Differential Equations, 30 (2018), 1873-1898.

 \bibitem{cui}
 H.Y. Cui, Y.R. Li,   {\em Asymptotic $H^2$ regularity of a stochastic reaction-diffusion equation},  Discrete Contin. Dyn. Syst. Ser. B, 27 (2022), 5653-5671.

\bibitem{cui2025}
H.Y. Cui, H. Liu, J. Xin,  {\em $(H,H^2)$-random attractor  of   Navier-Stokes equations with additive white noise on two-dimensional torus}, Submit.


\bibitem{debbi}
L. Debbi,   {\em Well-posedness of the multidimensional fractional stochastic Navier-Stokes equations on the torus and on bounded domains}, J. Math. Fluid Mech., 18 (2016),  25-69.








\bibitem{foias}
C. Foias, O. Manley, R. Rosa, R. Temam,   {\em Navier-Stokes Equations and Turbulence},vol. 83, Cambridge University Press, 2001.








 \bibitem{gal}
C.G. Gal, Y.Q. Guo,  {\em Inertial manifolds for the hyperviscous Navier-Stokes equations}, J. Differential Equations, 265 (2018), 4335-4374.


 \bibitem{gaoJDE}
H.J. Gao, H. Liu, {\em Well-posedness and invariant measures for a class of stochastic 3D Navier-Stokes equations with damping driven by jump noise}, J. Differential Equations, 267 (2019), 5938-5975.

 \bibitem{gess}
B. Gess, W. Liu, M. R\"{o}ckner,  {\em Random attractors for a class of stochastic partial differential equations driven by general additive noise}, J. Differential Equations, 251 (2011), 1225-1253.

 \bibitem{julia}J. García-Luengo, P. Marín-Rubio, J. Real, {\em$H^2$-boundedness of the pullback attractors for non-autonomous 2D Navier–Stokes equations in bounded domains}, Nonlinear Anal., 74 (2011), 4882-4887.

\bibitem{HanJDE}
Z.F. Han, S.F. Zhou, {\em Existence and continuity of random exponential attractors for stochastic 3D globally modified non-autonomous Navier-Stokes equation}, J. Differential Equations, 418 (2025), 1-55.

\bibitem{Huo2016}
W. Huo, A. Huang, {\em The global attractor of the 2D Boussinesq equations with fractional Laplacian in subcritical case}, Discrete Contin. Dyn. Syst. Ser. B, 21 (2016), 2531-2550.

 \bibitem{hou}
 Y.R. Hou, K.T. Li,   {\em The uniform attractor for the 2D non-autonomous Navier-Stokes flow in some unbounded domain}, Nonlinear Anal., 58 (2004), 609-630.

\bibitem{ju}
N. Ju,   {\em The $H^1$-compact global attractor for the solutions to the Navier-Stokes equations in two-dimensional unbounded domains}, Nonlinearity, 13 (2000), 1227-1238.


\bibitem{Kinra}
K. Kinra, M.T. Mohan, R.H. Wang, {\em Asymptotically Autonomous Robustness of Non-autonomous Random Attractors for Stochastic Convective Brinkman-Forchheimer Equations on $R^{3}$},  Int. Math. Res. Not. IMRN, 7 (2024), 5850-5893.


\bibitem{kostianko}
A. Kostianko, X.H. Li, C.Y. Sun and S. Zelik, {\em Inertial manifolds via spatial averaging revisited}, SIAM J. Math. Anal., 54 (2022), 268-305.




 \bibitem{Ladyzenskaja}
 O.A. Lady\u{z}enskaja,   {\em Solution "in the large" of the nonstationary boundary value problem for the Navier-Stokes system with two space variables}, Comm. Pure Appl. Math., 12 (1959), 427-433.

 \bibitem{langa}
 J.A. Langa, G. {\L}ukaszewicz, J. Real,   {\em Finite fractal dimension of pullback attractors for non-autonomous 2D Navier-Stokes equations in some unbounded domains}, Nonlinear Anal., 66 (2007), 735-749.

  \bibitem{langa2}
  J.A. Langa, J.C. Robinson,  {\em Fractal dimension of a random invariant set}, J. Math. Pures Appl., 85 (2006), 269-294.
   \bibitem{li2021}
  J.N. Li, H.X. Liu, H. Tang, {\em  Stochastic MHD equations with fractional kinematic dissipation and partial magnetic diffusion in $R^2$}. Stochastic Process. Appl., 135 (2021), 139-182.
 \bibitem{li2023}
S.H. Li, W. Liu, Y.C. Xie, {\em Stochastic 3D Leray-$\alpha$ model with fractional dissipation}, Sci. China Math., 66 (2023), 2589-2614.




\bibitem{li2020}
X.H. Li and C.Y. Sun, {\em Inertial manifolds for the 3D modified-Leray-image model}, J. Differential Equations, 268 (2020), 1532-1569.

\bibitem{liJDE}
X.J. Li,   {\em Pathwise property of 2D non-autonomous stochastic Navier-Stokes equations with less regular or irregular noise}, J. Differential Equations, 408 (2024), 166-200.


\bibitem{li}
X.J. Li,   {\em Uniform random attractors for 2D non-autonomous stochastic Navier-Stokes equations}, J. Differential Equations, 276 (2021), 1-42.









  \bibitem{liu2018}
  H. Liu, H.J. Gao,   {\em Ergodicity and dynamics for the stochastic 3D Navier-Stokes equations with damping}, Commun. Math. Sci., 16 (2018), 97-122.

  \bibitem{liu2023}
  H. Liu, H.J. Gao,   {\em Asymptotic regularity for the generalized MHD-Boussinesq equations}, Math. Methods Appl. Sci., 46 (2023),  11080-11098.

\bibitem{liusd}
H. Liu, L. Lin, Y.Y. Shi, {\em  Averaging principle for stochastic 3D generalized navier-stokes equations}, Stoch. Dyn., 24 (2024), 2450005.

\bibitem{liu}
H. Liu, L. Lin, C.F. Sun, {\em  Well-posedness and attractors of the multi-dimensional hyperviscous magnetohydrodynamic equations}, Appl. Anal., 102 (2023), 3971-3985.





\bibitem{liuvx}
V.X. Liu,   {\em Instability for the Navier-Stokes equations on the 2-dimensional torus and a lower bound for the Hausdorff dimension of their global attractors}, Comm. Math. Phys., 147 (1992), 217-230.

\bibitem{liuvx1}
V.X. Liu,   {\em A sharp lower bound for the Hausdorff dimension of the global attractors of the 2D Navier-Stokes equations}, Comm. Math. Phys., 158 (1993), 327-339.



 \bibitem{lu}
 S.S. Lu, H.Q. Wu, C.K. Zhong,   {\em Attractors for nonautonomous 2D Navier-Stokes equations with normal external forces}, Discrete Contin. Dyn. Syst., 13 (2005), 701-719.

 \bibitem{lukaszewicz}
 G. {\L}ukaszewicz,    {\em Pullback attractors and statistical solutions for 2-D Navier-Stokes equations}, Discrete Contin. Dyn. Syst. Ser. B, 9 (2008), 643-659.

 \bibitem{mohammed}
 S. Mohammed, T.S. Zhang,    {\em Dynamics of stochastic 2D Navier-Stokes equations}, J. Funct. Anal., 258 (2010), 3543-3591.



\bibitem{robinson01}
J.C. Robinson,   {\em Infinite-Dimensional Dynamical Systems: An Introduction to Dissipative Parabolic PDEs and the Theory of Global Attractors}, vol. 28, Cambridge University Press, 2001.

\bibitem{robinson11}
J.C. Robinson, {\em Dimensions, Embeddings, and Attractors}, Cambridge University Press, 2011.

\bibitem{Rockner}
M. R\"{o}ckner, T.S. Zhang, {\em Stochastic 3D tamed Navier-Stokes equations: Existence, uniqueness and small time large deviation principles}, J. Differential Equations, 252 (2012), 716-744.


\bibitem{rockner}
M. R\"{o}ckner, R.C. Zhu, X.C. Zhu, {\em Local existence and non-explosion of solutions for stochastic fractional partial differential equations driven by multiplicative noise}, Stochastic Process. Appl., 124 (2014), 1974-2002.

\bibitem{rosa}
R.  Rosa,   {\em The global attractor for the 2D Navier-Stokes flow on some unbounded domains}, Nonlinear Anal., 32 (1998), 71-85.




\bibitem{song}
X.L.  Song,  Y.R. Hou, {\em Attractors for the three-dimensional incompressible Navier-Stokes equations with damping}, Discrete Contin. Dyn. Syst., 31 (2011), 239-252.

\bibitem{song1}
X.L.  Song,  Y.R. Hou, {\em Uniform attractors for three-dimensional Navier-Stokes equations with nonlinear damping}, J. Math. Anal. Appl., 422 (2015), 337-351.

\bibitem{sun}
C.Y. Sun, {\em Asymptotic regularity for some dissipative equations}, J. Differential Equations, 248 (2010), 342-362.




\bibitem{temam1}
R. Temam,  {\em Infinite-dimensional dynamical systems in mechanics and physics}. Second edition. Springer-Verlag, New York, 1997.

\bibitem{temam}
R. Temam,  {\em Navier-Stokes Equations: Theory and Numerical Analysis}, vol. 343, American Mathematical Soc., 2001.

\bibitem{Tinh}
L. Tinh, {\em Asymptotic behavior of solutions for the three-dimensional generalized incompressible MHD equations with nonlinear damping terms},  Taiwanese J. Math., 29 (2025), 129-169.



\bibitem{wang}
B.X.  Wang,   {\em Uniform tail-ends estimates of the Navier-Stokes equations on unbounded channel-like domains}, Proc. Amer. Math. Soc., 151 (2023), 4841-4853.




\bibitem{wang2024}
R.H.  Wang, B.L. Guo, W. Liu, N.T. Da,    {\em Fractal dimension of random invariant sets and regular random attractors for stochastic hydrodynamical equations},  Math. Ann.,  389 (2024), 671-718.


\bibitem{wang2023}
R.H.  Wang, K. Kinra, M.T. Mohan,   {\em Asymptotically autonomous robustness in probability of random attractors for stochastic Navier-Stokes equations on unbounded Poincar\'{e} domains}, SIAM J. Math. Anal., 55 (2023), 2644-2676.

\bibitem{wangsiam}
R.N.  Wang, J.C. Zhao, A. Miranville,    {\em Hyperdissipative Navier-Stokes equations driven by time-dependent forces: invariant manifolds},   SIAM J. Appl. Dyn. Syst.,  22 (2023), 199-234.





\bibitem{wu2003}
J.H. Wu, Generalized MHD equations, J. Differential Equations, 195 (2003), 284-312.

\bibitem{wu2008}
J.H. Wu, Regularity criteria for the generalized MHD equations, Comm. Partial Differential Equations, 33 (2008), 285-306.







\bibitem{zhang}
L. Zhang, B. Liu, {\em Random perturbations for the chemotaxis-fluid model with fractional dissipation: global pathwise weak solutions}, Electron. J. Probab. 29 (2024), Paper No. 165, 56 pp.

\bibitem{zhao19}
C.D.  Zhao,   T. Caraballo,  {\em Asymptotic regularity of trajectory attractor and trajectory statistical solution for the 3D globally modified Navier-Stokes equations}, J. Differential Equations, 266 (2019), 7205-7229.

\bibitem{zhao}
C.D.  Zhao,  J.Q. Duan,  {\em Convergence of global attractors of a 2D non-Newtonian system to the global attractor of the 2D Navier-Stokes system}, Sci. China Math., 56 (2013), 253-265.





\bibitem{zhou}
S.F. Zhou and M. Zhao, {\em Fractal dimension of random invariant sets for nonautonomous random dynamical systems and random attractor for stochastic damped wave equation}, Nonl. Anal., 133 (2016), 292-318.

\bibitem{zhou2007}
Y. Zhou, {\em Regularity criteria for the generalized viscous MHD equations}. Ann. Inst. H. Poincar\'{e} Anal. Non Lin\'{e}aire, 24 (2007), 491-505. 


\bibitem{zhu}
R.C. Zhu, X.C. Zhu,  {\em Random attractor associated with the quasi-geostrophic equation}. J. Dynam. Differential Equations, 29 (2017), 289-322.
\end{thebibliography}
\end{document}